\documentclass{book}
\usepackage[mems,lang=american]{ems-book}
\usepackage{cite}
\usepackage{extarrows}
\theoremstyle{plain}
\newtheorem{theorem}{Theorem}[chapter]
\newtheorem{lemma}[theorem]{Lemma}
\newtheorem{prop}[theorem]{Proposition}
\newtheorem{corollary}[theorem]{Corollary}
\theoremstyle{definition}
\newtheorem{definition}[theorem]{Definition}
\newtheorem{remark}[theorem]{Remark}
\numberwithin{equation}{chapter}

\newcommand{\norm}[1]{\left\lVert#1\right\rVert}
\newcommand{\cM}{{\mathcal M}}
\newcommand{\cN}{{\mathcal N}}
\newcommand{\cC}{{\mathcal C}}
\newcommand{\cD}{{\mathcal D}}
\newcommand{\cL}{{\mathcal L}}
\newcommand{\cF}{{\mathcal F}}

\begin{document}
\frontmatter

\def\abstractname{Abstract}
\def\contentsname{Contents}
\def\chaptername{Chapter}
\def\refname{References}
\def\bibname{References}
\def\appendixname{Appendix}
\def\proofname{Proof}

\title{Extreme points of the unit balls and isometries of noncommutative quasi-Banach Marcinkiewicz spaces}
\author{Kai Fang, Yi Gao, Jinghao Huang and Fedor Sukochev}
\date{}
\maketitle

\begin{abstract}
We characterize extreme points of the (positive) unit balls of Marcinkiewicz spaces with respect to their natural quasi-norms. 
Having these results at hand, we show that every   positive surjective isometry on a noncommutative  
Marcinkiewicz space  is  of elementary form. 
In particular, 
 a   linear mapping  on a noncommutative  
 Marcinkiewicz space  over an atomic von Neumann algebra with all atoms having the same trace is a positive surjective isometry if and only if it is the restriction of  a  Jordan $*$-isomorphism
 which preserves the singular value functions. 
 
 We also study 
   (not necessarily positive) surjective isometries on weak $L_p$-spaces affiliated with a $\sigma$-finite non-atomic von Neumann algebra for $1<p<\infty$, Marcinkiewicz sequence spaces, and weak $\ell_p$ operator ideals, $p>0$,
   as well as one-parameter groups of isometries on the latter two.
\end{abstract}


\classification[46B85, 46L52, 46A16]{46B04}

\begin{funding}
K. Fang, Y. Gao and J. Huang were supported by the NNSF of China
(No.  12301160, 12471134 and 12671159);
F. Sukochev was supported by the ARC (DP230100434).
\end{funding}

\memsparagraph{Author addresses}
Kai Fang, Yi Gao and Jinghao Huang:
Institute for Advanced Study in Mathematics of HIT,
Harbin Institute of Technology, Harbin, 150001, China.
\par
\email{kaifang.8.25@gmail.com}; \email{yi.gao@stu.hit.edu.cn};
\email{jinghao.huang@hit.edu.cn}.
\par
Fedor Sukochev:
School of Mathematics and Statistics, University of NSW,
Sydney, 2052, Australia.
\par
\email{f.sukochev@unsw.edu.au}.

\tableofcontents
\mainmatter

\chapter{Introduction} \label{intro}
\section{Motivation}

The description of (surjective) isometries on (noncommutative) Banach function spaces has been widely studied.
Based on the work of Banach and Lamperti \cite{B32, Lamperti}, Lumer, Zaidenberg, Kalton, and others obtained extensive results on the characterization of isometries on various symmetric Banach function spaces (see, e.g., \cite{L63, Z95, KR, FJ}),
as well as  isometries on Banach sequence spaces  (see, e.g., \cite{A85, AC17}).
In the noncommutative setting, Kadison, Yeadon, Sourour,  the third and fourth co-authors and others obtained characterizations of isometries on operator algebras, noncommutative \(L_p\)-spaces, and noncommutative (Banach) symmetric spaces (see, e.g., \cite{K51, Y81, S81, FGHS, HS}).

However, many important spaces in classical and noncommutative analysis  cannot be equipped with a Banach norm, but only with a quasi-norm. 
In particular, for \(0<p\leq 1\), the weak \(L_p\)-function spaces \(L_{p,\infty}(0,\infty)\) and weak $\ell_p$-sequence spaces \(\ell_{p,\infty}\) cannot be equivalently renormed as Banach spaces \cite{Hunt, Ariono}.
In noncommutative setting,  the special place is occupied by the family of operator ideals $\mathcal{C}_{p,\infty}$ (see \cite{CCR, LMSZ23, LSZ, LSZ21}) given by the following formula:
\[\mathcal{C}_{p,\infty}:=\left\{x\in C_0(H):\norm{x}_{\mathcal{C}_{p,\infty}}:=\sup_{n\ge0}(n+1)^{1/p}\mu(n;x)<\infty\right\},\]
where 
$\mu(x)$ denotes the singular value sequence of $x\in C_0(H)$, the set of all compact operators on an infinite-dimensional  Hilbert space $H$ 
\cite{LMSZ23, LSZ, LSZ21}.
These spaces, as well as their commutative counterparts, admit an equivalent Banach norm if and only if $p>1$. 

Some characterizations of linear isometries on quasi-Banach symmetric spaces are also known.
Lamperti \cite{Lamperti} characterized linear isometries on classical \(L_p\)-spaces over \(\sigma\)-finite measure spaces for \(0<p<1\), while Sherman \cite{Sherman} considered surjective isometries between noncommutative \(L_p\)-spaces in the same range. 
More generally, Huang, Sukochev and Zanin \cite{HSZ} studied order-preserving isometries between symmetrically quasi-normed operator spaces under the strictly log-monotone quasi-norm assumption. 
However, weak \(L_p\)-spaces, as well as weak $\ell_p$-ideals $\mathcal{C}_{p,\infty}$,  equipped with quasi-norms do not satisfy the strictly log-monotone quasi-norm assumption \cite{SZ23}, so the techniques used in all results listed above are no longer applicable,
even in the commutative setting.
This forces us to develop new techniques.

Our main motivation in this paper is to describe extreme points of the (positive) unit balls of weak $L_p$-spaces and their surjective isometries under the natural quasi-norms for all $p\in (0,\infty)$.
Let us describe a general framework in which we present our results.

\section{General setting}

Recall that a non-decreasing function $\Phi:[0,\infty) \to[0,\infty)$ satisfies the $\Delta_2$-condition if there exists a constant $C_\Phi\in [1,\infty)$ such that for all $t\in[0,\infty)$,
\[\Phi(2t)\le C_\Phi\Phi(t)\]
(see \cite[Definition 2.1]{MNPT}).
A function $\Phi:[0,\infty) \to[0,\infty)$ is called {\itshape admissible} if it satisfies the following three conditions (see \cite[Definition 4.2]{MNPT}):
\begin{enumerate} [label=(\roman*)]
    \item $\Phi$ is non-decreasing and right-continuous.
    \item $\Phi$ satisfies the $\Delta_2$-condition.
    \item $\Phi(t)=0$ if and only if $t=0$.
\end{enumerate}

Let $\Phi:[0,\infty) \to[0,\infty)$ be admissible. 
Let $H$ be an infinite-dimensional Hilbert space and
$\mathcal{M}$ be a von Neumann algebra  on $H$ equipped with a  semifinite faithful normal trace $\tau$. 

If $\mathcal{M}=L_\infty\left(0,\infty \right)$, then the Marcinkiewicz function space $M_\Phi\left(0,\infty\right)$ is defined by
\[M_\Phi\left(0,\infty\right):=\left\{ x\in S\left(0,\infty\right): \sup_{t>0}\Phi\left(t\right)\mu\left(t;x\right)<\infty\right\} \]
equipped with the quasi-norm (see \cite[Proposition 4.3]{MNPT})
$$\norm{x}_{M_\Phi\left(0,\infty\right)}=\sup_{t>0}\Phi\left(t\right)\mu\left(t;x\right)<\infty.$$
The quasi-Banach Marcinkiewicz space $M_\Phi\left(0,\infty\right)$ is isomorphic to some Banach space if and only if 
\[\lim_{\lambda\to \infty }\inf_{t>0}\frac{\lambda\Phi(t)}{\Phi(\lambda t)}=\infty\] 
when $\Phi:[0,\infty)\to [0,\infty)$ is admissible, satisfying $\frac1 \Phi\in S\left(0,\infty\right)$  \cite[Theorem 7.6]{KS}.
The noncommutative Marcinkiewicz space 
\[M_\Phi\left(\mathcal{M},\tau\right)=\left\{x\in S\left(\mathcal{M},\tau\right): \sup_{t>0}\Phi\left(t\right)\mu\left(t;x\right)<\infty \right\}\]
equipped with 
 the quasi-norm (see 
 \cite{S14}) 
\[\norm{x}_{M_\Phi\left(\mathcal{M},\tau\right)}= \sup_{t>0}\Phi\left(t\right)\mu\left(t;x\right), 
~ x \in M_\Phi\left(\mathcal{M},\tau\right), \]
is a quasi-Banach  space, where $\mu(x)$ stands for the singular value function of  $x\in  S(\cM,\tau)$,  the $*$-algebra   of all $\tau$-measurable operators \cite{DPS,LMSZ23, LSZ, LSZ21}. 

Moreover, if $\mathcal{M}=\ell_\infty $  is equipped with the counting measure, then the Marcinkiewicz sequence space  $m_\Phi$ is defined by
\[m_\Phi:=\left\{x\in \ell_\infty: \norm{x}_{m_\Phi}:=\sup_{n\ge 0}l_{\Phi} \left(n\right) 
\mu\left(n;x\right) <\infty \right\},\]
where  $\mu\left(\cdot;x\right)$    stands for the decreasing rearrangement of $|x|$ and    $l_{\Phi}(n):=\Phi\left((n+1)_-\right) = \lim_{t\uparrow n+1} \Phi(t),~\forall n\in\mathbb{N}$.

Let $0<p<\infty$.
 Let $\Phi\left(t\right)= t^{1/p},~t\ge0$, and, in this case, $\Phi$ is admissible. 
The space $M_\Phi \left(\cM,\tau\right) $ becomes the weak 
$L_p$-space 
\[
L_{p,\infty} \left(\cM,\tau\right):=\left\{x\in S\left(\mathcal{M},\tau\right):
\norm{x}_{L_{p,\infty}\left(\mathcal{M},\tau\right)}:=\sup\limits_{t>0}t^{1/p}\mu\left(t;x\right)<\infty\right\},
\]
which is equipped with a natural quasi-norm $\norm{\cdot}_{L_{p,\infty}(\cM,\tau)}$.
In the special setting when $\cM=\ell_\infty $, 
$L_{p,\infty }\left(\cM,\tau\right)$ is the classical  
 weak $\ell_p$-sequence space $\ell_{p,\infty}$. 
If $\mathcal{M}=L_\infty\left(0,\infty\right)$ with  a semifinite faithful normal trace $\tau$ given by 
\[\tau\left(x\right)=\int_0^\infty xdm, ~0\le x\in L_\infty(0,\infty),\] 
then $L_{p,\infty}\left(\cM,\tau\right)$ is the 
 weak $L_p$-function space $L_{p,\infty}\left(0,\infty\right)$ on $\left(0,\infty\right)$. 
If $\mathcal{M}=B\left(H\right)$ equipped with the standard trace $\tau$, then $L_{p,\infty}(\cM,\tau)$ is the weak $L_p$-ideal $\cC_{p,\infty }$, whose importance  has been emphasized in monographs by  Simon \cite{Simon}, by Birman and Solomyak \cite{BS03}, and  \cite{LSZ,LSZ21,LMSZ23}. 
 The results in this paper apply to the setting of finite-dimensional operator ideals  \cite{V37}. 
 Although for $1<p<\infty$, the weak 
$L_p$-space $L_{p,\infty} \left(\cM,\tau\right)$ can be equivalently renormed as Banach spaces \cite{DPS, BS}, this is impossible when $0<p\le1$ and $\cM$ is infinite-dimensional. 
If $\cM$ is non-atomic or atomic with all atoms having the same trace,  
then surjective isometries on $L_{p,\infty} \left(\cM,\tau\right), ~1<p<\infty$, have been fully described  in \cite{FGHS,AC17,AC,KR} 
with respect to its equivalent norm.

Throughout this paper, we only consider linear isometries, and therefore, we always omit the term ``linear''.

 We only consider surjective isometries  $T$ on weak $L_p$-spaces $L_{p,\infty}\left(\mathcal{M},\tau\right)$ because 
 the disjointness-preserving property\footnote{For two  $\tau$-measurable  operators $x,y$, we write $x\perp y$ if $l(x)l(y)=0=r(x)r(y)$ ( 
 if, in addition, $x,y$ are self-adjoint, then $x\perp y$ if and only if $xy=0$).
 The mapping $T$ is called disjointness-preserving if $x,y\in M_\Phi\left(\mathcal{M},\tau\right)$ with $x\perp y$ implies $T(x)\perp T(y)$.} may fail when we consider non-surjective isometries.

\begin{remark}
 Assume that $0<p<\infty$.
Let $\{e_1,e_2\}$ be the standard basis of the $2$-dimensional weak $\ell_p$-sequence space.
Let $\{f_1,f_2,f_3\}$ be the standard basis of the $3$-dimensional weak $\ell_p$-sequence space.
Define a mapping $T:\ell_{p,\infty} \left(e_1,e_2\right)\to \ell_{p,\infty}\left(f_1,f_2,f_3\right)$ by setting $T\left(e_1\right) =f_1+\frac{3^{-1/p}}{1+2^{-1/p}}f_3$ and  $T\left(e_2\right) = f_2+\frac{3^{-1/p}}{1+2^{-1/p}}f_3 $. 
This is a non-surjective isometry  which does not preserve the disjointness of elements in $\ell_{p,\infty}(e_1,e_2)$. 
\end{remark}


\section{Extreme points of  the (positive) unit ball of a  noncommutative quasi-Banach Marcinkiewicz space}
 We call $x\in W$ an {\itshape extreme point} of a subset $W$ of a quasi-Banach space $X$ if $x=\frac{y+z}{2}$, $y,z\in W$, implies that $y=z=x$.


In Chapter \ref{unit ball},
we provide a complete description of extreme points of the unit ball of a Marcinkiewicz space $M_\Phi\left(\mathcal{M},\tau\right)$ with a quasi-norm.

\begin{theorem} \label{atomless extreme points}
    Let $\mathcal{M}$ be a semifinite non-atomic von Neumann algebra equipped with a semifinite faithful normal trace $\tau$. 
   Let $\Phi:[0,\infty) \to[0,\infty)$ be admissible
    satisfying $\frac1\Phi\in \left(L_1+L_\infty\right)\left(0,\infty\right)$ and $\Phi\left(\infty\right):=\displaystyle\lim_{t \to \infty} \Phi\left(t\right)=\infty$. 
    Assume that $x\in M_\Phi\left(\mathcal{M},\tau\right)$ with quasi-norm $1$. 
    Then,  $x$ is an extreme point of the unit ball of $M_\Phi\left(\mathcal{M},\tau\right)$ if and only if 
    \[\mu\left(t;x\right) =\frac{1}{\Phi\left(t\right)},~\forall t\in \left(0,\tau\left(\mathbf{1}\right)\right).\]
    
\end{theorem}


The above result is applicable to $L_{p,\infty}\left(\mathcal{M},\tau\right), ~1<p<\infty$, when $\mathcal{M}$ is non-atomic because $t\mapsto t^{-1/p}\in \left(L_1+L_\infty\right)\left(0,\infty\right)$ for $1<p<\infty$ (see Corollary \ref{1<p<infty}). 
For $0<p\le1$, the function $t\mapsto t^{-1/p}$ does not belong to $(L_1+L_\infty)\left(0,\infty\right)$ and the results of Theorem \ref{atomless extreme points} are not applicable in this case. In particular, as it turned out,  the  unit ball of $L_{p,\infty}\left(\mathcal{M},\tau\right)$ has no extreme point 
and the positive unit ball of $L_{p,\infty}\left(\mathcal{M},\tau\right)$ has a unique extreme point $0$
when $\mathcal{M}$ is non-atomic (see Corollary \ref{0<p le1} and Remark~\ref{positive no extr}).
 
The following theorem provides a complete description of extreme points of the unit ball of  a  quasi-Banach Marcinkiewicz space $M_\Phi\left(\mathcal{M},\tau\right)$ when $\mathcal{M}$ is atomic with all atoms having trace equal to $1$.
In this case, we do not need the condition that $\frac1\Phi\in (L_1+L_\infty) (0,\infty)$.

\begin{theorem} \label{atomic}
    Let $\mathcal{M}$ be a semifinite atomic von Neumann algebra equipped with a semifinite faithful normal trace $\tau$ and all atoms have trace equal to one. 
   Let $\Phi:[0,\infty) \to[0,\infty)$ be admissible satisfying $\Phi\left(\infty\right):=\displaystyle\lim_{t \to \infty} \Phi\left(t\right)=\infty$. 
    Assume that $x\in M_\Phi\left(\mathcal{M},\tau\right)$ with quasi-norm $1$.
    Then,  $x$ is an extreme point of the unit ball of $M_\Phi\left(\mathcal{M},\tau\right)$ if and only if 
 \[\mu\left(n;x\right) =\frac{1}{l_\Phi\left(n\right)},~ \forall n\in\mathbb{N}, ~ n<\tau\left(\mathbf{1}\right).\]
\end{theorem}

Letting $\Phi(t)=t^{1/p}, ~t\ge0$, and in this case, $l_\Phi(n)=\lim_{t\uparrow n+1}\Phi(t)=\Phi(n+1)=(n+1)^{1/p}, ~n\ge0$, we get a characterization of extreme points of the unit ball of $\mathcal{C}_{p,\infty}$ for $0<p<\infty$ (see Corollary~\ref{atom extremes}).


Throughout this paper, 
we denote the positive unit ball (i.e., the set of all positive elements in the unit ball) of a Marcinkiewicz space $M_{\Phi}\left(\mathcal{M},\tau\right)$ by $B^+_{M_{\Phi}\left(\mathcal{M},\tau\right)}$.
Denote the set of all extreme points of $B^+_{M_{\Phi}\left(\mathcal{M},\tau\right)}$ (respectively, $B_{M_\Phi\left(\cM,\tau\right)}$, i.e. the set of all elements in the unit ball)
by $\operatorname{extr}\left(B^+_{M_{\Phi}\left(\mathcal{M},\tau\right)}\right)$ (respectively, $\operatorname{extr}\left(B_{M_\Phi\left(\cM,\tau\right)}\right)$).

In Chapter \ref{positive unit ball},
we  provide a characterization of extreme points of the positive unit ball of a quasi-Banach Marcinkiewicz space as follows.

\begin{theorem} \label{posi non-atomic}
    Let $\mathcal{M}$ be a semifinite non-atomic von Neumann algebra equipped with a semifinite faithful normal trace $\tau$
    and $\Phi:[0,\infty) \to[0,\infty)$ be admissible satisfying
    $\frac1\Phi\in \left(L_1+L_\infty\right)\left(0,\infty\right)$ and $\Phi\left(\infty \right):=\displaystyle\lim_{t \to \infty} \Phi\left(t\right)=\infty$. 
    Then, \[\operatorname{extr}\left(B^+_{M_{\Phi}\left(\mathcal{M},\tau\right)}\right)=
    \left\{x\in B^+_{M_{\Phi}\left(\mathcal{M},\tau\right)}:\mu\left(x\right)=\frac{1}{\Phi}\chi_{_{\left(0,\tau\left(r(x)\right)\right)}}\right\}.\]
\end{theorem}

An application of the above result yields a characterization of extreme points of the positive unit ball of a weak $L_p$-space $L_{p,\infty}\left(\mathcal{M},\tau\right), ~p>1$, when $\mathcal{M}$ is non-atomic (see Corollary \ref{posi non-atomic1}), i.e.,
\[\operatorname{extr}\left(B^+_{L_{p,\infty}\left(\mathcal{M},\tau\right)}\right)=
    \left\{x\in B^+_{L_{p,\infty}\left(\mathcal{M},\tau\right)}:\mu\left(t;x\right)=t^{-1/p}\chi_{_{\left(0,\tau\left(r(x)\right)\right)}}(t), ~ 0<t<\tau\left(\mathbf{1}\right)\right\}.\]

It is noteworthy that the positive unit ball of a quasi-Banach Marcinkiewicz space $M_\Phi\left(\mathcal{M},\tau\right)$ does not necessarily have nonzero extreme points when
$\mathcal{M}$ is non-atomic and $\frac1\Phi\notin \left(L_1+L_\infty\right)\left(0,\infty\right)$ (see Remark \ref{positive no extr}).

If $\cM$ is atomic with all atoms having trace equal to $1$, then the above result may be sharpened.

\begin{theorem}  \label{extr point atomic}
    Let $\mathcal{M}$ be a semifinite atomic von Neumann algebra with a semifinite faithful normal trace $\tau$ 
     and $\Phi:[0,\infty) \to[0,\infty)$ be admissible satisfying $\Phi\left(\infty \right):=\displaystyle\lim_{t \to \infty} \Phi\left(t\right)=\infty$. 
    If  the trace of each  atom  in $\cM$  is $1$, 
    then  \[\operatorname{extr}\left(B^+_{M_{\Phi}\left(\mathcal{M},\tau\right)}\right)=
    \left\{x\in B^+_{M_{\Phi}\left(\mathcal{M},\tau\right)}:\mu\left(n;x\right)=\frac1{l_\Phi(n)},~\forall n\in \mathbb{N},n<\tau\left(r(x)\right)\right\}.\] 
\end{theorem}

In particular, in the case when $\Phi(t)=t^{1/p}, ~t\ge0$, the above result implies that
\[\operatorname{extr}\left(B^+_{\mathcal{C}_{p,\infty}}\right)=
    \left\{x\in B^+_{\mathcal{C}_{p,\infty}}:\mu\left(n;x\right)=(n+1)^{-1/p},~\forall n\in \mathbb{N},n<\tau\left(r(x)\right)\right\}.\]

\section{Positive  surjective isometries on a  noncommutative quasi-Banach Marcinkiewicz space}






Analogously to the case of weak $L_p$-spaces discussed above,
the main result in \cite{HSZ} does not apply to isometries on {quasi-Banach} Marcinkiewicz spaces in the present paper\footnote{ 
The Marcinkiewicz quasi-norms in the present paper are not necessarily monotone  
 with respect to the logarithmic submajorization\cite{HSXZ,SZ23}.
On the other hand,
 all 
  quasi-Banach symmetric  spaces
 have equivalent    symmetric  quasi-norms which are monotone with respect to the logarithmic submajorization\cite{K98}. }.

In Chapter \ref{positive isometries}, we study positive isometries on quasi-Banach Marcinkiewicz spaces $M_\Phi\left(\mathcal{M},\tau\right)$ when $\Phi$ is admissible and $\mathcal{M}$ is non-atomic or atomic with all atoms having trace equal to $1$.
An important ingredient of the proofs for the following two theorems 
(the  main results of Chapter \ref{positive isometries}) is a (quasi-Banach) noncommutative version of Abramovich's theorem \cite{Abr88} (see also \cite{Veksler1984} due to Veksler), which shows that any positive surjective isometry on Banach lattices is an order isomorphism.
This extends results in  \cite[Lemma 3.2]{CMS} and  \cite[Lemma 5.1]{HSZ}, and is of interest in its own right. 

{Having this at hand, we prove that all  positive surjective isometries on $M_\Phi\left(\mathcal{M},\tau\right)$ are disjointness-preserving and thereby obtain the following description.}



\begin{theorem} \label{Marcin6}
   Let $\Phi:[0,\infty) \to[0,\infty)$ be admissible  and  strictly increasing satisfying $\frac1\Phi\in \left(L_1+L_\infty\right)\left(0,\infty\right)$  and $\Phi\left(\infty\right):=\displaystyle\lim_{t \to \infty} \Phi\left(t\right)=\infty$.
     Let $\mathcal{M}$ be a semifinite non-atomic von Neumann algebra equipped with a semifinite faithful normal trace $\tau$ and $T$ be a positive surjective isometry on a quasi-Banach Marcinkiewicz space $M_{\Phi}\left(\mathcal{M},\tau\right)$. 
     Then, there  exist  a positive scalar $b$ and  a  (normal) Jordan $*$-isomorphism $J:S(\cM,\tau)\to S(\cM,\tau)$
     such that
    \[T\left(x\right)=bJ\left(x\right),~\forall x \in M_{\Phi}\left(\mathcal{M},\tau\right),\]
    where $J$ is such that 
   $\tau\left(J\left(p\right)\right)= \tau\left(J\left(q\right)\right)$ for all $\tau$-finite projections $p,q$   satisfying $\tau\left(p\right)=\tau\left(q\right)$.
\end{theorem}

{
In Yeadon's description of positive surjective isometries on noncommutative $L_p$-spaces, $b$ is affiliated with $Z(\cM)$ but need not be a scalar \cite[Theorem 2]{Y81}.
For (noncommutative) symmetric spaces which do not coincide with $L_p$-spaces, the description of positive surjective isometries, established in  \cite[Theorem 7.2]{KR}, \cite{zaidenberg, z77, R97} and \cite[Theorem~7.0.5]{FGHS}, showed that $b$ is necessarily a scalar.
Surprisingly, we obtain the same conclusion for the quasi-Banach Marcinkiewicz spaces, as established in Theorem~\ref{Marcin6}.
}

 In the special  setting when $\cM =B(H)$, 
  characterizations of positive surjective isometries  on 
  Marcinkiewicz operator ideals are given (the  weak $\ell_p$ operator ideals $C_{p,\infty}$, $0<p<\infty$, are covered by the result below).

\begin{theorem} \label{atomic J}
    Let $\Phi:[0,\infty) \to[0,\infty)$ be 
    admissible
    satisfying $\Phi\left(\infty\right):=\displaystyle\lim_{t \to \infty} \Phi\left(t\right)=\infty$.
     Let $\mathcal{M}$ be a semifinite atomic von Neumann algebra with all atoms having trace equal to one
         and $T$ be a positive surjective linear operator on a quasi-Banach Marcinkiewicz space $M_{\Phi}\left(\mathcal{M},\tau\right)$.
     Then, 
     $T$ is an isometry if and only if there exists a (normal) Jordan $*$-isomorphism $J:\cM\to \cM$ 
     such that
    \[T\left(x\right)=J\left(x\right),~\forall x \in M_{\Phi}\left(\mathcal{M},\tau\right),\]
    where $J$ preserves the singular value functions.
\end{theorem}

\section{Isometries on  noncommutative weak \texorpdfstring{$L_p$}{Lp}-spaces, \texorpdfstring{$1<p<\infty $}{1 < p < infinity} and \texorpdfstring{$\mathcal{C}_{p,\infty}, ~0<p<\infty$}{C(p,infinity), 0 < p < infinity}}


 
In Chapter \ref{s6}, we 
study not necessarily positive isometries on a wide class of noncommutative 
quasi-Banach symmetric spaces, whose quasi-norms have   equivalent symmetric norms. 
In particular, we obtain the following descriptions of isometries on  noncommutative weak $L_p$-spaces when $1<p<\infty $. 
The main ingredients of this chapter  are   the Banach envelope norm introduced by Peetre in \cite{Peetre}, and descriptions of isometries on (not necessarily separable) noncommutative symmetric spaces obtained in \cite{FGHS,AC17,AC,KR}. 

We emphasize that the results  in Theorem \ref{sigma-finite} are logically independent of those presented in the preceding sections.  


\begin{theorem} \label{sigma-finite}
    Assume that $1<p<\infty$. 
    Let $\mathcal{M}$ be a $\sigma$-finite non-atomic von Neumann algebra equipped with a semifinite faithful normal trace $\tau$.
    Let $T$ be a surjective linear  mapping on a weak $L_p$-space $L_{p,\infty}\left(\mathcal{M},\tau\right)$.
    Then, $T$ is an isometry if and only if there exist a (normal) Jordan $*$-isomorphism $J:S\left(\mathcal{M},\tau\right)\to S\left(\mathcal{M},\tau\right)$, a unitary operator $u$ in $\mathcal{M}$ and a positive scalar $a$ such that
    \[T(x)=auJ(x), ~\forall x\in L_{p,\infty}\left(\mathcal{M},\tau\right)\]
    and 
    \[\tau(J(x))=a^{-p}\tau(x), ~\forall x\in  S\left(\mathcal{M},\tau\right)_+ .\]

    If, in addition, $\tau(\mathbf{1})<\infty$, then
    $T$ is an isometry if and only if  there exist
    a trace-preserving (normal) Jordan $*$-isomorphism $J:S\left(\mathcal{M},\tau\right)\to S\left(\mathcal{M},\tau\right)$ and a unitary operator $u\in\mathcal{M}$ such that
    \[T(x)=uJ(x), ~\forall x\in L_{p,\infty}\left(\mathcal{M},\tau\right).\]
\end{theorem}

\begin{remark}
    We do not know the description of surjective isometries on $L_{p,\infty}\left(\mathcal{M},\tau\right), ~0<p\le1$, when $\mathcal{M}$ is   non-atomic. 
\end{remark}
 








In Chapter \ref{s5},  we consider isometries on $\cC_{p,\infty}$ when $0< p \le 1 $. The main idea of Chapter \ref{s5} is to extend an isometry on $\cC_{p,\infty}$ to the Hilbert space $\cC_2$.
This trick, together with the results obtained in Chapter \ref{s6}, yields the following. 
 
\begin{theorem} \label{111} 
    Assume that $0<p<\infty$. 
   Let $T$ be a   surjective linear mapping  on a weak $\ell_p$-sequence space $\ell_{p,\infty}$. 
Then, $T$ is an isometry if and only if there exists a permutation $\sigma$ of the set $\{0,1,2,\dots\}$ and  a 
unimodular
sequence $\{u_k\}_{k\ge0}$ such that 
    \[T\left(x\right)\left(n\right)=u_n x\left(\sigma\left(n\right)\right),~\forall x\in \ell_{p,\infty}.\]
\end{theorem}

\begin{theorem} \label{final}
    Assume that $0<p<\infty$. 
    Let $H$ be a complex separable infinite-dimensional Hilbert space. 
    Let $T$ is a  surjective linear mapping from a weak $\ell_p$-ideal $\mathcal{C}_{p,\infty}$ onto $\mathcal{C}_{p,\infty}$. 
    Then, $T$ is an isometry if and only if there exist unitary operators $U,V$ such that one of the following two statements holds:
    \begin{enumerate} [label=(\roman*)]
    \item $T(x)=UxV^*, ~\forall x\in \mathcal{C}_{p,\infty}$;
    \item $T(x)=Ux'V^*, ~\forall x\in \mathcal{C}_{p,\infty}$, where $x'$ denotes the transpose of $x$. 
    \end{enumerate}
\end{theorem}
In Chapter \ref{s5}, we also study isometries on  more general quasi-Banach Marcinkiewicz sequence spaces/operator ideals.  

Using the characterization of surjective isometries on $\cC_{p,\infty}$ and $\ell_{p,\infty}$ for $0<p<\infty$ (see Theorem \ref{111} and Theorem \ref{final}), we provide  a complete description of strongly continuous one-parameter groups of isometries of $\cC_{p,\infty}$ and $\ell_{p,\infty}$ in Chapter \ref{s8}.

\chapter{Preliminaries}
In this chapter, we recall the general concepts of $\tau$-measurable operators, the singular value functions, Jordan $*$-isomorphisms and develop some of their basic properties. 
In what follows, $H$ is a complex Hilbert space, $\mathcal{M}$ is a von Neumann algebra on $H$ and $B\left(H\right)$ denotes the $*$-algebra of all bounded linear operators on $H$. The lattice of all orthogonal projections in $\mathcal{M}$ (respectively, $B\left(H\right)$) is denoted by $P\left(\mathcal{M}\right)$ (respectively, $P\left(B\left(H\right)\right)$). 
The commutant of $\mathcal{M}$ is denoted by $\mathcal{M}'$.
For the convenience of the reader, some of the basic definitions from \cite{DPS} are recalled.

\section{\texorpdfstring{\(\tau\)}{tau}-measurable operators and the singular value functions} 
Throughout this paper, $H$ is an infinite-dimensional Hilbert space and $\mathcal{M}$ is assumed to be a semifinite von Neumann algebra on $H$ with a given semifinite faithful normal trace $\tau$.
A linear operator $x:\cD(x)\to H$ is said to be affiliated with 
$\mathcal{M}$, denoted by $x\eta \mathcal{M}$, 
if $yx \subseteq xy$ for all $y\in \cM'$.
A linear operator $x:\cD(x)\to H$ is termed measurable with respect to the von Neumann algebra $\mathcal{M}$, if $x$ is closed, affiliated with $\mathcal{M}$, and there exists a sequence $\{p_n\}_{n=1}^\infty$ in $P\left(\mathcal{M}\right)$ such that $p_n \uparrow \mathbf{1},~p_n\left(H\right) \subseteq \cD(x) $ and $p_n ^\perp$ is a finite
projection for every $n=1,2,\dots$.
The collection of all operators which are measurable with respect to $\mathcal{M}$ is denoted by $S\left(\mathcal{M}\right)$. 
The set $S\left(\mathcal{M}\right)$ is a complex $*$-algebra with unit element $\mathbf{1}$, with respect to the operations of strong sum, strong product, denoted simply by $x+y$ and $xy$ for any $x,y\in S\left(\mathcal{M}\right)$, respectively, and the $*$-operation of taking adjoints. 
The von Neumann algebra $\mathcal{M}$ is a $*$-subalgebra of $S\left(\mathcal{M}\right)$.

A projection $p\in P\left(\mathcal{M}\right)$ is called $\tau$-finite if $\tau\left(p\right)<\infty$. 
An operator $x\in S(\cM)$ is called $\tau$-measurable if   there exists a sequence $\{p_n\}_{n=1}^\infty\subset P\left(\mathcal{M}\right)$ such that $p_n\uparrow \mathbf{1},~p_n\left(H\right)\subseteq \cD(x)$ and $\tau\left(p_n^\perp\right)<\infty$ for all $n$. 
The set of all $\tau$-measurable operators is denoted by $S\left(\mathcal{M},\tau\right)$, which is a $*$-subalgebra of $S\left(\mathcal{M}\right)$.
The set of all self-adjoint elements in $\cM$ is denoted by $\cM_h$.
The set of all positive operators in $S(\cM,\tau)$ (respectively, in $\cM$, etc.) is denoted by $S(\cM,\tau)_+$ (respectively, $\cM_+$, etc.).
If $x\in S\left(\mathcal{M},\tau\right)$, then the spectral distribution function $d\left(|x|\right):[0,\infty)\to [0,\infty]$ of $|x|$ is defined by setting
\[d\left(s;|x|\right)=\tau\left(e^{|x|}\left(s,\infty\right)\right),~s\ge0.\]
Note that $d\left(|x|\right)$ is decreasing and right-continuous.
\begin{definition}
    For $x\in S\left(\mathcal{M},\tau\right)$, the singular value function $\mu\left(x\right):[0,\infty)\to [0,\infty]$ is defined to be the right-continuous inverse of the spectral distribution function $d\left(|x|\right)$, that is,
    \[\mu\left(t;x\right)=\inf\{s\ge0:d\left(s;|x|\right)\le t\},~t\ge0.\]
\end{definition}
From the properties of the right-continuous inverse, it follows that 
\[d\left(s;|x|\right)=m\{t\ge0:\mu\left(t;x\right)>s\}=d\left(s;\mu\left(x\right)\right),~s\ge0.\]
If $\mathcal{M}=L_\infty\left(0,\infty\right)$ with semifinite faithful normal trace $\tau$ given as Chapter \ref{intro}, then the distribution function $d\left(|x|\right)$ of $|x|$ and the decreasing rearrangement of the function $|x|$ are defined by
\[d\left(s;|x|\right)=m\{t>0:|x\left(t\right)|>s\}, ~s\in[0,\infty)\]
and
\[\mu\left(t;x\right)=\inf\{s\ge0:d\left(s;|x|\right)\le t\}, ~t\ge 0,\]
 respectively.
 For a closed, densely defined linear operator $x:\cD(x)\to H$, the null projection $n\left(x\right)$ of $x$ is the projection onto $\operatorname{Ker}\left(x\right)$, the range projection $l\left(x\right)$ of $x$ is the projection onto $\overline{\operatorname{Ran}\left(x\right)}$ and the support projection $r\left(x\right)=\mathbf{1}-n\left(x\right)$ 
 of $x$
 is the projection onto $\overline{\operatorname{Ran}\left(x^*\right)}$. 
 A two-sided ideal $\mathcal{F}\left(\mathcal{M},\tau\right)$ in $\mathcal{M}$ is defined by 
 \[\mathcal{F}\left(\mathcal{M},\tau\right)=\{x\in\mathcal{M}:\tau\left(r\left(x\right)\right)<\infty\}=\{x\in\mathcal{M}:\tau\left(l\left(x\right)\right)<\infty\}.\] 
 If $\mathcal{M}=B\left(H\right)$ is equipped with the standard trace, then $\mathcal{F}\left(\mathcal{M},\tau\right)$ consists of all finite-rank operators on $H$.

Recall that a linear subspace $E(\cM,\tau)$ of $S(\cM,\tau)$, equipped with a quasi-norm $\norm{\cdot}_{E\left(\cM,\tau\right)}$, is called {\itshape symmetric}
if $y\in E(\cM,\tau),x\in S\left(\cM,\tau\right)$ and $\mu(x)\le \mu(y)$ imply that $x\in E(\cM,\tau)$ and $\norm{x}_{E\left(\cM,\tau\right)}\le \norm{y}_{E\left(\cM,\tau\right)}$.
The set of all self-adjoint elements in $E\left(\cM,\tau\right)$ is denoted by $E\left(\cM,\tau\right)_h$.

\section{Jordan $*$-isomorphism} 
Let $(\cM,\tau)$ and $(\cN,\nu)$ be two von Neumann algebras equipped with semifinite faithful normal traces $\tau$ and $\nu$, respectively. 
A complex-linear mapping $J:\cM\to \cN$ (respectively, $S(\cM,\tau)\to S(\cN,\nu)$) is called a Jordan $*$-homomorphism if 
\[
J(x^*)=J(x)^*~ \text{and}~ J\left(x^2\right)=J(x)^2,~x\in \cM \,
\text{(respectively, $x\in S(\cM,\tau)$)},
\]
equivalently, (the following equation is equivalent to $J(x^2)=J(x)^2$)
\[
J(xy+yx)=J(x)J(y)+J(y)J(x),~x,y\in \cM\,
\text{(respectively, $x,y\in S(\cM,\tau)$)}
\]
(see \cite{BR,Y81}).
A Jordan $*$-homomorphism $J$ is called a {\itshape Jordan $*$- monomorphism} if $J$ is injective.
If $J$ is a bijective Jordan $*$-homomorphism, then $J$ is said to be a {\itshape Jordan $*$- isomorphism}. 

Recall that a net $\{x_\alpha\}_{\alpha\in \Lambda}$ of self-adjoint elements in $S\left(\mathcal{M},\tau\right)$ is called increasing (denoted simply by $x_\alpha\uparrow$) if $x_\alpha\le x_\beta$ whenever $\alpha\le \beta$ in $\Lambda$.
If $\{x_\alpha\}_{\alpha\in \Lambda}$ is increasing and $x=\sup_{\alpha}x_\alpha$ exists, then we write $x_\alpha\uparrow x$ \cite{DPS}.

Let $\mathcal{M}_1$ and $\mathcal{M}_2$ be two semifinite von Neumann algebras equipped with semifinite faithful normal traces $\tau_1$ and $\tau_2$, respectively.
Note that a Jordan $*$-isomorphism $J$ from $S\left(\mathcal{M}_1,\tau_1\right)$ onto $S\left(\mathcal{M}_2,\tau_2\right)$ is necessarily normal.
Indeed, let 
$\{x_\alpha\}_{\alpha\in \Lambda}\subset S\left(\mathcal{M}_1,\tau_1\right)$ with 
$0\le x_\alpha\uparrow x\in S\left(\mathcal{M}_1,\tau_1\right)$.
Note that $J$ and $J^{-1}$ both are order-preserving (e.g., see the proof of \cite[p. 113]{DPS}).
Then, $0\le J(x_\alpha)\le J(x)$,
which shows that $\sup_{\alpha\in \Lambda}J(x_\alpha)\le J(x)$.
Let $0\le y\in S\left(\mathcal{M}_2,\tau_2\right)$ with $\sup_{\alpha\in \Lambda}J\left(x_\alpha\right)\le y$.
Then, $x_\alpha\le J^{-1}(y)$ for all $\alpha\in \Lambda$ and so, $x\le J^{-1}(y)$. That is, $J(x)\le y$.
This together with $\sup_{\alpha\in \Lambda}J(x_\alpha)\le J(x)$ implies that $J(x)=\sup_{\alpha\in \Lambda}J\left(x_\alpha\right)$.
Moreover, a Jordan $*$-isomorphism from $\cM_1$ onto $\cM_2$ must be normal \cite[Appendix A]{RR}.

For more details on Jordan $*$-homomorphisms on von Neumann algebras, the reader is referred to \cite[Theorem 3.2.3]{BR}.

\chapter{Extreme points of the unit ball of a  Marcinkiewicz space} \label{unit ball}
Throughout this chapter, unless stated otherwise,
we assume that $\mathcal{M}$ is a semifinite von Neumann algebra equipped with a semifinite faithful normal trace $\tau$.
\section{The case when the unit ball of a quasi-Banach Marcinkiewicz space has extreme points}
In this section, 
we give a necessary and sufficient condition for $x$ to be an extreme point of the unit ball of $M_\Phi\left(\mathcal{M},\tau\right)$
in the case when $\cM$ is either non-atomic or  atomic with all  atoms having the same trace.
 Note that  the unit ball of the  separable part $\ell_{p,\infty}^0$ has no extreme point  when $0<p<\infty$\footnote{It suffices to consider the case when   $ x=\left(x_n\right)_{n=1}^\infty$ is a positive and decreasing element in the unit ball of $\ell_{p,\infty}^0$.  
There exists $n$ such that $x_{n}\ne n^{-1/p}$.
Note that $x_n<  n^{-1/p}$ and  $|2 x_n -n^{-1/p} |\le n^{-1/p} $. 
Letting $y_1=(x_1,x_2,\dots, x_{n-1}, n^{-1/p} ,x_{n+1},x_{n+2},\dots )$
 and
 $y_2= (x_1,x_2,\dots, x_{n-1}, 2  x_{n} -n^{-1/p} ,x_{n+1},x_{n+2},\dots )    $ (in the unit ball of $\ell_{p,\infty}^0$),
 we have 
 $2 x= y_1 + y_2$.

 }.

Recall that a set 
\[W \subseteq \left(L_1+L_\infty\right)\left(\cM,\tau\right) 
 = \{  x\in S(\cM,\tau):  \mu(x)\chi_{_{(0,1)}}\in L_1(0,\infty ) \}\]
 is said to be {\itshape fully symmetric} if $ y \in  \left(L_1+L_\infty\right)\left(\cM,\tau\right)  $ and $x \in W$ with
\[
\int_0^t \mu(s;y)ds\le \int_0^t\mu(s;x)ds
\]
for all $t\ge 0$ (denoted by $y \prec\prec x$)
implies $y \in W$. 
Let $x\in \left(L_1+L_\infty\right)\left(\cM,\tau\right)$. 
The orbit $\Omega\left(x\right)$ of $x$ is defined by
\[\Omega\left(x\right):=\{y\in \left(L_1+L_\infty\right)\left(\cM,\tau\right):y\prec\prec x\}.\]

Below, we consider the case when $\cM$ is non-atomic.

\begin{proof}[Proof of Theorem \ref{atomless extreme points}]
$\left(\Leftarrow\right)$. 
Assume that $\mu\left(x\right) =\frac{1}{\Phi}\chi_{_{\left(0,\tau\left(\mathbf{1}\right)\right)}}$. 
By \cite[Theorem 1.1]{CKS},  $x$ is an extreme point of $\Omega\left(x\right)$. 
Since the unit ball of $M_\Phi\left(\mathcal{M},\tau\right)$ is a subset of $\Omega\left(x\right)$, it follows that $x$ is an extreme point of the unit ball of $M_\Phi\left(\mathcal{M},\tau\right)$.

$\left(\Rightarrow\right)$.
Noting that $x$ is $\tau$-compact,
by \cite[Lemma 1.3]{CKS92}, 
there 
exist a  non-atomic abelian von Neumann subalgebra $\cN$ of $r(x)\cM r(x)$ and a $*$-isomorphism
\[
U:S\left(\mathcal{N},\tau|_\cN\right)\to 
S\left(0,\tau\left(r(x) \right)\right)
\]
such that 
\begin{align}\label{mupreserving}
\mu(z)=\mu\left(U(z)\right) ~ \text{for all} ~z\in S\left(\cN,\tau|_\cN\right).
\end{align}
Assume by contradiction  that \[\mu\left(|x|\right)= \mu\left(x\right)\ne\frac{1}{\Phi}\chi_{_{\left(0,\tau\left(r(x)\right)\right)}}.\]
Let $ 0\le y= U\left(|x|\right)\in S  \left(0,\tau\left(r\left(|x|\right) \right)\right)$.
Note that $\mu(y)\stackrel{\eqref{mupreserving}}{=}\mu(x)$ and $\norm{y}_{M_\Phi\left(\mathcal{M},\tau\right)}\le 1$. 
Since $\mu\left(x\right)$ is right-continuous, it follows that there exist $t_1>t_0>0$ and $\varepsilon>0$ such that for all $t_0\le t\le t_1$,  
\begin{align*} 
0<\mu(t;x)-\varepsilon<\mu(t;x)+\varepsilon<\frac1{\Phi(t)}.
\end{align*}
By \cite[Lemma 3.7.8]{DPS}, there exist two measurable subsets $B_1$ and $B_0$ of $ \left(0,\tau\left(r\left(|x|\right)\right)\right)$ such that $B_0\subset B_1$, $ m\left(B_0\right) = t_0$, $m\left(B_1\right)=t_1$, 
\[ \{\mu\left(t_0;x\right)<  y  \}\subset  B_0 \subset \{\mu\left(t_0;x\right) \le  y  \} \]
and 
\[ \{\mu\left(t_1;x\right) < y  \}\subset  B_1 \subset \{\mu\left(t_1;x\right) \le  y \}. \]
Denote $B:=B_1\setminus B_0$.

Let $x=u|x|$ be the polar decomposition of $x$ and
\[a:=y + \frac{\varepsilon }{2}\chi_{_{B}},  \quad b:=y  - \frac{\varepsilon }{2}  \chi_{_{B}}. \] 
Since \cite[Lemma 3.7.10]{DPS}
\[\mu\left(y \chi_{_{B_0}} \right)=\mu(x)\chi_{_{\left(0,t_0\right)}},~
\mu(y\chi_{_B})=\mu\left(\mu(x)\chi_{_{(t_0,t_1)}}\right),~ 
\mu\left(y\chi_{_{\left(0,\infty \right)\setminus B_1 }}\right)=\mu\left(\mu(x)\chi_{_{\left(t_1,\infty \right)}}\right),\]
it follows that
\[\mu\left(y \chi_{_{B_0}} \right)\le \frac{1}{\Phi} \chi_{_{\left(0,t_0\right)}}, ~\mu\left( y\chi_{_B} \pm \frac{\varepsilon}{2}\chi_{_B} \right)
 \le\mu\left(\frac{1}{\Phi} \chi_{_{\left(t_0,t_1\right)}}\right), ~
  \mu\left(y \chi_{_{\left(0,\infty \right)\setminus B_1 }} \right) \le \mu\left(\frac{1}{\Phi} \chi_{_{\left(t_1,\infty \right)}}\right).\]
This implies that $\mu\left(a\right),\mu\left(b\right)\le \frac{1}{\Phi}$.
Noting that 
$\mu(x)\in S_0\left(0,\tau(r(x))\right)$ and $B\subset {\rm supp} \left(\mu(x)\right)$,
we have $a,b\in S\left(0,\tau(r(x))\right)$.
Hence,
\[
\mu\left(uU^{-1}(a)\right)\le \mu\left(U^{-1}(a)\right)=\mu(a)\le \frac{1}{\Phi},
\]
i.e., $\norm{uU^{-1}(a)}_{M_\Phi(\cM,\tau)}\le 1$. 
Similarly,  $\norm{uU^{-1}(b)}_{M_\Phi(\cM,\tau)}\le 1$.
By the fact that 
  $y=\frac{a+b}{2}$ and $a\ne b$,
we have 
\[
x=u|x|=uU^{-1}\left(y\right)
=\frac{uU^{-1}\left(a\right)+uU^{-1}\left(b\right)}{2}.
\]
    Since $x$ is an extreme point of the unit ball, it follows that  $uU^{-1}\left(a\right)=uU^{-1}\left(b\right)$, i.e.,
   \[U^{-1}\left(a\right)= r(x)U^{-1}\left(a\right)=r(x) U^{-1}\left(b\right) =U^{-1}\left(b\right).\] 
   By the bijectivity of $U$,
   we have 
    $a=b$, which is a contradiction. 
    Therefore, \[\mu\left(x\right)= \frac{1}{\Phi}\chi_{_{\left(0,\tau\left(r(x)\right)\right)}}.\] 
    
    It suffices to consider the case when 
      $\tau\left(r(x)\right)<\tau\left(\mathbf{1}\right)\le\infty$. 
      There exists $y\in M_{\Phi}\left(\mathcal{M},\tau\right)$ such that $x\perp y$
      and $\mu\left(y\right)=\mu\left(\frac1\Phi\chi_{_{\left(\tau\left(r(x)\right),\tau\left(\mathbf{1}\right)\right)}}\right)$. 
      Then,
      \[\mu\left(x+y\right)=\mu\left(x-y\right)=\frac1\Phi\chi_{_{\left(0,\tau\left(\mathbf{1}\right)\right)}},\] 
      i.e., $x+y,x-y$ are in the unit ball of $M_\Phi(\cM,\tau)$.
      Note that $x=\frac{\left(x+y\right)+\left(x-y\right)}2$ and $x+y\ne x-y$, a contradiction. 
      As a consequence, $\mu\left(x\right)= \frac{1}{\Phi}\chi_{_{\left(0,\tau\left(\mathbf{1}\right)\right)}}$.
    The proof is complete. 
\end{proof}

The following corollary is a direct consequence of Theorem \ref{atomless extreme points} when $\Phi(t)=t^{\frac{1}{p}},~ 1<p<\infty $.

\begin{corollary} \label{1<p<infty}
    Let $\cM$ be a non-atomic von Neumann algebra equipped with a semifinite faithful normal trace $\tau$. 
    Assume that $1<p<\infty$ and 
    $x\in L_{p,\infty}\left(\mathcal{M},\tau\right)$ with quasi-norm $1$. 
    Then, $x$ is an extreme point of the unit ball of $L_{p,\infty}\left(\mathcal{M},\tau\right)$ if and only if $\mu\left(t;x\right) =t^{-\frac{1}{p}},~ \forall t\in \left(0,\tau\left({\bf 1}\right)\right)$.
\end{corollary}

For $0<p\le1$, the unit ball of  $L_{p,\infty}\left(\mathcal{M},\tau\right)$ has no extreme point (see Corollary~\ref{0<p le1} below).

Next, we turn to the case where $\mathcal{M}$ is atomic.

\begin{proof}[Proof of Theorem \ref{atomic}]
$\left(\Rightarrow\right)$.
By an  argument similar to that used in the proof of Theorem \ref{atomless extreme points}, we obtain that $\mu\left(n;x\right)=\frac1{ l_\Phi(n)},~ \forall n\in \mathbb{N},n<\tau\left(\mathbf{1}\right)$.

$\left(\Leftarrow\right)$. 
Assume that 
$\mu\left(n;x\right) =\frac{1}{l_\Phi(n)},~n<\tau(\mathbf{1})$. 
Let $L_{\infty}(0,1)$ be a commutative von Neumann algebra equipped with the trace $\iota $ defined by
    \[\iota\left(f\right)=\int_{0}^{1} fdm,\]
where $m$ is the Lebesgue measure.
We denote by $\sigma:=\iota \otimes \tau$ a tensor product of the traces $\iota$ and $\tau$ satisfying
    \[\sigma\left(\lambda \mathbf{1} \otimes x\right)=\lambda \tau\left(x\right).\]
Let $\tilde{\pi}:x \mapsto \mathbf{1}\overline\otimes x$ be a $*$-monomorphism from $S\left(\mathcal{M},\tau\right)$ into $S\left(L_\infty \left(0,1\right)\otimes \mathcal{M},\sigma\right)$. 
The set $\Omega\left(\mathbf{1}\overline\otimes x\right)$ is convex and fully symmetric 
in $L_1\left(L_\infty(0,1)\otimes \mathcal{M},\sigma\right)+L_\infty(0,1)\otimes \mathcal{M}$.
By Theorem \ref{atomless extreme points}, we obtain that $\mathbf{1}\overline\otimes x$ is an extreme point of $\Omega\left(\mathbf{1}\overline \otimes x\right)$. 
Hence,  $x$ is an extreme point of $\Omega\left(x\right)$. 
By the proof of Theorem \ref{atomless extreme points}, $x$ is an extreme point of the unit ball of $M_\Phi\left(\mathcal{M},\tau\right)$. \end{proof}

An immediate application of Theorem \ref{atomic} gives a complete description of extreme points of the unit ball in $\ell_{p,\infty},~0<p<\infty$, as follows.

 \begin{corollary}\label{atom extremes}
        Let $0<p<\infty$. 
        Let $x\in \ell_{p,\infty}$ (respectively, $C_{p,\infty }$) with quasi-norm $1$. 
        Then,  $x$ is an extreme point of the unit ball of $\ell_{p,\infty }$ (respectively, $C_{p,\infty }$) if and only if  $\mu\left(n;x\right) =\left(n+1\right)^{-\frac{1}{p}}$, $\forall n\ge 0$. 
 \end{corollary}  

\section{The case when the unit ball of a quasi-Banach Marcinkiewicz space has no extreme point}

For $1<p<\infty$, letting $x\in L_{p,\infty}\left(0,\infty\right)$ with quasi-norm $1$,
we have proved that $x$ is an extreme point of the unit ball of $L_{p,\infty}\left(0,\infty\right)$ if and only if $\mu(t;x)=t^{-1/p}$ for all $t>0$.
For $0<p\le1$,
note that  the function $t\mapsto t^{1/p}$ does not satisfy assumptions of Theorem \ref{atomless extreme points} ($t\mapsto t^{-1/p}$ does not belong to $\left(L_1+L_\infty\right)\left(0,\infty\right)$).
In this section,
we surprisingly observe that the unit ball of $L_{p,\infty}\left(0,\infty\right), ~0<p\le1$, has no extreme point (see Corollary~\ref{0<p le1}).

The main result of this section, Proposition \ref{extend}, gives
a sufficient condition on $\Phi$ under which
the unit ball of $M_\Phi(0,\infty)$ has no extreme point.

Let $\Phi:[0,\infty)\to[0,\infty)$ be strictly increasing, admissible and continuous on $[0,\infty)$ satisfying $\Phi(\infty)=\infty$.
Choose
 \[h(t):=\frac1{\Phi(t)\left(\Phi(t)+1\right)}, ~t>0.\]
Let 
\begin{align} \label{u,}
u(t)=\frac1{\Phi(t)}-h(t)=\frac{1}{\Phi(t)+1},  ~t>0
\end{align}
and
\begin{align} \label{,v}
v(t)=\frac1{\Phi(t)}+h(t)=\frac{\Phi(t)+2}{\Phi(t)\left(\Phi(t)+1\right)}, ~t>0.
\end{align}
Note that for all $t>0$, $u(t)\le \frac1{\Phi(t)} \le v(t)$ and
\begin{align}\label{u v 2}
\frac{u(t)+v(t)}{2}=\frac1{\Phi(t)}.
\end{align}
Observe that $\Phi$ is bijective.
Set
\begin{align} \label{a(t)}
a(t):=\frac1 {v(t)}=\frac{\Phi(t)\left(\Phi(t)+1\right)}{\Phi(t)+2}, ~t>0.
\end{align}
Since the maps $t\mapsto \frac{t(t+1)}{t+2}, ~t>0$ and $\Phi$ are both strictly increasing, bijective and each of them has range $\left(0,\infty\right)$, it follows that the inverse of $a$ exists, denoted by $a^{-1}$. 
Define $t_+:=\max\{t,0\}$ for all $t\in \mathbb{R}$.
The following lemma is the main tool in the proof of Proposition \ref{extend}. 
It shows that, under the condition given in Proposition \ref{extend}, $\frac{1}{\Phi}$ admits a representation as the midpoint of two distinct elements in $B_{M_\Phi(0,\infty)}$.

It is noteworthy that the condition \eqref{b(t)>0} of Lemma \ref{plus} is sufficient but not necessary  for the decomposition \eqref{decomposition}.

\begin{lemma} \label{plus}
    Let $\Phi:[0,\infty)\to[0,\infty)$ be strictly increasing, admissible and continuous on $[0,\infty)$ satisfying $\Phi(\infty)=\infty$. 
    If for all $t>0$,
    \begin{align}\label{b(t)>0}
    a^{-1}(\Phi(t))+\Phi^{-1}\left(\left(\Phi(t)-1\right)_+\right)<2t,
    \end{align}
    then there exist two non-negative functions $f,g\in B_{M_\Phi(0,\infty)}$ 
    with $f(t)\ne g(t)$ for all $t\in\left(0,\infty\right)$ 
    satisfying
    \begin{align} \label{decomposition}
    \frac{f(t)+g(t)}{2}=\frac1{\Phi(t)}, ~\forall t>0.
    \end{align}
\end{lemma}


\begin{remark}
    If $\Phi$ satisfies the assumptions of Lemma \ref{plus}, then $\frac 1\Phi\notin \left(L_1+L_\infty\right)\left(0,\infty\right)$.
    Indeed, assume by contradiction that $\frac 1\Phi\in \left(L_1+L_\infty\right)\left(0,\infty\right)$.
    Fix a scalar $r>0$ such that $\Phi(2r)<1$.
    In this case, for all $0<t\le r$, we have that
    $$\Phi^{-1}\left(\left(\Phi(t)-1\right)_+\right)\le \Phi^{-1}\left(\left(\Phi(2r)-1\right)_+\right)=0.$$
    By \eqref{b(t)>0}, we have that
    $$a^{-1}\left(\Phi(t)\right)<2t.$$
    Then,
    \begin{align} \label{>2-1}
        \frac 1 {\Phi(t)}>\frac 1{a(2t)} \stackrel{\eqref{a(t)}}{=}\frac{\Phi(2t)+2}{\Phi(2t)\left(\Phi(2t)+1\right)}&=\frac 2{\Phi(2t)}-\frac1{\Phi(2t)+1}\nonumber\\
        &>\frac 2{\Phi(2t)}-1.
    \end{align}
    Integrating both sides of \eqref{>2-1} over $(0,r)$, by $\frac 1\Phi\in \left(L_1+L_\infty\right)\left(0,\infty\right)$, we obtain that
    $$\int^r_0\frac 1{\Phi(t)}dt>\int_0^r\frac{2}{\Phi(2t)}-1dt=\int_0^{2r}\frac1 {\Phi(t)}dt-r,$$
    which implies that
    $$r\stackrel{\Phi(2r)<1}{<}\int_r^{2r}\frac1{\Phi(t)}dt<r,$$
    a contradiction.
    Therefore, $\frac 1\Phi\notin \left(L_1+L_\infty\right)\left(0,\infty\right)$.
\end{remark}

 To prove Lemma \ref{plus}, we need the following two auxiliary lemmas.

\begin{lemma} \label{E}
    Let $f:\left(0,\infty\right)\to \left(0,\infty\right)$ be a continuous function.
    Then, there exists a measurable set $E\subset \left(0,\infty\right)$ with infinite measure such that
    \[\left|m\left(E\cap \left(0,t\right)\right)-\frac t 2\right|<f(t), ~\forall t>0.\]
\end{lemma}
\begin{proof}
    For all $k\in \mathbb{Z}$, let 
    \[c_k:=\min_{t\in \left[2^k,2^{k+1}\right]}f(t)>0.\]
    and we may decompose $\left[2^k,2^{k+1}\right]$ into finitely many  disjoint consecutive intervals $\left\{I_i^k\right\}_{i=0}^{N_k}$ with the same measure satisfying
    \begin{align} \label{ck}
    m\left(I_i^k\right)<4c_k, ~\forall i\in\{0,1,\dots,N_k\} .
    \end{align}
    For all such intervals $I_i^k$, let $E$ be a measurable set with infinite measure such that $E\cap I_i^k$ is the left half of $I_i^k$.
    Define
    \[g(t):=m\left(E\cap \left(0,t\right)\right)-\frac t 2, ~t>0.\]
     Fix a scalar $t>0$.
    There exists $k$ such that $t\in \left[2^k,2^{k+1}\right]$.
    Then, $t\in I_j^{k}$ for some $0\le j \le N_k$.
    Without loss of generality, we may assume that $I_j^k=[\alpha,\beta)$.
    In this case,  $m\left(E\cap(0,\alpha)\right)-\frac \alpha2=0$
    and so,
    $g(t)=m\left(E\cap(\alpha,t)\right)-\frac{t-\alpha}{2}$,
    that is,
    \[
    g(t)=
    \begin{cases}
    t-\alpha-\frac{t-\alpha}{2}, & t<\frac{\alpha+\beta}{2},\\
    \frac{\beta-\alpha}{2}-\frac{t-\alpha}{2}, & t\ge\frac{\alpha+\beta}{2}.
    \end{cases}
    \]
    This implies that 
    \begin{align*}
    |g(t)|\le \frac{\beta-\alpha}{4}=\frac{m\left(I_j^k\right)}{4}\stackrel{\eqref{ck}}{<}c_k\le f(t).
    \end{align*}
    This completes the proof.
\end{proof}

Although the following result is trivial, we present a  proof here for the sake of completeness.
\begin{lemma} \label{equivalence}
    Let $\Phi$ be admissible and continuous.
    Then,
    \[\sup_{s>0}s\Phi\left(d\left(s;|f|\right)\right)=\sup_{s>0}\Phi(s)\mu\left(s;f\right), ~\forall f\in M_\Phi\left(0,\infty\right).\]
\end{lemma}
\begin{proof}
    Denote the left hand supremum and the right hand supremum by $\alpha$ and $\beta$, respectively.
    Let $t>0$ be fixed  such that $\mu(t;f)>0$.
    For any $0<s<\mu(t;f)$, by \cite[Remark 3.1.2]{DPS}, we have that 
    $d(s;|f|)>t$
    and so,
    \[\Phi\left(d(s;|f|)\right)\ge \Phi(t).\]
    This implies that for all $0<s<\mu(t;f)$,
    \[\alpha\ge s\Phi\left(d(s;|f|)\right)\ge s\Phi(t).\]
    Taking $s\uparrow \mu(t;f)$, we obtain that
    \[\alpha\ge  \mu(t;f)\Phi(t),\]
    which shows that $\alpha\ge \beta$.

    For the proof of the converse inequality, fix a scalar $s>0$ such that $d(s;f)>0$.
    For any $0<t<d(s;|f|)$, by \cite[Proposition 3.1.4]{DPS}, we get that $\mu(t;f)>s$.
    Then, for all $0<t<d(s;|f|)$,
    \[\beta\ge \mu(t;f)\Phi(t)>s\Phi(t).\]
    Taking $t\uparrow d(s;|f|)$, since $\Phi$ is left-continuous, it follows that
    \[\beta\ge s\Phi\left(d(s;|f|)\right).\]
    This shows that $\beta\ge \alpha$.
    The proof is complete.
\end{proof}

\begin{proof}[Proof of Lemma \ref{plus}]
Let $u,v$ be defined as in \eqref{u,} and \eqref{,v}. 
For any measurable subset $B\subset \left(0,\infty\right)$ with positive measure,
define
\[f_B:=u\chi_{_B}+v\chi_{_{B^c}}, ~g_B:=v\chi_{_B}+u\chi_{_{B^c}}.\]
Note that $f_B(t)\ne g_B(t)$ for all $t\in \left(0,\infty\right)$.

By \eqref{u v 2}, it suffices to show the existence of a measurable set
$E\subset \left(0,\infty\right)$ with positive infinite measure
such that $f_E,g_E\in B_{M_\Phi\left(0,\infty\right)},$
i.e., 
\[d\left(\frac1{\Phi(s)};f_E\right),d\left(\frac1{\Phi(s)};g_E\right)\le s, ~\forall s>0\] 
(see Lemma \ref{equivalence}).

For any $B\subset \left(0,\infty\right)$ with positive measure and $s>0$,
we have

\[\left\{u>\frac1{\Phi(s)}\right\}=\left(0,\Phi^{-1}\left(\left(\Phi(s)-1\right)_+\right)\right), ~\left\{v>\frac1{\Phi(s)}\right\}=\left(0,a^{-1}(\Phi(s))\right).\]
Then, we have 
\begin{align*}
\left\{t>0:f_B(t)>\frac1{\Phi(s)}\right\}=&\left(B\cap\left\{u>\frac1{\Phi(s)}\right\}\right)\cup \left(B^c\cap \left\{v>\frac1{\Phi(s)}\right\}\right)\\
=&\left(B\cap\left(0,\Phi^{-1}\left(\left(\Phi(s)-1\right)_+\right)\right)\right)\cup \left(B^c\cap\left(0,a^{-1}(\Phi(s))\right)\right)
\end{align*}
and 
\[\left\{t>0:g_B(t)>\frac1{\Phi(s)}\right\}=\left(B^c\cap\left(0,\Phi^{-1}\left(\left(\Phi(s)-1\right)_+\right)\right)\right)\cup \left(B\cap\left(0,a^{-1}(\Phi(s))\right)\right).\]
Letting 
\[\gamma_B(s):=m\left(B\cap\left(\Phi^{-1}\left(\left(\Phi(s)-1\right)_+\right),a^{-1}(\Phi(s))\right)\right),\]
 we have 
\begin{align*} 
&\quad d\left(\frac1{\Phi(s)};f_B\right)\\
&=m\left(\left(B\cap\left(0,\Phi^{-1}\left(\left(\Phi(s)-1\right)_+\right)\right)\right)\cup \left(B^c\cap\left(0,a^{-1}(\Phi(s))\right)\right)\right)\\
&=m\left(\left(B\cap\left(0,a^{-1}(\Phi(s))\right)\right)\cup \left(B^c\cap\left(0,a^{-1}(\Phi(s))\right)\right)\right)\\
&\quad-m\left(B\cap\left(\Phi^{-1}\left(\left(\Phi(s)-1\right)_+\right),a^{-1}(\Phi(s))\right)\right)\\
&=a^{-1}(\Phi(s))-\gamma_B(s)
\end{align*}
and
\begin{align*}
&\quad d\left(\frac1{\Phi(s)};g_B\right)\\
&=m\left(\left(B^c\cap\left(0,\Phi^{-1}\left(\left(\Phi(s)-1\right)_+\right)\right)\right)\cup \left(B\cap\left(0,a^{-1}(\Phi(s))\right)\right)\right)\\
&=m\left(\left(B^c\cap\left(0,\Phi^{-1}\left(\left(\Phi(s)-1\right)_+\right)\right)\right)\cup \left(B\cap\left(0,\Phi^{-1}\left(\left(\Phi(s)-1\right)_+\right)\right)\right)\right)\\
&\quad+m\left(B\cap\left(\Phi^{-1}\left(\left(\Phi(s)-1\right)_+\right),a^{-1}(\Phi(s))\right)\right)\\
&=\Phi^{-1}\left(\left(\Phi(s)-1\right)_+\right)+\gamma_B(s).
\end{align*}
These imply that   $d\left(\frac1{\Phi(s)};f_B\right),d\left(\frac1{\Phi(s)};g_B\right)\le s$ is equivalent to
\begin{align}\label{le ge}
a^{-1}(\Phi(s))-s\le \gamma_B(s)\le s-\Phi^{-1}\left(\left(\Phi(s)-1\right)_+\right).
\end{align}
Let\footnote{In the special case when  $B$ occupies half of $\left(\Phi^{-1}\left(\left(\Phi(s)-1\right)_+\right),a^{-1}(\Phi(s))\right)$, that is, 
\[m\left(B\cap\left(\Phi^{-1}\left(\left(\Phi(s)-1\right)_+\right),a^{-1}(\Phi(s))\right)\right)=\frac12 m\left(\left(\Phi^{-1}\left(\left(\Phi(s)-1\right)_+\right),a^{-1}(\Phi(s))\right)\right),\]
  by the definition of $\gamma$,  we have 
\[\gamma_B(s)=\frac{a^{-1}(\Phi(s))-\Phi^{-1}\left(\left(\Phi(s)-1\right)_+\right)}{2}\]
and 
\[\begin{aligned}
\gamma_B(s)-\left(a^{-1}(\Phi(s))-s\right)
&=s -\Phi^{-1}\left(\left(\Phi(s)-1\right)_+\right)-\gamma_B(s)\\
&=s-\frac{a^{-1}(\Phi(s))+\Phi^{-1}\left(\left(\Phi(s)-1\right)_+\right)}{2}.\end{aligned}\]}
\[b(t):=t-\frac{a^{-1}(\Phi(t))+\Phi^{-1}\left(\left(\Phi(t)-1\right)_+\right)}{2}
, ~\forall t>0\]
and 
\[F_B(t):=m\left(B\cap \left(0,t\right)\right)-\frac t2, ~\forall t>0.\]
By \eqref{b(t)>0}, we have that $b(t)>0$ for all $t>0$.
Observe that for any measurable subset $B\subset \left(0,\infty\right)$ with positive measure, we have 
\begin{align*}
&F_B\left(a^{-1}(\Phi(t))\right)-F_B\left(\Phi^{-1}\left(\left(\Phi(t)-1\right)_+\right)\right)\\=&m\left(B\cap \left(0,a^{-1}(\Phi(t))\right)\right)-\frac {a^{-1}(\Phi(t))}2\\
&-\left(m\left(B\cap \left(0,\Phi^{-1}\left(\left(\Phi(t)-1\right)_+\right)\right)\right)-\frac {\Phi^{-1}\left(\left(\Phi(t)-1\right)_+\right)}2\right)\\
=&\gamma_B(t)-\frac{a^{-1}(\Phi(t))-\Phi^{-1}\left(\left(\Phi(t)-1\right)_+\right)}{2}.
\end{align*}
Consequently, by \eqref{le ge}, $d\left(\frac1{\Phi(s)};f_B\right),d\left(\frac1{\Phi(s)};g_B\right)\le s$ is equivalent to
\begin{equation}\label{equivalent}
\begin{aligned}
&\quad\left|\gamma_B(s)-\frac{a^{-1}(\Phi(s))-\Phi^{-1}\left(\left(\Phi(s)-1\right)_+\right)}{2}\right|\\
&=\left|F_B\left(a^{-1}(\Phi(s))\right)-F_B\left(\Phi^{-1}\left(\left(\Phi(s)-1\right)_+\right)\right)\right|\\
&\le b(s)=s-\frac{a^{-1}(\Phi(s))+\Phi^{-1}\left(\left(\Phi(s)-1\right)_+\right)}{2}.
\end{aligned}
\end{equation}
Note that 
\begin{align}
\label{two F}
&\quad\left|F_B\left(a^{-1}(\Phi(s))\right)-F_B\left(\Phi^{-1}\left(\left(\Phi(s)-1\right)_+\right)\right)\right|\nonumber\\
&\le \left|F_B\left(a^{-1}(\Phi(s))\right)\right|+\left|F_B\left(\Phi^{-1}\left(\left(\Phi(s)-1\right)_+\right)\right)\right|
\end{align}
and

\begin{equation} \label{ab}
\begin{aligned}
b\left(\Phi^{-1}(a(t))\right)=b(s) &\text{ when } t=a^{-1}(\Phi(s)),\\
b\left(\Phi^{-1}\left(\Phi(t)+1\right)\right)=b(s) 
&\text{ when }t=\Phi^{-1}\left(\left(\Phi(s)-1\right)_+\right)>0.
\end{aligned}
\end{equation}
Define
\[w(t):=\frac18\min\left\{b\left(\Phi^{-1}(a(t))\right),b\left(\Phi^{-1}(\Phi(t)+1)\right)\right\}, ~t>0.\]
    Since $b(t)>0$ for all $t>0$, it follows that $w(t)>0$ for all $t>0$.
    Observe that $w(t)$ is a continuous function.
    Applying Lemma \ref{E} to $w$, we obtain that there exists a measurable set $E\subset \left(0,\infty\right)$ with infinite measure such that
    \[\left|F_E(t)\right|<w(t), ~\forall t>0.\]
    Then, by \eqref{ab},
    \[\left|F_E\left(a^{-1}(\Phi(s))\right)\right|<w\left(a^{-1}(\Phi(s))\right)\le \frac18 b\left(\Phi^{-1}\left(a\left(a^{-1}\left(\Phi(s)\right)\right)\right)\right) =\frac18b(s)\]
    and for $\Phi(s)-1>0$,
    \begin{align*}
\left|F_E\left(\Phi^{-1}\left(\left(\Phi(s)-1\right)_+\right)\right)\right|&<w\left(\Phi^{-1}\left(\left(\Phi(s)-1\right)_+\right)\right)\\
&\le \frac18 b\left(\Phi^{-1}\left(\Phi\left(\Phi^{-1}\left(\left(\Phi(s)-1\right)_+\right)\right)+1\right)\right) =\frac18b(s)
    \end{align*}
    (note that, for $-1< \Phi(s)-1\le 0$, $F_E\left(\Phi^{-1}\left(\left(\Phi(s)-1\right)_+\right)\right)=0$).
    This together with \eqref{two F} proves \eqref{equivalent}, i.e., $d\left(\frac1{\Phi(s)};f_E\right),d\left(\frac1{\Phi(s)};g_E\right)\le s$. 
    In particular, $f_E,g_E\in B_{M_\Phi\left(0,\infty\right)}$.
    By replacing $f$  and $ g$ with $f_E$ and $g_E$, respectively, 
    the proof is complete.
\end{proof}

The following result is a special case of \cite[Corollary 7.6]{BS}.
\begin{lemma} \label{map}
   Let $\Phi:[0,\infty)\to [0,\infty)$ be admissible satisfying $\Phi(\infty)=\infty$.
    Let $f\in M_\Phi\left(0,\infty\right)$ and $S:=\{t>0:|f(t)|>0\}$.
    Then, there exists a measure-preserving map $\phi:S\to \left(0,m(S)\right)$ such that
    \[|f|=\mu(f)\circ \phi,~ a.e. ~\text{on}~ S.\]
\end{lemma}

Now, we apply Lemma \ref{plus} and Lemma \ref{map} to prove Proposition \ref{extend}.

\begin{prop} \label{extend}
    Let $\Phi:[0,\infty)\to[0,\infty)$ be strictly increasing, admissible and continuous on $[0,\infty)$ satisfying $\Phi(\infty)=\infty$.
    Let $a$ be given by \eqref{a(t)}.
    If for all $t>0$,
    \begin{align*}
    a^{-1}(\Phi(t))+\Phi^{-1}\left(\left(\Phi(t)-1\right)_+\right)<2t,
    \end{align*}
    then the unit ball of  weak Marcinkiewicz function space $M_\Phi\left(0,\infty\right)$ has no extreme point.
\end{prop}
\begin{proof}
    Note that $0$ is not an extreme point of the unit ball of $M_{\Phi}\left(\mathcal{M},\tau\right)$.
    Fix  $0\ne f\in M_\Phi(0,\infty )$ such that $\norm{f}_{M_\Phi\left(0,\infty\right)}\le1$.
    Let $S:=\{t>0:|f(t)|>0\}$.
    By Lemma~\ref{map}, there exists a measure-preserving map $\phi:S\to \left(0,m(S)\right)$ such that
    \[|f|=\mu(f)\circ \phi,~ a.e. ~\text{on}~ S.\]
    Extend $\phi$ to $(0,\infty)$ by setting $\phi(t)=0,~t\not\in S$.
    Define
    \begin{align*}
    g(t)=
    \begin{cases}
    \frac1{\Phi\circ \phi}\chi_{_S}(t),& t\in S, \\
    0, & t\notin S.
    \end{cases}
    \end{align*}
    Since $\phi$ is measure-preserving, it follows that 
    \[d(s;g)=m\{t>0:g(t)>s\}=m\left\{t\in S:0<\phi(t)<\Phi^{-1}\left(\frac1s\right)\right\}\le \Phi^{-1}\left(\frac1s\right).\]
    By Lemma \ref{equivalence}, we have that $g\in B_{M_\Phi\left(0,\infty\right)}$.
    Note that 
    \begin{align} \label{strict?}
    |f(t)|=\mu(\phi(t);f)\le \frac1{\Phi(\phi(t))}=g(t) 
    \end{align}
    almost everywhere on $S$. 
    Let 
    \[
    \theta(t)=
    \begin{cases}
        \frac {f(t)}{|f(t)|}, &t\in S,\\
        0, & t\notin S.
    \end{cases}
    \]

    Below, we will prove that $f$ is not an extreme point of the unit ball.
    The proof is decomposed into two parts.
    \begin{enumerate} [label=(\roman*)]
    \item Assume that the inequality in \eqref{strict?} is strict on some subset $S_1$ of $S$ with positive measure.
    Let $a_1:=\theta g$ and $b_1:=\theta \left(2|f|-g\right)$.
    Note that $a_1,b_1 \in B_{M_\Phi\left(0,\infty\right)}$ and $f=\frac{a_1+b_1}{2}$.
    However, $a_1\ne b_1$.
    Hence, $f$ is not an extreme point of the unit ball.
    \item Assume that $|f|=g$ a.e. on $S$. By Lemma \ref{plus}, there
    exist two non-negative functions $a_1,b_1$ on $\left(0,\infty\right)$ in the unit ball of $M_\Phi\left(0,\infty\right)$  with $a_1(t)\ne b_1(t)$ for all $t\in \left(0,\infty\right)$ satisfying
    \[\frac{a_1(t)+b_1(t)}{2}=\frac1{\Phi(t)}, ~\forall t>0.\]
    Let $c_1(t):=\theta a_1(\phi(t)), ~d_1(t):=\theta b_1(\phi(t)), ~t>0$.
    Then, 
    \[\frac{c_1+d_1}{2}=\frac{\theta}{\Phi\circ\phi}=\theta g=f ~\text{a.e. on}~ \left(0,\infty\right).\] 
    Note that $c_1\ne d_1$ and $c_1,d_1$ are in the unit ball of $M_\Phi\left(0,\infty\right)$.
    Therefore, $f$ is not an extreme point of the unit ball.
    \end{enumerate}
    The proof is complete.
\end{proof}



\begin{corollary} \label{1,infty}
    Let $\mathcal{M}$ be a semifinite non-atomic von Neumann algebra equipped with a semifinite faithful normal trace $\tau$.
    Let $\Phi:[0,\infty)\to[0,\infty)$ be strictly increasing, admissible and continuous on $[0,\infty)$ satisfying $\Phi(\infty)=\infty$.
    Let $a(t)$ be as defined in \eqref{a(t)}.
    If for all $t>0$,
    \begin{align*}
    a^{-1}(\Phi(t))+\Phi^{-1}\left(\left(\Phi(t)-1\right)_+\right)<2t,
    \end{align*}
    then the unit ball of $M_\Phi\left(\mathcal{M},\tau\right)$ has no extreme point.
\end{corollary}

\begin{remark} \label{convexity}
Let $\Phi:[0,\infty)\to[0,\infty)$ be strictly increasing, admissible and continuous on $[0,\infty)$ satisfying $\Phi(\infty)=\infty$.
If, in addition, $\Phi$ is convex, then \eqref{b(t)>0} holds.
Indeed, fix a scalar $t>0$ and let $s=\Phi(t)$.
Let $a(t)$ be as defined in \eqref{a(t)}.
\begin{enumerate} [label=(\roman*)]
    \item Assume that $0<s\le 1$.
    In this case, 
    \begin{align} \label{s-1 = 0}
    \Phi^{-1}\left(\left(\Phi(t)-1\right)_+\right)=\Phi^{-1}\left(\left(s-1\right)_+\right)=0.
    \end{align}
    Note that
    \[a\left(\Phi^{-1}(2s)\right)-s=\frac{s^2}{s+1}>0,\]
    which together with the convexity of $\Phi$ implies that 
    \begin{align} \label{Phi(t) < 2t}
    a^{-1}\left(\Phi(t)\right)=a^{-1}(s)<\Phi^{-1}(2s)\le2\Phi^{-1}(s)=2t.
    \end{align}
    By \eqref{s-1 = 0} and \eqref{Phi(t) < 2t}, we have that 
    \[a^{-1}(\Phi(t))+\Phi^{-1}\left(\left(\Phi(t)-1\right)_+\right)<2t.\]
    This shows that \eqref{b(t)>0} holds.
    \item Assume that $s>1$.
    In this case, $\left(s-1\right)_+=s-1$.
    Observe that 
    \[a\left(\Phi^{-1}\left(s+1\right)\right)-s=\frac{2}{s+3}>0.\]
    Then, $a^{-1}(s)<\Phi^{-1}\left(s+1\right)$.
    By the convexity of $\Phi$, we have that
    \[\frac{a^{-1}(s)+\Phi^{-1}\left(s-1\right)}{2}<\frac{\Phi^{-1}\left(s+1\right)+\Phi^{-1}\left(s-1\right)}{2}\le \Phi^{-1}(s)=t.\]
    This implies that \eqref{b(t)>0} holds.
\end{enumerate}
 
\end{remark}

A combination of Corollary \ref{1,infty} and Remark \ref{convexity} yields the following result. 

\begin{corollary} \label{0<p le1}
    Let $0<p\le1$. Let $\mathcal{M}$ be a semifinite non-atomic von Neumann algebra equipped with a semifinite faithful normal trace $\tau$.
    If $\Phi(t)=t^{1/p}, ~t>0$, then $\Phi$ is convex and therefore 
    the unit ball of $L_{p,\infty}\left(\mathcal{M},\tau\right)$ has no extreme point.
    In particular,  the unit ball of $L_{p,\infty}\left(0,\infty \right)$ has no extreme point.
\end{corollary}


In particular, for $0<p\le1$, $L_{p,\infty}\left(\mathcal{M},\tau\right)$ is not isometric to $\ell_{p,\infty}$ and $\mathcal{C}_{p,\infty}$ when $\mathcal{M}$ is non-atomic.

\chapter{Extreme points of the positive unit ball of a  Marcinkiewicz space} \label{positive unit ball}



In this chapter, we discuss extreme points of the positive unit balls of quasi-Banach Marcinkiewicz spaces $M_\Phi\left(\mathcal{M},\tau\right)$  in the setting  where $\cM$ is either non-atomic or  atomic with all  atoms having the same trace,  and $\Phi:[0,\infty)\to [0,\infty)$ is an admissible function. 
 The results in this chapter are new, even in the commutative setting. 

We start this chapter with the following lemma (see  \cite[Lemma 5.4]{CKS} for a complete proof).

\begin{lemma} \label{Lemma 5.4}
    Let $\mathcal{M}$  be  non-atomic and  $0\le x,y,z \in \left(L_1+L_\infty\right)\left(\mathcal{M},\tau\right)$. 
    If $x=\frac{y+z}{2}$,  $\mu\left(x\right)=\mu\left(y\right)=\mu\left(z\right)$ and $\mu\left(\infty;x\right)=0
    $, then $x=y=z$.
\end{lemma}

Now, we can apply the above lemma and the proof of Theorem \ref{atomless extreme points} to get the following important characterization of extreme points of the positive unit ball.

\begin{proof}[Proof of Theorem \ref{posi non-atomic}]
$\left(\Leftarrow\right)$.
    By \cite[Remark 3.2.6(ii)]{DPS}, we obtain that for all $x\in S(\cM,\tau)$,
    \begin{align}\label{sptrace}
    \tau\left(r(x)\right)=\operatorname{inf} \{t\ge0:\mu\left(t;x\right)=0\}.
    \end{align}
    Assume that $x \in M_\Phi(\cM,\tau), y\in B^+_{M_{\Phi}\left(\mathcal{M},\tau\right)}$
    such that  
    \[
    \mu\left(x\right)=\frac{1}{\Phi}\chi_{_{\left(0,\tau\left(\mathbf{1}\right)\right)}},\quad
    \mu\left(y\right)=\frac{1}{\Phi}\chi_{_{\left(0,\tau\left(r(y)\right)\right)}}.
    \]
    Let $v,w\in B^+_{M_{\Phi}\left(\mathcal{M},\tau\right)}$ be arbitrary such that $y=\frac{v+w}2$.
    Since  $\mu\left(v\right),\mu\left(w\right)\le \mu\left(x\right)$, it follows that for all $t\in\left(0,\tau\left(r(y)\right)\right)$,
\begin{align*}
\int_0^t \mu\left(s;y\right)ds\stackrel{\tiny \mbox{\cite[Thm. 3.9.9]{DPS}}}{\le} \int_0^t\frac{1}{2}\left(\mu\left(s;v\right)+\mu\left(s;w\right)\right)ds&\le \int_0^t\frac{\mu\left(s;x\right)+\mu\left(s;x\right)}{2}ds\\
&=\int_0^t \mu\left(s;y\right)ds.
\end{align*}
This implies that 
\begin{align}\label{intmueq}
\int_0^t \mu\left(s;v\right)ds=\int_0^t \mu\left(s;w\right)ds=\int_0^t\mu\left(s;y\right)ds,
~t\in\left(0,\tau\left(r(y)\right)\right),
\end{align}  
 which together with \eqref{sptrace} yields that
$\tau(r(v)),\tau(r(w))\ge \tau(r(y))$.
    On the other hand, 
     by 
     $y=\frac{v+w}{2}$ and $y, v,w\ge 0$, we have 
 $r(y)\ge r(v),r(w)$.
Hence, \[\tau\left(r(y)\right)=\tau\left(r(v)\right)=\tau\left(r(w)\right) .\] 
In particular, 
$\mu\left(y\right)=\mu\left(v\right)=\mu\left(w\right)$ (see \eqref{intmueq}). 
Using Lemma \ref{Lemma 5.4}, 
we have $y=v=w$. 
Consequently, $y\in\operatorname{extr}\left(B^+_{M_{\Phi}\left(\mathcal{M},\tau\right)}\right).$

    
 $\left(\Rightarrow\right)$.   
    Now, let $x\in \operatorname{extr}\left(B^+_{M_{\Phi}\left(\mathcal{M},\tau\right)}\right)$.
    It follows immediately from the same argument used in the proof of Theorem \ref{atomless extreme points} that $\mu\left(x\right)=\frac1\Phi\chi_{_{\left(0,\tau\left(r(x)\right)\right)}}$.
\end{proof}

\begin{corollary} \label{posi non-atomic1}
    Assume that $p>1$.
    Let $\mathcal{M}$ be a semifinite non-atomic von Neumann algebra equipped with a semifinite faithful normal trace $\tau$. 
    Then, \[\operatorname{extr}\left(B^+_{L_{p,\infty}\left(\mathcal{M},\tau\right)}\right)=
    \left\{x\in B^+_{L_{p,\infty}\left(\mathcal{M},\tau\right)}:\mu\left(t;x\right)=t^{-\frac{1}{p}}\chi_{_{\left(0,\tau\left(r(x)\right)\right)}}(t),~\forall t>0
    \right\}.\]
\end{corollary}

Note that the positive unit ball of a weak $L_p$-space $L_{p,\infty}\left(\mathcal{M},\tau\right),~0<p\le1$, has a unique  extreme point $0$ when $\mathcal{M}$ is non-atomic.    

\begin{remark} \label{positive no extr}
Let $\mathcal{M}$ be a non-atomic von Neumann algebra equipped with a semifinite faithful normal trace $\tau$.
Let $\Phi:[0,\infty)\to[0,\infty)$ be strictly increasing, admissible and continuous on $[0,\infty)$ satisfying $\Phi(\infty)=\infty$.
Let $a$ be as defined in \eqref{a(t)}.
Assume that  for all $t>0$,
    \begin{align*}
    a^{-1}(\Phi(t))+\Phi^{-1}\left(\left(\Phi(t)-1\right)_+\right)<2t.
    \end{align*}
For a fixed $f \in B_{M_\Phi\left(0,\infty\right)}^+$,
letting
\begin{align*}
h(t):=
\begin{cases}
    \min\{f(t),g(t)-f(t)\}, &t\in S_1,\\
    0, & t\notin S_1.
\end{cases}
\end{align*}
and replacing $a_1$ and $a_2$ with $f-h$ and $f+h$ in part (i), respectively,
the same argument as used in the proof of Proposition~\ref{extend} 
shows that $f$ is an extreme point of the positive unit ball of $M_\Phi\left(0,\infty\right)$ if and only if $f=0$.
Therefore, 
the positive unit ball of a quasi-Banach Marcinkiewicz space $M_\Phi\left(\mathcal{M},\tau\right)$ has a unique extreme point $0$.
In particular, by Remark \ref{convexity}, the positive unit ball of a weak $L_p$-space $L_{p,\infty}\left(\mathcal{M},\tau\right), ~0<p\le1$, has a unique extreme point $0$.
\end{remark}

\begin{proof}[Proof of Theorem \ref{extr point atomic}]
The statement follows from  an argument similar to the proof of 
  Theorem~\ref{atomic} by replacing  Theorem \ref{atomless extreme points}
 with   Theorem \ref{posi non-atomic}.
\end{proof}

\chapter{Positive isometries on noncommutative quasi-Banach Marcinkiewicz spaces \texorpdfstring{$M_{\Phi}(\mathcal{M},\tau)$}{M(Phi)(M,tau)}} \label{positive isometries} 


The objective of this chapter is to provide a complete description of positive surjective isometries on quasi-Banach Marcinkiewicz spaces $M_{\Phi}\left(\mathcal{M},\tau\right)$.

\section{The order isomorphism property of positive surjective isometries}

Let $\cM$ be a semifinite von Neumann algebra equipped with a  semifinite faithful normal trace $\tau$.
{Let
$\left(E(0,\infty),\norm{\cdot}_{E\left(0,\infty\right)}\right)$
be a  quasi-Banach symmetric space.
Define the corresponding quasi-Banach symmetric space \cite{S14} by setting
\begin{align}\label{construction}
E\left(\cM,\tau\right):=\{x\in S\left(\cM,\tau\right):\mu(x)=\mu(f) ~\text{for some}~f\in E\left(0,\infty\right)\},\end{align}
equipped with the quasi-norm
\[\norm{x}_{E\left(\cM,\tau\right)}=\norm{f}_{E\left(0,\infty\right)}.\]
In the special case when   $\cM$ is atomic with all atoms having the same trace,
any quasi-Banach symmetric space $E(\cM,\tau)$ can be obtained by construction \eqref{construction} above. 
Indeed, $E(\cM,\tau)$ corresponds to a quasi-Banach symmetric sequence space  $\ell_E$~\cite{S14}, and 
one  may define 
 a  quasi-Banach symmetric function space $E\left(0,\infty\right)$  associated with $\ell_E$ (and associated with $E(\cM,\tau)$ in the sense of \eqref{construction}), see  \cite[Section 3.6]{LSZ} and \cite[Corollary~14]{S14}.}

In this section, we always assume that  $\cM$ is  a semifinite von Neumann algebra equipped with a  semifinite faithful normal trace $\tau$.
The main purpose of this section is to prove that every positive surjective isometry on $E(\cM,\tau)$ is an order isomorphism, i.e., for any $x\in E\left(\cM,\tau\right)$, $T(x)\ge 0$ if and only if $x\ge0$.


The following result is an immediate consequence of  \cite[Lemma 3.4.1]{Z11}.

\begin{lemma} \label{arise}
    There exists $0<p<1$ such that for any nonzero $\tau$-finite projection $e$, the linear mapping $f_e:E\left(\mathcal{M},\tau\right)\to L_p\left(\mathcal{M},\tau\right)$ defined by
    \[f_e(x):=xe, ~\forall x\in E\left(\mathcal{M},\tau\right),\]
    is bounded.
\end{lemma}

Let $0<p<1$ be obtained  as in Lemma \ref{arise}.
Let $\mathcal{L}$\label{def:L} denote the set of all continuous linear mappings
\[L:E\left(\cM,\tau\right)\to L_p\left(\cM,\tau\right)\]
satisfying $\norm{L}_{op}\le1$, that is,
\[\norm{Lx}_p\le \norm{x}_{E\left(\cM,\tau\right)}, ~\forall x\in E\left(\cM,\tau\right).\]
Note that $\left\{\frac{f_e}{\norm{f_e}_{op}}\right\}_{0<\tau(e)<\infty}$, {defined} in Lemma \ref{arise}, is a subset of $\cL$.
Define the functional $\kappa:E\left(\mathcal{M},\tau\right)\to [0,\infty)$ by setting
\[\kappa(x):=\sup_{L\in \mathcal{L}}\norm{Lx}_p, ~\forall x\in E\left(\mathcal{M},\tau\right).\]
Observe that $\kappa(\lambda x)=|\lambda|\kappa(x)$ for all $\lambda\in \mathbb{C}$ and $x\in E\left(\mathcal{M},\tau\right)$.

In the next lemma, we collect some basic properties of the functional $\kappa$.

\begin{lemma} \label{contraction}
Let $T$ be a  (not necessarily positive) surjective isometry on $E\left(\mathcal{M},\tau\right)$.
Then, the following statements hold:
\begin{enumerate} [label=(\roman*)]
    \item $\kappa(x)=0$ if and only if $x=0$.
    \item $\kappa(Tx)=\kappa(x)$ for all $x\in E\left(\mathcal{M},\tau\right)$.
    \item $\kappa(|x|)=\kappa(x)$ for all $x\in E\left(\mathcal{M},\tau\right)$.
    \item $|\kappa(x)^p-\kappa(y)^p|\le\left(\kappa(x-y)\right)^p\le \norm{x-y}_{E\left(\cM,\tau\right)}^p$.
\end{enumerate}
\end{lemma}
\begin{proof}
~
    \begin{enumerate} [label=(\roman*)]
        \item Assume that $\kappa(x)=0$. 
        It follows that $ \| x e\|_{p} =0$ for any   $\tau$-finite projection $e$, i.e., $xe =0$. Hence, $x=0$.


        \item Note that if $L\in \cL$, then 
        $L\circ T\in \cL$.
        This implies that
        \begin{align} \label{isomeT}
        \kappa(Tx)=\sup_{L\in \cL}\norm{(L\circ T)x}_p\le\sup_{L\in \cL}\norm{Lx}_p=\kappa(x).
        \end{align}
        Arguing 
        similarly, we have 
        %
        \[\kappa(x)\le \kappa(Tx).\]
        This together with \eqref{isomeT} yields that
        $\kappa(T(x))=\kappa(x)$ for all $ x\in E\left(\cM,\tau\right)$.
         \item The proof is  similar to that of (ii).
        \item Since
        $\kappa(x-y)=\sup_{L\in\cL}\norm{L(x-y)}_p\le \norm{x-y}_{E\left(\cM,\tau\right)},$
        it follows that
        \[\left(\kappa(x-y)\right)^p\le \norm{x-y}_{E\left(\cM,\tau\right)}^p.\]

        For any $L\in\cL$, 
        we obtain that
        \begin{eqnarray*}
        \norm{Lx}_p^p&=&\norm{Ly+L(x-y)}_p^p\\
&\stackrel{\mbox{\tiny \cite[Prop. 4.6 (i)]{FK86}}}{\le}& \norm{Ly}_p^p+\norm{L(x-y)}_p^p\\
&\le &\kappa(y)^p+\left(\kappa(x-y)\right)^p.
\end{eqnarray*}
        Taking the supremum over $L\in \cL$ shows that
        \[\kappa(x)^p-\kappa(y)^p\le\left(\kappa(x-y)\right)^p.\]
        Similarly, 
        \[\kappa(y)^p-\kappa(x)^p\le\left(\kappa(y-x)\right)^p=\left(\kappa(x-y)\right)^p.\]
        In particular,
        $|\kappa(x)^p-\kappa(y)^p|\le \left(\kappa(x-y)\right)^p.$
        This completes the proof.
    \end{enumerate} 
\end{proof}

Let $\Omega$ be an open subset of $\mathbb{C}$.
Recall that a function $f:\Omega\to E\left(\cM,\tau\right) $ is called {\itshape analytic} if for each $\lambda_0\in \Omega$, there exist a number $R>0$ and a sequence $\{a_j\}_{j=0}^\infty\subset E\left(\cM,\tau\right)$ and a power series expansion
\[f(\lambda)=\sum_{j=0}^\infty (\lambda-\lambda_0)^ja_j, ~\forall |\lambda-\lambda_0|<R,\]
where the series converges in $\norm{\cdot}_{E\left(\cM,\tau\right)}$ (see, e.g., \cite[p. 371]{BC90}, \cite[p. 538]{K88}).
A function $g:\Omega\to [-\infty,\infty)$ is called {\itshape subharmonic}  if it satisfies the following two conditions \cite[Definition 2.2.1]{R95}:
\begin{enumerate} [label=(\roman*)]
    \item $g$ is upper semicontinuous {\footnote{Let $X$ be a metric space.
    A function $u:X\to [-\infty,\infty)$ is called {\itshape upper semicontinuous} if for any $c\in \mathbb{R}$, the set $\{x\in X:u(x)<c\}$ is open \cite[p. 36]{K91}.}};
    \item given $w\in \Omega$, there exists $R>0$ such that
    \[g(w)\le \frac1{2\pi}\int_0^{2\pi}g(w+re^{it})dt, ~\forall 0\le r<R.\]
\end{enumerate}

For analytic functions with values in a complex Banach space, every
positive power of the norm is subharmonic; see
\cite[p.~882]{HaagerupPisier1989}.
For noncommutative $L_p$-spaces with $0<p<1$, the corresponding
submean inequality follows from the uniform Hardy convexity
established by Xu \cite{Xu1990}; see also
\cite[Theorem~5.7]{PisierXu2003}.
We record the resulting subharmonicity property and its consequence
for bounded analytic functions on a strip.
Here, analyticity for quasi-Banach-space-valued functions is understood
in the sense of locally convergent power series.

\begin{lemma} \label{analytic}
    Assume that $0<p<1$.
    Let $f$ be an analytic $L_p\left(\cM,\tau\right)$-valued function defined on $\mathbb{C}$.
    Then,
    \[\lambda\mapsto \norm{f(\lambda)}_p^p\]
    is subharmonic.
    Moreover, if $f$ is bounded on the strip $\{0<{\rm Re}\lambda<1\}$ and continuous on the closed strip $\{0\le{\rm Re}\lambda\le1\}$ and 
    \[\norm{f(it)}_p\le K_0, ~\norm{f(1+it)}_p\le K_0\]
    for some constant $K_0>0$,
    then
    \[\norm{f(\lambda)}_p\le K_0, ~\forall \lambda\in \{0<{\rm Re}\lambda<1\}.\]
\end{lemma}
\begin{proof}

Fix $\lambda_0\in\mathbb C$. 
Since $f$ is analytic, it follows that there exist a number $R>0$ and $\{a_j\}_{j=0}^\infty\subset L_p\left(\cM,\tau\right)$ such that
\[
f(\lambda_0+w)=\sum_{j=0}^{\infty}w^j a_j,
~|w|<R.
\]
Fix $0\le r<R$. By \cite[p. 371]{BC90}, we have that
\[
P_N(z):=\sum_{j=0}^{N}r^jz^j a_j,~ |z|\le1.
\]
converges to $f(\lambda_0+rz)$ uniformly on 
$\{z:|z|\le 1\}$
whenever $N\to \infty$, 
i.e.,  
\begin{align} \label{sup c}
    c_N:=\sup_{|z|\le 1}\norm{P_N(z)-f(\lambda_0+rz)}_p\to 0 ~\text{as}~ N\to \infty.
\end{align}
Moreover, $P_N(0)=a_0=f(\lambda_0)$.
Since the case when $P_N(0)=0$ is trivial, without loss of generality, we
assume that $P_N(0)\ne 0$. 
By \cite[Theorem~5.7]{PisierXu2003} and the definition in \cite[p. 1483]{PisierXu2003}, we have that
$
    \frac1{2\pi}\int_0^{2\pi} \norm{\frac{P_N(e^{it})}{\norm{P_N(0)}_p}}_p^pdt\ge1,
$
that is,
\begin{align} \label{f(lambda 0)}
\norm{f(\lambda_0)}_p^p=\norm{P_N(0)}_p^p\le \frac1{2\pi}\int_0^{2\pi} \norm{P_N(e^{it})}_p^pdt.
\end{align}
On the other hand, we have that
\begin{align*}
    \left|\norm{P_N(e^{it})}_p^p-\norm{f\left(\lambda_0+re^{it}\right)}_p^p\right|\stackrel{\mbox{\tiny  \cite[Prop. 4.6(i)]{FK86}}}{\le} \norm{P_N(e^{it})-f\left(\lambda_0+re^{it}\right)}_p^p.
\end{align*}
This implies that
\begin{align*}
&\quad \left|\frac1{2\pi}\int_0^{2\pi}\norm{P_N(e^{it})}_p^pdt-\frac1{2\pi}\int_0^{2\pi}\norm{f\left(\lambda_0+re^{it}\right)}_p^pdt\right|\\
&\le\frac1{2\pi}\int_0^{2\pi}\norm{P_N(e^{it})-f\left(\lambda_0+re^{it}\right)}_p^pdt\stackrel{\eqref{sup c}}{\le}c_N^p\to 0 ~\text{as} ~ N\to \infty.
\end{align*}
This together with \eqref{f(lambda 0)} yields that
\[\norm{f(\lambda_0)}_p^p\le \frac1{2\pi}\int_0^{2\pi}\norm{f\left(\lambda_0+re^{it}\right)}_p^pdt.\]
Noting that \(\lambda\mapsto\|f(\lambda)\|_p^p\) is upper semicontinuous, the first statement is proved.

The second statement follows immediately from \cite[Corollary 2.3.6]{R95}.
\end{proof}

The following result plays an important role in proving the order isomorphism property of $T$, which is analogous to  \cite[Lemma 2.1]{CMS} and \cite[Lemma 4.2]{HSZ}.


\begin{lemma} \label{longlong}
    Let $0\le a\in E\left(\mathcal{M},\tau\right)$ and $h\in E\left(\mathcal{M},\tau\right)_h$.
    If $-a\le h\le a$, then $\kappa(h)\le \kappa(a)$.
    In particular, if $0\le h\le a$,
    then $\kappa(h)\le \kappa(a)$.
\end{lemma}
\begin{proof}
The case $a=0$ is immediate. Assume that $a\ne 0$.
We divide the proof into several steps.

{\bf Step $1$}.
We first construct $c\in \cM_h$ with $\norm{c}_\cM\le 1$, $r(c)\le r(a)$ and represent $h$ as $a^{1/2}ca^{1/2}$.
    Note that $0\le a+h\le 2a$.
    Applying \cite[Proposition 2.2.24]{DPS} to $x=(2a)^{1/2}$ and $y=(a+h)^{1/2}$, we obtain that there exists $b\in \cM$ such that
    \begin{align} \label{bx=y}
    \norm{b}_\cM\le1, ~b(2a)^{1/2}=(a+h)^{1/2}, 
    ~r(b)\le r\left((2a)^{1/2}\right)=r(a).
    \end{align}
    This implies that 
    $b^*b\le \mathbf{1}, ~b=br(a)$
    and so,
    \begin{align} \label{b*b}
    0\le b^*b=r(a)b^*br(a)\le r(a).
    \end{align}
    Define
    \begin{align} \label{def of c}
    c:=2b^*b-r(a)\in \cM_h.
    \end{align}
    Then, by \eqref{b*b}, we have that
    $-r(a)\le c\le r(a).$
    In particular,
    \begin{align} \label{c in M}
    \norm{c}_\cM\le1.
    \end{align}
    Since $b=br(a)$, it follows that $c=r(a)cr(a)$.
    Moreover,
\begin{eqnarray} \label{=h}
        a^{1/2}ca^{1/2}&\stackrel{\eqref{def of c}}{=}&2a^{1/2}b^*ba^{1/2}-a^{1/2}r(a)a^{1/2}\nonumber\\
        &=&(2a)^{1/2}b^*b(2a)^{1/2}-a\nonumber\\
        &\stackrel{\eqref{bx=y}}{=}&(a+h)^{1/2}(a+h)^{1/2}-a\nonumber\\
        &=& h.
    \end{eqnarray}
{\bf Step $2$}.
    We next  construct operators $c_j$ from $c$ using the spectral projections of $a$.
    Let $e^a$ denote the spectral measure of $a$.
    For any $k\in \mathbb{Z}$, let
    \[e_k:=e^a\left[2^k,2^{k+1}\right).\]
    Let
    \begin{align} \label{the def of x}
    x:= \sum_{k\in \mathbb{Z}}2^{k}e_k, ~y:= \sum_{k\in\mathbb{Z}}2^{-k}ae_k+\left(1-r(a)\right).
    \end{align}
    Note that 
    \[0\le x\in S\left(\cM,\tau\right), ~a=xy=yx.\]
   Since  $\mathbf{1}\le y\le 2\mathbf{1}, ~ \frac12\mathbf{1}\le  y^{-1}\le \mathbf{1},$
it follows that   $x=ay^{-1}\in E\left(\cM,\tau\right)_+$.
    For any $j\in\mathbb{Z}$, let
    \begin{align} \label{the def of cj}
    c_j:=\sum_{k\in\mathbb{Z}} e_{k+j}ce_k,
    \end{align}
    where the series converges in the so-topology.
    Since $\norm{c}_\cM \le 1$  and $\{e_{k+j}\}_{k\in \mathbb{Z}}$ is a pairwise orthogonal sequence, 
    it follows 
    that 
    \begin{align} \label{cj bounded}
    \norm{c_j}_\cM\le 1.
    \end{align}
    Observe that
    \[x^{1/2}c_j=\sum_{k\in \mathbb{Z}}2^{\frac{k+j}2}e_{k+j}ce_k=2^{\frac j2}\sum_{k\in \mathbb{Z}}2^{\frac k 2}e_{k+j}ce_k\]
    and 
    \[c_jx^{1/2}=\sum_{k\in\mathbb{Z}}2^{\frac k 2}e_{k+j}ce_k.\]
    In particular,
    \begin{align} \label{xcj}
    x^{1/2}c_j=2^{\frac j 2}c_jx^{1/2}.
    \end{align}
    Arguing similarly, we have 
    \begin{align} \label{xcj2}
        xc_j=2^{j}c_jx.
    \end{align}

    {\bf Step $3$}.
    Define
    \begin{align} \label{the def of hj}
    h_j:=a^{1/2}c_ja^{1/2}.
    \end{align}
    Below, we will prove that $\sum_{j\in \mathbb{Z}}h_j$ exists in $E\left(\cM,\tau\right)$ and $\sum_{j\in \mathbb{Z}}h_j=h$.
    Since $a^{1/2}=x^{1/2}y^{1/2}=y^{1/2}x^{1/2}$, it follows  that if $j\ge0$, then we have
    \begin{eqnarray*}
    h_j=a^{1/2}c_ja^{1/2}&=&x^{1/2}y^{1/2}c_jx^{1/2}y^{1/2}\nonumber\\
    &\stackrel{\eqref{xcj}}{=}&2^{-\frac j2}x^{1/2}y^{1/2}x^{1/2}c_jy^{1/2}\\
    &=&2^{-\frac j2}xy^{1/2}c_jy^{1/2}\\
    &\stackrel{xy^{1/2}=ay^{-1/2}}{=}&2^{-\frac j2}ay^{-1/2}c_jy^{1/2};
    \end{eqnarray*}
    and  if $j<0$, then we have 
    \begin{eqnarray*}
    h_j=a^{1/2}c_ja^{1/2}&=&y^{1/2}x^{1/2}c_jy^{1/2}x^{1/2}\nonumber\\
    &\stackrel{\eqref{xcj}}{=}&2^{\frac j2}y^{1/2}c_jx^{1/2}y^{1/2}x^{1/2}\\
    &=&2^{\frac j2}y^{1/2}c_jxy^{1/2}\\
    &\stackrel{xy^{1/2}=ay^{-1/2}}{=}&2^{\frac j2}y^{1/2}c_jay^{-1/2}.
    \end{eqnarray*}
    That is,
    \begin{align} \label{hj}
    h_j=
    \begin{cases}
    2^{-\frac j2}ay^{-1/2}c_jy^{1/2}, & j\ge0, \\
    2^{\frac j2}y^{1/2}c_jay^{-1/2}, & j<0. \\
    \end{cases}
    \end{align}
    By \eqref{cj bounded} and \eqref{hj}, we have $h_j\in E\left(\cM,\tau\right)$ and 
    \begin{align} \label{hj estimate}
    \norm{h_j}_{E\left(\cM,\tau\right)}\le \sqrt2\cdot2^{-\frac{|j|}2}\norm{a}_{E\left(\cM,\tau\right)}.
    \end{align}
    Hence, $\left\{\sum_{|j|\le N}h_j\right\}_{N\in \mathbb{N}}$ is a Cauchy sequence in $E\left(\cM,\tau\right)$.
    Since $E\left(\cM,\tau\right)$ is complete, it follows that
    $\sum_{j\in \mathbb{Z}}h_j$ exists in $E\left(\cM,\tau\right)$.
    By \eqref{the def of cj}, we have that 
    \begin{align} \label{ejckel}
    e_jc_ke_l=
        \begin{cases}
           e_jc e_l, & k=j-l, \\
           0, & k\ne j-l.
        \end{cases}
    \end{align}
    Since $a$  commutes with $e_j,e_l$, it follows from  \eqref{the def of hj} and \eqref{=h}, in combination with \eqref{ejckel} that
    \begin{align*}
    e_jh_ke_l=
        \begin{cases}
            e_jhe_l, & k=j-l, \\
            0, & k\ne j-l.
        \end{cases}
    \end{align*}
    This implies that
    $e_j\sum_{k\in \mathbb{Z}}h_ke_l=e_jhe_l,$
    and so, 
    \[e_j\left(\sum_{k\in \mathbb{Z}}h_k-h\right)e_l=0, ~\forall j,l\in \mathbb{Z}.\]
    {
    Note that $p_n:=\sum_{|k|\le n}e_k\uparrow r(a)$ as $n\to \infty$
    and $\sum_{k\in \mathbb{Z}}h_k-h=r(a)\left(\sum_{k\in \mathbb{Z}}h_k-h\right)r(a)$.
    Applying \cite[Lemma 16]{BCSZ18} to $p_n$ in $r(a)\cM r(a)$, we have that
    \begin{align} \label{sum hj=h}
         \sum_{j\in \mathbb{Z}}h_j=h.
    \end{align}}

    {\bf Step $4$}.
    For any $N\ge0$, define
    \begin{align} \label{wj,N}
    w_{j,N}:=
        \begin{cases}
            1-\frac{|j|}{N+1}, & |j|\le N, \\
            0, & |j|>N. \\
        \end{cases}
    \end{align}
and
    \begin{align} \label{h(N)}
    h^{(N)}:=\sum_{|j|\le N}w_{j,N}h_j.
    \end{align}
Using \eqref{hj estimate} {and \eqref{sum hj=h}, by standard argument}, we have   $\norm{h^{(N)}-h}_{E\left(\cM,\tau\right)}\to 0$ as $N\to \infty $,     that is, $h=\lim_{N\to \infty}h^{(N)}$.
     Indeed, note that 
    \begin{align} \label{wj,N-1}
    |w_{j,N}-1|=\min\left\{\frac{|j|}{N+1},1\right\}\le \frac{|j|}{N+1}.
    \end{align}
    Since $E\left(\cM,
    \tau\right)$ is complete, it follows from \eqref{hj estimate} and \eqref{wj,N-1} that
 $ \lim_{N\to\infty}h^{(N)}$ exists in $E\left(\mathcal{M},\tau\right)$.
    By
    \eqref{wj,N} and \eqref{h(N)}, we have that for any $N\le M\in \mathbb{Z}$,
    \begin{align} \label{h(N)-h}
    h^{(N)}-\sum_{|j|\le M}h_j=\sum_{|j|\le M} \left(w_{j,N}-1\right)h_j.
    \end{align}
    By  \eqref{sum hj=h}, we have that
    \begin{align} \label{limit to}
    &\quad\norm{\left(h^{(N)}-\sum_{|j|\le M}h_j\right)-\left(h^{(N)}-h\right)}_{E\left(\cM,\tau\right)}\nonumber\\
    &\le\norm{\sum_{|j|\le M}h_j-h}_{E\left(\cM,\tau\right)}\to0 ~\text{as} ~ M \to \infty.
    \end{align}
 By the Aoki--Rolewicz theorem \cite[Theorem 1.3]{KPR}, there exist $p_E>0$ and a $p_E$-subadditive quasi-norm $\norm{\cdot}_{(p_E)}$ 
such that $ C_{1}\norm{\cdot}_{E\left(\cM,\tau\right)}\le \norm{\cdot}_{(p_E)} \le \norm{\cdot}_{E\left(\cM,\tau\right)} $ for some positive constant $C_1$. Therefore, 
    \begin{eqnarray*}
        \norm{h^{(N)}-h}_{E\left(\cM,\tau\right)}^{p_E}&\le &\frac1{C_{1}^{p_E}}\norm{h^{(N)}-h}_{(p_E)}^{p_E}\\
        &\stackrel{\eqref{limit to}}{=}&\frac{1}{C_{1}^{p_E}}\lim_{M\to\infty}\norm{h^{(N)}-\sum_{|j|\le M}h_j}_{(p_E)}^{p_E}\\
        &\stackrel{\eqref{h(N)-h}}{=}&\frac{1}{C_1^{p_E}}\lim_{N\le M\to\infty}\norm{\sum_{|j|\le M} \left(w_{j,N}-1\right)h_j}_{(p_E)}^{p_E}\\
        & \le&\frac1{C_{1}^{p_E}}\lim_{N\le M\to\infty}\sum_{|j|\le M} \left|w_{j,N}-1\right|^{p_E}\norm{h_j}_{E\left(\cM,\tau\right)}^{p_E}\\
        &=&\frac1{C_{1}^{p_E}}\sum_{j\in \mathbb{Z}} \left|w_{j,N}-1\right|^{p_E}\norm{h_j}_{E\left(\cM,\tau\right)}^{p_E}\\
        &\stackrel{\substack{\eqref{wj,N-1}\\\eqref{hj estimate}}}{\le}&\frac1{C_{1}^{p_E}}\sum_{j\in \mathbb{Z}}\left(\frac{|j|}{N+1}\sqrt2\cdot2^{-\frac{|j|}2}\norm{a}_{E\left(\cM,\tau\right)}\right)^{p_E}\\
        &=&\frac1{C_{1}^{p_E}}\left(\frac{1}{N+1}\sqrt2\cdot\norm{a}_{E\left(\cM,\tau\right)}\right)^{p_E}\sum_{j\in \mathbb{Z}}\left(|j|2^{-\frac{|j|}{2}}\right)^{p_E}<\infty.
    \end{eqnarray*}
Taking $n\to\infty$, we have that $\norm{h^{(N)}-h}_{E\left(\cM,\tau\right)}\to 0$ as $N\to \infty $.

   {\bf Step $5$}.
    For any $N\ge0$, define
    \begin{align} \label{c(N)}
    c^{(N)}:=\sum_{|j|\le N} w_{j,N}c_j.
    \end{align}
    Below, we will prove that $\norm{c^{(N)}}_\cM\le1$.
    For any $0\le r\le N$,
    let
    \begin{align} \label{the def of Pr(c)}
    q_{r,k}:=\sum_{l=0}^Ne_{r+(N+1)k+l}, ~P_r(c)=\sum_{k\in \mathbb{Z}}q_{r,k}cq_{r,k}.
    \end{align}
    Note that $\{q_{r,k}\}_{k\in\mathbb{Z}}$ is a mutually orthogonal sequence.
    This implies that $\{q_{r,k}cq_{r,k}\}_{k\in\mathbb{Z}}$ is a mutually orthogonal sequence.
    By \cite[Proposition 5.5.4]{C12}, we have that
    \begin{align} \label{the norm of Pr(c)}
\norm{P_r(c)}_\cM=\sup_{k\in\mathbb{Z}}\norm{q_{r,k}cq_{r,k}}_\mathcal{M}\le\norm{c}_\cM\stackrel{\eqref{c in M}}{\le}1.
\end{align}
    Observe that
    \begin{align} \label{mathbb Z}
    \mathbb Z=
\bigcup_{r=0}^{N}\left(r+(N+1)\mathbb Z\right).
    \end{align}
    This implies that
    \begin{eqnarray} \label{multiply}
        \sum_{r=0}^NP_r(c)&\stackrel{\eqref{the def of Pr(c)}}{=}&\sum_{r=0}^N\sum_{k\in \mathbb{Z}}\left(\sum_{l=0}^Ne_{r+(N+1)k+l}\right)c\left(\sum_{j=0}^Ne_{r+(N+1)k+j}\right)\nonumber\\
        &=&\sum_{l,j=0}^N\sum_{r=0}^N\sum_{k\in \mathbb{Z}}e_{r+(N+1)k+l}ce_{r+(N+1)k+j}\nonumber\\
        &\stackrel{\substack{\eqref{mathbb Z}\\i=r+(N+1)k}}{=}&\sum_{l,j=0}^N\sum_{i\in \mathbb{Z}}e_{i+l}ce_{i+j}\nonumber\\
        &\stackrel{s=i+j}{=}&\sum_{l,j=0}^N\sum_{s\in \mathbb{Z}}e_{s-j+l}ce_{s}\nonumber\\
        &\stackrel{\eqref{the def of cj}}{=}&\sum_{l,j=0}^Nc_{-j+l}\nonumber\\
        &=&\sum_{|j|\le N}(N+1-|j|)c_j.
    \end{eqnarray}
    Multiplying both sides of \eqref{multiply} by $\frac1{N+1}$, we obtain that
    $c^{(N)}\stackrel{\eqref{c(N)}}{=}\frac1{N+1}\sum_{r=0}^NP_r(c).$
    Then,
    \[\norm{c^{(N)}}_\cM=\norm{\frac1{N+1}\sum_{r=0}^NP_r(c)}_\cM \stackrel{\eqref{the norm of Pr(c)}}{\le}1.\]
    
    {\bf Step $6$}.
    Finally, we apply Lemma \ref{analytic} with  
   suitable analytic functions   to obtain $\kappa\left(h^{(N)}\right)\le \kappa(a)$, which implies $\kappa(h)\le \kappa(a)$ by the convergence established in Step 4.
    
    For any $\lambda\in \mathbb{C}$, define
    \begin{align} \label{FN}
    {F_N(\lambda):=\sum_{|j|\le N}w_{j,N}2^{j\lambda}y^\lambda c_j y^{-\lambda}a, ~B_N(\lambda):=\sum_{|j|\le N}w_{j,N}2^{j\lambda}y^\lambda c_j y^{-\lambda}.}
    \end{align}
    By $\mathbf{1}\le y\le 2 \mathbf{1}$, we have that
    $\log y$ is bounded and $y^\lambda=e^{\lambda \log y}, ~y^{-\lambda}=e^{-\lambda \log y}$ are analytic on $\mathbb{C}$.
    This implies that $B_N$ is analytic on $\mathbb{C}$ and so,
    $F_N=B_N a$ is analytic on $\mathbb{C}$.
    {Note that
    \begin{align} \label{1/2}
    F_N\left(\frac 12\right)\stackrel{\substack{\eqref{hj}\\\eqref{FN}}}{=}\sum_{|j|\le N}w_{j,N}h_j\stackrel{\eqref{h(N)}}{=}h^{(N)}.
    \end{align}}
    For any $t\in \mathbb{R}$, let
    \[u_t:=a^{it}r(a)+\left(\mathbf{1}-r(a)\right).\]
    Observe that $u_t$ is a unitary operator in $\mathcal{M}$.
    Since
    \[
        x^{it}c_j x^{-it}\stackrel{\substack{\eqref{the def of x}\\\eqref{the def of cj}}}{=}\sum_{k\in \mathbb{Z}} 2^{i(k+j)t} e_{k+j}ce_k 2^{-ikt}=2^{ijt}c_j,
    \]
    it follows from $u_tr(a)=a^{it}r(a)=a^{it}=x^{it}y^{it}$ that
    \begin{align} \label{utcj}
    u_tc_ju_t^*\stackrel{c_j=r(a)c_jr(a)}{=}x^{it}y^{it}c_j y^{-it}x^{-it}=2^{ijt}y^{it}c_j y^{-it}.
    \end{align}
    This implies that
    \begin{eqnarray} \label{it}
        F_N(it)&\stackrel{\eqref{FN}}{=}&\sum_{|j|\le N}w_{j,N}2^{ijt}y^{it} c_j y^{-it}a\nonumber\\
        &=&\sum_{|j|\le N}w_{j,N} u_t c_j u_t^*a\nonumber\\
        &\stackrel{\eqref{c(N)}}{=}&u_tc^{(N)}u_t^*a.
    \end{eqnarray}
    Note that
\begin{eqnarray} \label{1+it}
        F_N\left(1+it\right)&\stackrel{\eqref{FN}}{=}&\sum_{|j|\le N} w_{j,N}2^{j(1+it)}y^{1+it} c_j y^{-\left(1+it\right)}a\nonumber\\
        &=&\sum_{|j|\le N}2^{ijt} w_{j,N}y^{it}\left(2^jyc_jy^{-1}a\right)y^{-it}\nonumber\\
        &\stackrel{a=yx}{=}&\sum_{|j|\le N} 2^{ijt} w_{j,N}y^{it}\left(2^jyc_jx\right)y^{-it}\nonumber\\
        &\stackrel{\eqref{xcj2}}{=}&\sum_{|j|\le N}2^{ijt} w_{j,N}y^{it} 2^jy2^{-j}(xc_j)y^{-it}\nonumber\\
        &\stackrel{a=yx}{=}&\sum_{|j|\le N} 2^{ijt} w_{j,N}y^{it} ac_jy^{-it}\nonumber\\
        &\stackrel{\eqref{utcj}}{=}&\sum_{|j|\le N} w_{j,N} au_tc_ju_t^*\nonumber\\
        &\stackrel{\eqref{c(N)}}{=}& u_tac^{(N)}u_t^*.
    \end{eqnarray}
    Note that 
    \begin{eqnarray*}
    \norm{L F_N(\lambda)}_p&\le& \norm{F_N(\lambda)}_{E\left(\mathcal{M},\tau\right)}\\
    &\stackrel{\eqref{FN}}{\le}& \norm{B_N(\lambda)}_{\cM} 
    \norm{a}_{E\left(\cM,\tau\right)}\\
    &\stackrel{\eqref{wj,N}}{\le}& \norm{a}_{E\left(\cM,\tau\right)}\left(\sum_{|j|\le N}\left|2^{j\lambda}\right|\norm{y^\lambda}_\cM\norm{c_j}_\cM\norm{y^{-\lambda}}_\cM\right). 
    \end{eqnarray*}
    Then,
    $L\circ F_N$ is bounded on the strip $\{0<{\rm Re}\lambda<1\}$.
    Since $F_N$ is analytic on $\mathbb{C}$, 
    it follows that $L\circ F_N$ is an  analytic $L_p\left(\cM,\tau\right)$-valued function on the strip $\{0<{\rm Re}\lambda<1\}$ and continuous on the closed strip $\{0\le {\rm Re}\lambda \le 1\}$.
    Since $u_t$ is unitary, $\norm{c^{(N)}}_\mathcal{M}\le1$, $x\mapsto Lu_tc^{(N)}xu_t^*$ and $x\mapsto Lu_txc^{(N)}u_t^*\in \mathcal{L}$ {($\mathbb{\cL}$ is defined in p.~\pageref{def:L})}, it follows that
    \[\norm{L F_N(it)}_p\stackrel{\eqref{it}}{=}\norm{Lu_tc^{(N)}u_t^*a}_p\le \sup_{L\in \cL}\norm{La}_p=\kappa(a), \]
    and
    \[\norm{L F_N(1+it)}_p\stackrel{\eqref{1+it}}{=}\norm{Lu_tac^{(N)}u_t^*}_p\le \sup_{L\in \cL}\norm{La}_p=\kappa(a).\]
    By Lemma \ref{analytic}, we have that
    \[\norm{Lh^{(N)}}_p\stackrel{\eqref{1/2}}{=}\norm{LF_N\left(\frac12\right)}_p\le \kappa(a).\]
    Taking the supremum over $L\in \cL$ yields that
    \begin{align} \label{xixi1}
    \kappa\left(h^{(N)}\right)\le \kappa(a).
    \end{align}
    Taking $N\to \infty$,  we have that
    \begin{align*} 
    |\kappa(h^{(N)})^p-\kappa(h)^p|\stackrel{\mbox{\tiny Lem.~\ref{contraction}}}{\le} \norm{h^{(N)}-h}_{E\left(\cM,\tau\right)}^p\stackrel{\rm \tiny Step~4}{\longrightarrow} 0.
    \end{align*}
    This together with \eqref{xixi1} yields  that
    $\kappa(h)\le \kappa(a)$.
\end{proof}

Now we apply Lemmas \ref{contraction} and \ref{longlong} to prove that every positive surjective isometry on $E\left(\cM,\tau\right)$ is an order isomorphism, which extends the main results of \cite{Abr88} as well as key technical results in \cite{CMS} (see Lemma 3.2 there) and in \cite{HSZ} (see Lemma 5.1 there).

\begin{prop} \label{T(x)ge0}
Let $T:E\left(\mathcal{M},\tau\right)\to E\left(\mathcal{M},\tau\right)$ be a positive surjective isometry.
Then, $T^{-1}$ is positive.
In particular, $T\left(E\left(\mathcal{M},\tau\right)_+\right)=E\left(\mathcal{M},\tau\right)_+$.
\end{prop}
\begin{proof}
    Let $x\in E\left(\cM,\tau\right)$ be such that $T(x)\ge0$.
    Since $T$ is positive, it follows that $x$ is self-adjoint.
    Let $x=x_+-x_-$ and $a:=T(x_+),~b:=T(x_-)$.
    Then, $a-b=T(x)\ge0$.
    
    Below, we will prove that $x_-=0$.
    For any $n\ge0$, we have that
    \[-\left(a+nb\right)\le a-\left(n+1\right)b\le a+nb.\]
    Since $x_+x_-=0$, it follows that
    \begin{align*} 
        |x_+-(n+1)x_-|=x_++(n+1)x_-.
    \end{align*}
    This implies that for all $n\ge0$,
    \begin{eqnarray*}
        \kappa\left(x_++(n+1)x_-\right)&\stackrel{\tiny \mbox{Lem. \ref{contraction}(iii)}}{=}&\kappa\left(x_+-(n+1)x_-\right)\\ 
        &\stackrel{\tiny \mbox{Lem. \ref{contraction}(ii)}}{=}&\kappa\left(a-(n+1)b\right)\\
        &\stackrel{\mbox{\tiny Lem.~\ref{longlong}}}{\le}&\kappa(a+nb)\\
        &\stackrel{\tiny \mbox{Lem. \ref{contraction}(ii)}}{=}&\kappa(x_++nx_-).
    \end{eqnarray*}
    In particular, 
    \begin{align} \label{combin1}
    \kappa(x_++nx_-)\le \kappa(x_+), ~\forall n\ge0.
    \end{align}
    Since $0\le nx_-\le x_++nx_-$, it follows from Lemma \ref{longlong} that
    \begin{align} \label{combin2}
    n\kappa(x_-)=\kappa(nx_-)\le \kappa(x_++nx_-).
    \end{align}
    A combination of \eqref{combin1} and \eqref{combin2} yields that
    \[n\kappa(x_-)\le\kappa(x_+)\le \norm{x}_{E\left(\cM,\tau\right)} .\]
  Since $n$ is arbitrarily taken, it follows that  $\kappa(x_-)=0$, 
  which together with Lemma~\ref{contraction}(i) yields that $x_-=0$,
 i.e.,
    $x=x_+\ge0.$
In other words, $T^{-1}$ is positive.
\end{proof}

The following lemma is often used to prove the normality of $T$, i.e., for any $0\le x_i \uparrow x\in M_\Phi(\cM,\tau)$, we have $T(x) = \sup_i T(x_i)$. 

\begin{prop} \label{order isom}
    Let $\cM$ be a semifinite von Neumann algebra equipped with a semifinite faithful normal trace $\tau$.
   Let $T:E\left(\mathcal{M},\tau\right)\to E\left(\mathcal{M},\tau\right)$ be a positive surjective isometry.
    Then, 
    $T$ is normal. 
\end{prop}
\begin{proof}
    
    By Proposition \ref{T(x)ge0}, we have that $T^{-1}$ is positive.
    The normality follows from the argument used in the proof of \cite[Proposition 5.2]{HSZ}. 
\end{proof}

\section{The  disjointness-preserving property  of positive isometries} 
Assume that $\Phi:[0,\infty) \to[0,\infty)$ is admissible. 
Let $T$ be a positive surjective isometry on quasi-Banach Marcinkiewicz spaces $M_{\Phi}\left(\mathcal{M},\tau\right)$.
We begin this section with  studying  the disjointness-preserving property of $T$, which allows us to characterize $T$ in the next section. 



\begin{lemma} \label{Marcin1}
    Assume that $\mathcal{M}$ is non-atomic.
Let $\Phi:[0,\infty) \to[0,\infty)$ be admissible
    satisfying $\frac1\Phi\in \left(L_1+L_\infty\right)\left(0,\infty\right)$ and $\Phi\left(\infty\right):=\displaystyle\lim_{t \to \infty} \Phi\left(t\right)=\infty$. 
    Let $0\le x,y\in M_{\Phi}\left(\mathcal{M},\tau\right)$ be such that $\mu\left(t;x\right)=\frac{1}{\Phi}\chi_{_{\left(0,s\right)}}\left(t\right)$ and $\mu\left(t;x+y\right)=\frac{1}{\Phi}\chi_{_{\left(0,r\right)}}\left(t\right)$ for some constants $r\ge s>0$. 
   For any   positive surjective isometry $T:M_{\Phi}\left(\mathcal{M},\tau\right)\to M_{\Phi}\left(\mathcal{M},\tau\right)$, 
we have   $T\left(x\right)\perp T\left(y\right)$. In particular, $x\perp y$.

   As a consequence, if  $ 0\le x,y \in M_\Phi\left(\mathcal{M},\tau\right)$
   have $\tau$-finite supports and are such that $x\perp y$  and $y$ is uniformly bounded (that is, $y\in \mathcal{M}$), then we have $T\left(x\right)\perp T\left(y\right)$. 
\end{lemma}
\begin{proof}
By Theorem \ref{posi non-atomic}, we obtain $x,x+y$ are extreme points of the positive unit ball.
Applying Proposition \ref{T(x)ge0} to $T$, we have that $T^{-1}$ is positive.
This implies that $T(x),T(x+y)$ are extreme points. 
    Assume that $T\left(x\right)\not\perp T\left(y\right)$. 
    We have $r\left(T\left(x\right)\right)T\left(y\right)r\left(T\left(x\right)\right)\ne 0$.
    {Indeed, if $r\left(T\left(x\right)\right)T\left(y\right)r\left(T\left(x\right)\right)= 0$, then $T(y)^{1/2}r\left(T\left(x\right)\right)=0$ and so, $r\left(T\left(x\right)\right)\le n\left(T(y)^{1/2}\right)=n\left(T(y)\right)$,
    which contradicts with the assumption that $T\left(x\right)\not\perp T\left(y\right)$.}
    Consequently,  there exists 
    $0<a\le\tau(r(T(x)))$
    such that 
    \begin{eqnarray}
        \int_0^a\frac{1}{\Phi\left(t\right)}dt
        &\stackrel{\tiny \mbox{Thm. \ref{posi non-atomic}}}{=}&
        \int_0^a \mu\left(t;T\left(x+y\right)\right)dt \notag\\
        &\ge& \int_0^a \mu\left(t;r\left(T\left(x\right)\right)T\left(x+y\right)r\left(T\left(x\right)\right)\right)dt \notag\\
        &\stackrel{\tiny \mbox{\cite[Prop. 2.2]{CKS}}}{>}&\int_0^a \mu\left(t;r\left(T\left(x\right)\right)T\left(x\right)r\left(T\left(x\right)\right)\right)dt \label{in L1+Linfty} \notag \\
        &\stackrel{\tiny \mbox{Thm. \ref{posi non-atomic}}}{=}&
        \int_0^a\frac{1}{\Phi\left(t\right)}dt, 
\end{eqnarray}
which yields a contradiction.
    Hence, $T\left(x\right)\perp T\left(y\right)$.

    To prove the final assertion, assume that
     $ 0\le x,y \in M_\Phi\left(\cM,\tau\right)$ have $\tau$-finite supports such that $x\perp y$  and $y$ is uniformly bounded, there exists $z_1\ge x$ with $z_1\perp y$ such that
     \begin{align} \label{z1}
     \mu(z_1)=c_1\frac{1}{\Phi}\chi_{_{\left(0,\tau\left(r(x)\right)\right)}}
     \end{align}
     for some $c_1>0$.
     Indeed,    
     by \cite[Lemma 1.3]{CKS92}, there exist  
a  weakly closed commutative $*$-subalgebra $\cN$ of $r(x)\cM r(x)$ containing all spectral projections of $x$ and a
$*$-isomorphism \[J:S\left(0,\tau\left(r(x)\right)\right)\to S(\cN,\tau),\]
which preserves singular value functions.
     By Lemma~\ref{map}, $J^{-1}(x) = \mu( J^{-1}(x)) \circ \phi $ for some measure-preserving map $\phi$. By definition, $\mu( J^{-1}(x)) \le c_1 \frac{1}{\Phi} \chi_{_{\left(0,\tau\left(r(x)\right)\right)}} $ for some $c_1>0$. 
     Then, 
     \[z':=  c_1 \frac{1}{\Phi} \chi_{_{\left(0,\tau\left(r(x)\right)\right)}} \circ \phi \ge 
     \mu( J^{-1}(x)) \circ \phi
     =J^{-1}(x).\]
     Setting $z_1 = J(z')$, we obtain that $z_1\ge x$ and \eqref{z1} is satisfied.

     Since $y$ is uniformly bounded and has $\tau$-finite support, it follows that there exists $z_2\ge y$ with $z_2\perp z_1$ such that
     \[\mu(z_2)=\mu\left(c_2\frac1\Phi\chi_{_{\left(\tau(r(x)),\tau\left(r(x)+r(y)\right)\right)}}\right)\]
     for some $c_1\le c_2<\infty$.
     This together with \eqref{z1} implies that
     \[\mu\left(\frac{c_2}{c_1}z_1+z_2\right)=c_2\frac1\Phi\chi_{_{\left(0,\tau\left(r(x)+r(y)\right)\right)}}.\]
     

    By the first statement, we have $T\left(\frac{c_2}{c_1}z_1\right)T\left(z_2\right)=0$. 
    Since $T$ is positive, it follows that $T\left(x\right)T\left(y\right)=0$.
\end{proof}


\begin{remark}
The  assumption that $\frac1\Phi\in \left(L_1+L_\infty\right)\left(0,\infty\right)$ can be omitted when $\mathcal{M}$ is atomic with all atoms having the same trace 
by 
$\int_0^a \mu(t;x)dt\le \int_0^a\frac1 {l_\Phi(0)}dt<\infty$ for any $0<a<\infty$ and $x\in \operatorname{extr}\left(B_{M_\Phi\left(\cM,\tau\right)}\right)$.


However,  the condition $\frac1\Phi\in \left(L_1+L_\infty\right)\left(0,\infty\right)$ 
  is essential in the proof for the non-atomic case (see \eqref{in L1+Linfty} and the characterization of extreme points of the positive unit ball,  Remark~\ref{positive no extr}). 

   In Proposition \ref{0<p<2} below, we obtain a criterion for two operators in $L_{p,\infty}\left(\mathcal{M},\tau\right)$, $0<p\le 1$, to be orthogonal.
\end{remark}


The following result follows from a similar argument used in the proof of Lemma~\ref{Marcin1}.

\begin{lemma} 
Assume that $\mathcal{M}$ is 
  atomic 
  with all atoms  having trace equal to one. 
Let $\Phi:[0,\infty) \to[0,\infty)$ be admissible satisfying $\Phi\left(\infty\right):=\displaystyle\lim_{t \to \infty} \Phi\left(t\right)=\infty$. 
  Let $0\le x,y\in M_{\Phi}\left(\mathcal{M},\tau\right)$ be such that 
 $\mu\left(n;x\right)=\frac{1}{ l_\Phi(n)},~\forall n\in\mathbb{N},n<\tau\left(r(x)\right)$ and $\mu\left(n;x+y\right)=\frac{1}{ l_\Phi(n)},~\forall n\in\mathbb{N},n<\tau(r\left(x+y\right))$ with $0<\tau(r(x))\le\tau(r\left(x+y\right))$.
 For any positive surjective isometry
  $T:M_{\Phi}\left(\mathcal{M},\tau\right)\to M_{\Phi}\left(\mathcal{M},\tau\right)$, we have  $T\left(x\right)\perp T\left(y\right)$.
In particular, $x\perp y$. 

   As a consequence, if  $ 0\le x,y \in M_\Phi\left(\mathcal{M},\tau\right)$
   have $\tau$-finite supports and are such that $x\perp y$  and $y$ is uniformly bounded (that is, $y\in \mathcal{M}$), then
   we have $T\left(x\right)\perp T\left(y\right)$. 
\end{lemma}

By standard approximation argument, we can extend Lemma \ref{Marcin1} to the setting when $x,y$ are general positive
elements in $M_\Phi(\cM,\tau)$. 
\begin{prop} \label{non-atomic posi1}
 Let $\Phi:[0,\infty) \to[0,\infty)$ be admissible
   satisfying $\frac1\Phi\in \left(L_1+L_\infty\right)\left(0,\infty\right)$ and  $\Phi\left(\infty\right):=\displaystyle\lim_{t \to \infty} \Phi\left(t\right)=\infty$.
   Let $\mathcal{M}$ be non-atomic and $T$ be a  positive surjective isometry  from $M_{\Phi}\left(\mathcal{M},\tau\right)$ onto $M_{\Phi}\left(\mathcal{M},\tau\right)$. 
If $0\le x,y\in M_\Phi\left(\mathcal{M},\tau\right)$ such that $x\perp y$, then $T\left(x\right)\perp T\left(y\right)$.
\end{prop}
\begin{proof} 
{\bf Case $1$}. 
Assume that $x$ has  $\tau$-finite support.
The case when $x=0$ is trivial.
If $\tau(\mathbf{1})<\infty$ and $\tau(r(y))=\tau(\mathbf{1})$, then $r(y)=\mathbf{1}$ and in this case, $xy=0$ implies that $x=0$.
For the remaining cases,
we may assume that $\tau(r(x))\le\tau(r(y))$,
$\mu\left(x\right) = \frac{1}{\Phi}\chi_{_{\left(0,s_1\right)}}$ and 
$\mu\left(y\right) = \frac{1}{\Phi}\chi_{_{\left(0,\tau\left(r(y)\right)\right)}} $
\footnote{Otherwise, 
there exist $z_1,z_2\in {\rm extr}\left(B^+_{M_{\Phi}(\cM,\tau)}\right)$ 
such that $z_1\ge x$ and $z_2\ge y$. 
Then,
$z_1,z_2$ are of the form in Theorem \ref{posi non-atomic} and
$T(z_1)T(z_2)=0$ implies $T(x)T(y)=0$. }.
By \cite[Lemma 3.7.7]{DPS}, for any positive integer $n$ satisfying $s_1>\frac 1n$,  there exists a decomposition $ r(y) = e_1\left(n\right)+e_2\left(n\right)+e_3$ such that 
$e_1\left(n\right),e_2\left(n\right),e_3$ are mutually orthogonal and commuting with $y$, satisfying  
\begin{enumerate} [label=(\roman*)]
    \item $y_1(n):=ye_1\left(n\right)$, $y_2(n):=ye_2\left(n\right)$, $y_3:=ye_3$.
    \item 
    $ e^{y} \left( \frac{1}{\Phi\left(\frac1n\right)},\infty \right)   \le e_1\left(n\right)\le e^{y} \left[\frac{1}{\Phi\left(\frac1n\right)},\infty \right)$.   
    \item $\tau\left(e_1\left(n\right)\right)=\frac1n$, $\tau\left(e_2\left(n\right)\right) = s_1-\frac1n$, 
    $\tau\left(e_3\right)=\tau\left(r(y)\right)-s_1$.
    \item $e^{y} \left( \frac{1}{\Phi\left(s_1\right)},\infty \right)   \le e_1\left(n\right)+e_2\left(n\right)\le e^{y} \left[\frac{1}{\Phi\left(s_1\right)},\infty \right)$.
    \item $\mu\left(y_1\left(n\right)+y_2\left(n\right)\right)=\mu\left(x\right)$, $\mu\left(y_1\left(n\right)\right)= \frac{1}{\Phi}\chi_{_{\left(0, \frac1n \right)}}$,   $\mu\left(y_2\left(n\right)\right)= \mu\left(\frac{1}{\Phi}\chi_{_{\left(\frac1n, s_1\right)}}\right) $ and $\mu\left(y_3\right) =  \mu\left( \frac{1}{\Phi}\chi_{_{\left(s_1,\tau\left(r(y)\right) \right)}}\right) $.
\end{enumerate} 
Note that $e_1\left(n\right)$ is decreasing as $n\to \infty$. 
By Lemma \ref{Marcin1}, we obtain 
\begin{align} \label{y1, y2}
T\left(x\right)\perp T\left(y_3\right), ~ T\left(x\right)\perp T\left(y_2\left(n\right)\right). 
\end{align}

Now, it suffices to prove that $T\left(y_1\left(n\right)\right) \to 0$ in  measure  as $n\to \infty $. 
Recall that $\mu\left(y_1\left(n\right)\right)= \frac{1}{\Phi}\chi_{_{\left(0, \frac1n \right)}}$,
which together with Theorem \ref{posi non-atomic} yields that 
$y_1(n)\in \operatorname{extr}\left(B^+_{M_{\Phi}\left(\mathcal{M},\tau\right)}\right)$.
Then, we have
$T\left(y_1(n)\right)\in \operatorname{extr}\left(B^+_{M_{\Phi}\left(\mathcal{M},\tau\right)}\right)$. 
It follows from Theorem \ref{posi non-atomic} that 
\begin{align}\label{muTy1n}
\mu\left(T\left(y_1\left(n\right)\right)\right)=\frac{1}{\Phi}\chi_{_{\left(0,t\left(n\right)\right)}},
\end{align}
where $t(n):=\tau(r(T(y_1(n))))$.
Now, we claim that $t\left(n\right)\to 0$ as $n\to \infty $. 
Since $y_1\left(n\right)\downarrow$ as $n\uparrow$ 
and $T$ is positive, it follows that $t\left(n\right)\downarrow $  as $n\uparrow$. 
Let $t_0:= \inf_n t\left(n\right)$. 
\begin{itemize}
    \item Assume that   $t_0 =\tau\left(\mathbf{1}\right)$. 
Then,  by \cite[Proposition 2.2]{CKS}, we have 
\[ \frac{1}{\Phi\left(s\right)}= \mu\left(s;T\left(y\right)\right) > \mu\left(s;T\left(y_1\left(n\right)\right)\right) =\frac{1}{\Phi\left(s\right)} \]
for some $s>0$,
which is a contradiction. 
\item Assume that $t_0\in \left(0,\tau\left(\mathbf{1}\right) \right)$. 
Then, for  any $m\ge n$,  we have 
\[\mu\Big(T\left(y_1\left(m\right)\right) + T\left( y_1\left(n \right)\right)-T\left(y_1\left(m\right)\right)\Big)
=\mu\left(T\left(y_1\left(n\right) \right)\right)
\stackrel{\eqref{muTy1n}}{=}\frac{1}{\Phi}\chi_{_{\left(0,t\left(n\right)\right)}}.  \]
By Lemma \ref{Marcin1}, we obtain 
$ T\left(y_1\left(m\right)\right) \perp T\left(y_1\left(n \right)-y_1\left(m\right)\right)$.
This together with \eqref{muTy1n} yields that
\[\mu\left(T\left(y_1\left(n \right)-y_1\left(m\right)\right)\right)= \mu \left(\frac{1}{\Phi} \chi_{_{\left(t\left(m\right),t\left(n\right)\right)}}\right). \]
Since $t\left(n\right)\to t_0$ as $n\to \infty $, it follows that 
\[ \mu\left(T\left(y_1\left(n \right)-y_1\left(m\right)\right)\right)
= \mu \left(\frac{1}{\Phi} \chi_{_{\left(t\left(m\right),t\left(n\right)\right)}}\right)   
\le{\frac1{\Phi(t_0) }} \chi_{_{\left(0, t\left(n\right)-t_0 \right)}}
\xrightarrow{\norm{\cdot}_{M_\Phi{ \left(0,\infty\right)}}}  
0  \]
for all $m\ge n$ whenever $n\to \infty$. 
However, for any fixed $n$, 
there exists $m$ large enough such that
\begin{align*}
    \norm{y_1(n)-y_1(m)}_{M_\Phi(\mathcal{M},\tau)}&=\sup_{s>0}\Phi(s)\mu\left(s;\frac1\Phi\chi_{\left(\frac1m,\frac1n\right)}\right)\\
    &=\sup_{s>0}\frac{\Phi(s)}{\Phi(s+\frac1m)}\chi_{\left(0,\frac1n-\frac1m\right)}(s)\\
    &\ge \sup_{s>0}\frac{\Phi(s)}{\Phi(\frac1n)}\chi_{\left(0,\frac1n-\frac1m\right)}(s)\\
    &\ge \frac{\Phi\left(\frac 1{2n}\right)}{\Phi\left(\frac1n\right)}\ge \frac1{C_\Phi}>0.
\end{align*}
This is a contradiction. 
\end{itemize}

Hence, $t_0=0$ and our claim is proved.
In other words, $T\left(y_1\left(n\right)\right)\to 0$ in measure as $n \to \infty $,
which together with \eqref{y1, y2} yields $T(x)\perp T(y)$.

{\bf Case $2$}. Assume that both $x$  and   $y$ have  $\tau$-infinite supports. 
Without loss of generality, we may assume that $x,y$ are extreme points. 
For each $n$,
let $e(n)\in P(\cM)$ be such that
$e^{y} \left(\frac{1}{\Phi\left(n\right)},\infty \right)\le e\left(n\right)\le e^{y} \left[\frac{1}{\Phi\left(n\right)},\infty \right)$ and
$$\mu\left(ye(n)\right)=\frac1 \Phi\chi_{_{\left(0,n\right)}},
~\mu\left(y-ye(n)\right)=\mu\left(\frac1 \Phi\chi_{_{\left(n,\tau\left(\mathbf{1}\right)\right)}}\right).$$
Denote $y_1\left(n\right):=ye\left(n\right)$ and $y_2\left(n\right):=y-y_1\left(n\right)$.


By Case $1$, we have $T\left(x\right)\perp T\left(y_1\left(n\right)\right)$. 
Now, it suffices to show that $T\left(y_2\left(n\right)\right)\to 0$ in measure as $n\to \infty $. 
Since $y_1\left(n\right),y$ are  extreme points, it follows that 
$\mu\left(T\left(y_1(n)\right)\right) =\frac{1}{\Phi}\chi_{_{\left(0,t\left(n\right)\right)}}$ and $\mu\left(T\left(y\right)\right)= \frac{1}{\Phi}\chi_{_{\left(0,s_0\right)}}$ for some constant $0<s_0 \le \tau\left(\mathbf{1}\right)$. 
By Lemma \ref{Marcin1}, we obtain that $T\left(y_1\left(n\right)\right)\perp T\left(y_2\left(n\right)\right)$. 
Hence, 
\begin{align*}
    \mu(T(y))=\mu\left(\mu\left( T\left(y_1(n)\right)\right) \oplus \mu\left(T\left(y_2\left(n\right)\right)\right)\right) 
\end{align*}
and so,
\begin{align}\label{Ty2}
\mu\left(T\left(y_2\left(n\right)\right)\right)
=
\mu\left( \frac{1}{\Phi} \chi_{_{\left(t\left(n\right),s_0 \right)}}\right) .
\end{align}

Note that  $\tau\left(\mathbf{1}\right)=\infty$. 
We claim that $t\left(n\right) \to s_0$ and $s_0=\tau\left(\mathbf{1}\right)$ as $n\to \infty  $.
Indeed, otherwise, we have  $t_0:=\lim_{n\to \infty }t\left(n\right)  <  \infty $.
Note that \[\mu\left(T\left(y_1\left(n\right) \right)\right) = \mu\left(T\left(y_1\left(m\right)\right) + T\left( y_1\left(n \right)\right)-T\left(y_1\left(m\right)\right)\right)=\frac{1}{\Phi}\chi_{_{\left(0,t\left(n\right)\right)}}  \]
for any $m\le n$. 
By Lemma \ref{Marcin1}, we have $T\left(y_1\left(m\right)\right) \perp  T\left( y_1\left(n \right)\right)-T\left(y_1\left(m\right)\right)$.
Since \[\mu\left(T\left(y_1\left(n\right)\right)-T\left(y_1\left(m\right)\right)\right)\le \frac 1{\Phi\left(t\left(m\right)\right)}\chi_{_{\left(0,t_0-t\left(m\right)\right)}},\] 
it follows that \[\norm{\mu\left(T\left(y_1\left(n \right)-y_1\left(m\right)\right)\right)}_{M_\Phi\left(\mathcal{M},\tau\right)}
\to 0,\quad m\to \infty.\]
Moreover, since \[\mu\left(y_1\left(n\right)-y_1\left(m\right)\right)=\frac1 {\Phi\left(\cdot+m\right)}\chi_{_{\left(0,n-m\right)}},\]
it follows that 
\begin{align*}
\norm{\mu\left(y_1\left(n \right)-y_1\left(m\right)\right)}_{M_\Phi(0,\infty)}%
&=\sup_{t>0}\frac{\Phi(t)}{\Phi\left(\cdot+m\right)}\chi_{\left(0,n-m\right)}(t)\\
&\ge \sup_{t>0}\frac{\Phi\left(\left(n-m\right)_-\right)}{\Phi\left(2\left(n-m\right)_-\right)}\chi_{\left(0,n-m\right)}(t)\\
&\ge\frac1 {C_\Phi},
\quad n\to \infty
\end{align*}
for any fixed $m$,
a contradiction. 
 Hence, $t\left(n\right) \to s_0=\tau\left(\mathbf{1}\right)= \infty  $ as $n\to \infty  $.
Note that $\frac1 {\Phi\left(t\left(n\right)\right)}\to0$ as $n\to \infty $. 
By \eqref{Ty2}, we have  $T\left(y_2\left(n\right)\right)\in \cM$ and  $\norm{T\left(y_2\left(n\right)\right)}_\infty \to 0$ as $n\to \infty $. 
Hence, 
\[T\left(y\right)-T\left(y_1\left(n\right)\right)=T(y_2(n)) \to 0\]
in measure as $n\to \infty$ \cite[Proposition 4.4.4]{DPS}.
Since 
$T\left(x \right)T\left(y_1\left(n\right)\right)=0 $, it follows from \cite[Proposition 2.6.11]{DPS} that $T\left(x\right)T\left(y\right)=0$. 
This completes the proof. 
\end{proof}

The case for atomic von Neumann algebras follows from a similar argument.

\begin{prop} \label{atomic posi}
   Assume that $\mathcal{M}$ is atomic with all atoms having trace equal to one.
Let $\Phi:[0,\infty) \to[0,\infty)$ be admissible
   satisfying $\Phi\left(\infty \right):=\displaystyle\lim_{t \to \infty} \Phi\left(t\right)=\infty $.
   Let $T$ be a  positive surjective isometry  from $M_{\Phi}\left(\mathcal{M},\tau\right)$ onto $M_{\Phi}\left(\mathcal{M},\tau\right)$. 
   Assume that $0\le x,y\in M_\Phi\left(\mathcal{M},\tau\right)$. 
   If $x\perp y$, then $T\left(x\right)\perp T\left(y\right)$.
\end{prop}


\section{Descriptions of positive isometries}
Let $\Phi:[0,\infty) \to[0,\infty)$ be admissible.
Having the disjointness-preserving property at hand, we provide a complete description of positive surjective isometries on 
$M_{\Phi}\left(\mathcal{M},\tau\right)$ in this section.
Define $\Phi(s_-):=\lim_{t\uparrow s}\Phi(t)$ for all $s>0$.
We start with the following simple observation. 
  \begin{lemma} \label{alpha}
Let $\mathcal{M}$ be non-atomic.
Let $\Phi:[0,\infty) \to[0,\infty)$ be admissible.
 Fix  two nonzero projections
    $p,q\in P(\cM)\cap \cF(\cM,\tau)$ satisfying $pq=0$.
Then, 
\[\norm{p+\alpha q}_{M_{\Phi}\left(\mathcal{M},\tau\right)}
=\Phi\left(\left(\tau\left(p\right)\right)_-\right),
\quad\norm{\alpha p+ q}_{M_{\Phi}\left(\mathcal{M},\tau\right)}
=\Phi\left(\left(\tau\left(q\right)\right)_-\right)\]
whenever $\alpha$ is sufficiently small.
\end{lemma}
\begin{proof}
Since $pq=0, ~\tau(p),\tau(q)<\infty$, it follows that when $\alpha$ is sufficiently small, we have 
    \[\norm{p+\alpha q}_{M_{\Phi}\left(\mathcal{M},\tau\right)}=\max\{\Phi\left(\left(\tau(p)\right)_-\right),\alpha\Phi\left(\left(\tau(p+q)\right)_-\right)\}=
    \Phi\left(\left(\tau\left(p\right)\right)_-\right)\]
    and
    \[\norm{\alpha p+ q}_{M_{\Phi}\left(\mathcal{M},\tau\right)}=\max\{\Phi\left(\left(\tau(q)\right)_-\right),\alpha\Phi\left(\left(\tau(p+q)\right)_-\right)\}=\Phi\left(\left(\tau\left(q\right)\right)_-\right).\]
\end{proof}



 The following result should be compared with 
 the corresponding $L_p$-results of Yeadon \cite{Y81} 
 and 
 with the cases of (noncommutative) symmetric spaces treated in
 \cite{KR,z77,Z76,zaidenberg,R97,FGHS}.

\begin{theorem}  \label{non-atomic posi2}
Let $\Phi:[0,\infty) \to[0,\infty)$ be strictly increasing and  admissible
     satisfying $\Phi\left(\infty \right):=\displaystyle\lim_{t \to \infty} \Phi\left(t\right)=\infty $ and $\frac1\Phi\in \left(L_1+L_\infty\right)\left(0,\infty\right)$.
     Let $\mathcal{M}$ be 
     non-atomic.
     If    $T:M_{\Phi}\left(\mathcal{M},\tau\right)\to M_{\Phi}\left(\mathcal{M},\tau\right)$ is a  positive surjective isometry,  
     then 
     there  exist a (normal)  Jordan $*$-isomorphism $J:\cM\to \cM$ 
     and a positive scalar $b$ such that
    \[T\left(x\right)=bJ\left(x\right),~ \forall x \in M_{\Phi}\left(\mathcal{M},\tau\right)\cap \mathcal{M}, \]
where $J$ is such that    $\tau\left(J\left(p\right)\right)= \tau\left(J\left(q\right)\right)$ for all $\tau$-finite projections $p,q$   satisfying $\tau\left(p\right)=\tau\left(q\right)$. 
\end{theorem}

\begin{proof}
    By  Proposition \ref{non-atomic posi1} and  Proposition~\ref{order isom},   $T$ is disjointness-preserving and normal. 
By the same argument used in the proof of  \cite[Corollary 5.3]{HSZ},
    there  exist  a positive self-adjoint (not necessarily measurable) operator $b$ affiliated with $Z\left(\mathcal{M}\right)$ and a normal Jordan $^*$-isomorphism $J:\cM\to \cM$ such that 
\begin{align}\label{eq:T=bJ}
T\left(x\right)=bJ\left(x\right),~x\in M_{\Phi}\left(\mathcal{M},\tau\right)\cap \mathcal{M}.
\end{align}
By the proof of \cite[Proposition 3.8]{HSZ}, we have that
$r(b)=\mathbf{1}$.

 Next we will prove that $b$ is a scalar multiple of $\mathbf{1}$. 
There are two possibilities:
\begin{enumerate}[label=\arabic*.]
    \item for any two nonzero $\tau$-finite projections $p,q$ with $p\perp q$ satisfying $\tau\left(p\right)=\tau\left(q\right)$, we have  $\tau\left(J\left(p\right)\right)= \tau\left(J\left(q\right)\right)$.
    \item there  exist  two nonzero $\tau$-finite projections $p_1,p_2$ with $p_1\perp p_2$ satisfying $\tau\left(p_1\right)=\tau\left(p_2\right)$ and $\tau\left(J\left(p_1\right)\right)\neq \tau\left(J\left(p_2\right)\right)$.
\end{enumerate}
 
\begin{enumerate}
 \item[Case $1$.]
 Assume that there does not exist a scalar $c$ such that $b= c\mathbf{1}$.
We claim that there exist two nonzero $\tau$-finite projections $p_1, p_2$ with $p_1\perp p_2, ~\tau\left(p_1\right)=\tau\left(p_2\right)$  and $\tau\left(J(p_1)\right)=\tau\left(J(p_2)\right)<\infty$ such that
\begin{align} \label{J(p1)}
\left(h_1-\varepsilon_1\right)J\left(p_1\right)\le bJ\left(p_1\right)\le \left(h_1+\varepsilon_1\right)J\left(p_1\right)
\end{align}
 and
 \begin{align} \label{J(p2)}
 \left(h_2-\varepsilon_2\right)J\left(p_2\right)\le bJ\left(p_2\right)\le \left(h_2+\varepsilon_2\right)J\left(p_2\right),
 \end{align}
where $h_1-\varepsilon_1>h_2+\varepsilon_2>0$  for some sufficiently small $\varepsilon_1, \varepsilon_2>0$.
Indeed, 
there exist two spectral projections $r_1,r_2$  of $b$ with $r_1\perp r_2$ 
such that
\begin{align} \label{hepsilon1}
\left(h_1-\varepsilon_1\right)r_1\le br_1\le \left(h_1+\varepsilon_1\right)r_1
\end{align}
 and
 \begin{align} \label{hepsilon2}
 \left(h_2-\varepsilon_2\right)r_2\le br_2\le \left(h_2+\varepsilon_2\right)r_2,
 \end{align}
 where $h_1,h_2>0$ are such that $h_1-\varepsilon_1>h_2+\varepsilon_2>0$ for some sufficiently small constants $\varepsilon_1, \varepsilon_2>0$.
Then, there exist $r_1',r_2'\in P(\mathcal{M})$ such that
\begin{align} \label{new}
0<r_1'\le r_1, ~0<r_2'\le r_2, ~\tau(r_1'),\tau(r_2')<\infty.
\end{align}
Fix a scalar $\theta<\infty$ such that $0< \theta<
\min \{\tau(J^{-1}(r_1'),\tau(J^{-1}(r_2')\}$.
Applying \cite[Lemma~3.7.3]{DPS} to $\theta$, there exist $p_1,p_2\in P(\mathcal{M})$ such that 
\begin{align} \label{theta1}
0<p_i<J^{-1}(r_i')\le J^{-1}(r_i), ~i=1,2 
~\text{ and }~
\tau(p_1)=\tau(p_2)=\theta<\infty.
\end{align}
By \eqref{hepsilon1}-\eqref{theta1} and the Condition $1$, we have that 
\begin{align*} 
\tau(J(p_1))=\tau(J(p_2))\le \tau(r_1')<\infty ~\text{and}~ 0\le br_1,br_2\in \mathcal{M}
\end{align*}
For $i=1,2$,
since $J(p_i)\le r_i'\le r_i$ and $br_i$ is bounded, 
it follows that
\begin{align} \label{bJ(pi)}
bJ(p_i)=br_iJ(p_i)\in \cM.
\end{align}
By \eqref{hepsilon1} and \eqref{hepsilon2},
we have
\begin{eqnarray*}
\left(h_i-\varepsilon_i\right)J(p_i)
&=&J(p_i)\left(h_i-\varepsilon_i\right)r_iJ(p_i)\\
&\le& J(p_i)br_iJ(p_i)\\
&\stackrel{{\eqref{bJ(pi)}}}{=}&
{ J(p_i)bJ(p_i)}\\
&\stackrel{{\substack{\eqref{eq:T=bJ}\\\text{$bJ(p_i)=T(p_i)\ge0$}}}}{=}&bJ(p_i)\\
&\le& J(p_i)\left(\left(h_i+\varepsilon_i\right)r_i\right)J(p_i)\\
&=&\left(h_i+\varepsilon_i\right)J(p_i), ~i=1,2.
\end{eqnarray*}
Furthermore, since $r_1\perp r_2$, we have that $p_1\perp p_2$,
which proves our claim.
Observe that $\norm{\alpha_1 p_1+\alpha_2p_2}_{M_{\Phi}\left(\mathcal{M},\tau\right)}=\norm{\alpha_2 p_1+\alpha_1p_2}_{M_{\Phi}\left(\mathcal{M},\tau\right)}$ for all $\alpha_1>\alpha_2>0$. 
Since $T$ is an isometry, it follows from \eqref{eq:T=bJ} that
\begin{align} \label{contradi}
&\quad\norm{\alpha_1 bJ(p_1)+\alpha_2bJ(p_2)}_{M_{\Phi}\left(\mathcal{M},\tau\right)}\nonumber\\
&=
\norm{\alpha_2 bJ(p_1)+\alpha_1bJ(p_2)}_{M_{\Phi}\left(\mathcal{M},\tau\right)}.
\end{align}

By Lemma \ref{alpha},
we may assume that 
  $\alpha_2$ is small enough such that 
\begin{align*}
&\quad \norm{\alpha_1\left(h_1-\varepsilon_1\right) J\left(p_1\right)+\alpha_2\left(h_2-\varepsilon_2\right)J\left(p_2\right)}_{M_{\Phi}\left(\mathcal{M},\tau\right)}\nonumber\\
&=\alpha_1\left(h_1-\varepsilon_1\right)\Phi\left(\left(\tau\left(J\left(p_1\right)\right)\right)_-\right)\nonumber \\
&>\alpha_1\left(h_2+\varepsilon_2\right)\Phi\left(\left(\tau\left(J\left(p_1\right)\right)\right)_-\right)\nonumber \\
&=\norm{\alpha_2\left(h_1+\varepsilon_1\right)  J\left(p_1\right)+\alpha_1\left(h_2+\varepsilon_2\right)J\left(p_2\right)}_{M_{\Phi}\left(\mathcal{M},\tau\right)}.
\end{align*}
This together with \eqref{J(p1)} and \eqref{J(p2)} implies that
\[\norm{\alpha_1 bJ(p_1)+\alpha_2bJ(p_2)}_{M_{\Phi}\left(\mathcal{M},\tau\right)}>
\norm{\alpha_2 bJ(p_1)+\alpha_1bJ(p_2)}_{M_{\Phi}\left(\mathcal{M},\tau\right)},\]
which contradicts with \eqref{contradi}.
\item[Case $2$.] Firstly,  we claim that for Case $2$, $b$ is not a scalar multiple of $\mathbf{1}$.
Assume by contradiction that there exists a scalar $c$ such that $b=c\mathbf{1}$. 
We first consider the case that one of $\tau\left(J\left(p_1\right)\right),\tau\left(J\left(p_2\right)\right)$ is infinite. Without loss of generality, we may assume that $\tau\left(J\left(p_1\right)\right)=\infty$.
Since $T$ is an isometry, it follows that \[\norm{p_1}_{M_{\Phi\left(\mathcal{M},\tau\right)}}
\stackrel{\eqref{eq:T=bJ}}{=}
\norm{bJ\left(p_1\right)}_{M_{\Phi\left(\mathcal{M},\tau\right)}}.\]
By $b=c\mathbf{1}$, we have that 
$\norm{bJ\left(p_1\right)}_{M_{\Phi\left(\mathcal{M},\tau\right)}}=\infty$, which is impossible 
 and proves our claim in this case.

Next, we consider the case that 
$\tau\left(J\left(p_1\right)\right),\tau\left(J\left(p_2\right)\right)<\infty$.
For all $\alpha_1,\alpha_2>0$, 
we have that
\[\norm{\alpha_1p_1+\alpha_2p_2}_{M_{\Phi}\left(\mathcal{M},\tau\right)}=\norm{\alpha_1p_2+\alpha_2p_1}_{M_{\Phi}\left(\mathcal{M},\tau\right)}.\] 
Since $T$ is an isometry, it follows from \eqref{eq:T=bJ} that \[\norm{\alpha_1J\left(p_1\right)+\alpha_2J\left(p_2\right)}_{M_{\Phi}\left(\mathcal{M},\tau\right)}=\norm{\alpha_1J\left(p_2\right)+\alpha_2J\left(p_1\right)}_{M_{\Phi}\left(\mathcal{M},\tau\right)}.\]
Let $\alpha_1=1$ and $\alpha_2$ be  sufficiently small. 
By Lemma \ref{alpha}, we have \[\norm{\alpha_1J\left(p_1\right)+\alpha_2J\left(p_2\right)}_{M_{\Phi}\left(\mathcal{M},\tau\right)}=\Phi\left(\left(\tau\left(J\left(p_1\right)\right)\right)_-\right),\] 
\[\norm{\alpha_1J\left(p_2\right)+\alpha_2J\left(p_1\right)}_{M_{\Phi}\left(\mathcal{M},\tau\right)}=\Phi\left(\left(\tau\left(J\left(p_2\right)\right)\right)_-\right).\]  
Since $\Phi$ is strictly increasing and $\tau\left(J\left(p_1\right)\right)\neq \tau\left(J\left(p_2\right)\right)$, it follows that \[\norm{\alpha_1J\left(p_1\right)+\alpha_2J\left(p_2\right)}_{M_{\Phi}\left(\mathcal{M},\tau\right)}\ne\norm{\alpha_1J\left(p_2\right)+\alpha_2J\left(p_1\right)}_{M_{\Phi}\left(\mathcal{M},\tau\right)},\]
a contradiction.
Therefore, $b$ is not a scalar 
and our claim is proved.

Below, we will prove that Case 2 is impossible.
Let $e^b$ denote the spectral measure of $b$.
There exist $0< l_2\le s_2<l_1\le s_1<\infty$ such that
\[r_i:=e^b[l_i,s_i]\ne0, ~i=1,2.\]
Arguing similarly as Case $1$, we obtain that there exist two nonzero $\tau$-finite projections $q_1,q_2$ with $q_1\perp q_2$ such that
\begin{align*} 
\tau(q_1)=\tau(q_2)=s, ~\tau(J(q_1)),\tau(J(q_2))<\infty , ~q_i\le J^{-1}(r_i), ~i=1,2
\end{align*}
for some constant $s>0$
and
\[l_iJ(q_i)\le bJ(q_i)\le s_iJ(q_i), ~i=1,2.\]
Let 
\[y_i:=T(q_i)\stackrel{\eqref{eq:T=bJ}}{=}bJ(q_i)\in \cM, ~i=1,2.\]
This implies that
\begin{align*} 
l_iJ(q_i)\le y_i\le s_iJ(q_i), ~i=1,2.
\end{align*}
Then,
\begin{align} \label{support}
r(y_i)=l(y_i)=J(q_i), ~l_i \chi_{_{[0,\tau(J(q_i)))}}\le \mu(y_i)\le s_i\chi_{_{[0,\tau(J(q_i)))}}, ~i=1,2.
\end{align}
Choose $\delta$ such that
\begin{align} \label{0< delta}
0<\delta<\min\left\{\frac{\Phi\left(\left(2s\right)_-\right)}{\Phi\left(s_-\right)},\frac{l_1}{s_2}\right\}-1.
\end{align}
Note that
\begin{eqnarray*}
& &\norm{(1+\delta)q_1+q_2}_{M_\Phi\left(\cM,\tau\right)}\\
&=&\max\left\{(1+\delta)\Phi\left(\left(\tau(q_1)\right)_-\right),\Phi\left(\left(\tau(q_1+q_2)\right)_-\right)\right\}\\
    &=&\max\left\{(1+\delta)\Phi\left(s_-\right),\Phi\left(\left(2s\right)_-\right)\right\}\\
    &\stackrel{\eqref{0< delta}}{=}&\Phi\left(\left(2s\right)_-\right)
\end{eqnarray*}
and 
\begin{align*}
    \norm{(1+\delta)q_2+q_1}_{M_\Phi\left(\cM,\tau\right)}=\Phi\left(\left(2s\right)_-\right), ~ \norm{q_1+q_2}_{M_\Phi\left(\cM,\tau\right)}=\Phi\left(\left(2s\right)_-\right).
\end{align*}
Since $T$ is an isometry, it follows that
\begin{equation} \label{1+delta}
\begin{aligned}
\norm{(1+\delta)y_1+y_2}_{M_\Phi\left(\cM,\tau\right)}
&=\norm{(1+\delta)y_2+y_1}_{M_\Phi\left(\cM,\tau\right)}\\
&=\norm{y_1+y_2}_{M_\Phi\left(\cM,\tau\right)}\\
&=\Phi\left(\left(2s\right)_-\right).
\end{aligned}
\end{equation}
By $q_1\perp q_2$ and Proposition \ref{non-atomic posi1}, we have that
\begin{align*} 
y_1=T(q_1)\perp T(q_2)=y_2.
\end{align*}
Since $0<l_2\le s_2<l_1\le s_1<\infty$, it follows from \cite[Lemma 2.3]{CDS97} that \footnote{$
   \mu\left(x\oplus y\right)$ is the singular value function of $\mu(x)\oplus \mu(y)$ affiliated with the von Neumann algebra $L_\infty (0,\infty )\oplus L_\infty (0,\infty )$ with respect to the measure $m\oplus m$.  }
\begin{eqnarray}\label{y1+y2 1}
\mu(y_1+y_2)&=&\mu(y_1\oplus y_2)\\
&\stackrel{\eqref{support}}{=}&\mu(y_1)\chi_{_{[0,\tau\left(J(q_1)\right)}}\nonumber\\
&\quad+&\mu(\cdot-\tau\left(J(q_1)\right);y_2)\chi_{_{[\tau(J(q_1)),\tau(J(q_1+q_2)))}},\notag
\end{eqnarray}
\begin{align}\label{y1+y2 2}
&\qquad \quad  \mu\left((1+\delta)y_1+y_2\right) \\
&\stackrel{\eqref{support}}{=}(1+\delta)\mu(y_1)\chi_{_{\left[0,\tau(J(q_1))\right)}}+\mu\left(\cdot-\tau(J(q_1));y_2\right)\chi_{_{[\tau(J(q_1)),\tau(J(q_1+q_2)))}}\notag
\end{align}
and
\begin{align} \label{y1+y2 3}
&\qquad \quad  \mu\left((1+\delta)y_2+y_1\right)\\
&  \stackrel{\substack{\tiny \eqref{support} \\\eqref{0< delta}}}{=}\mu(y_1)\chi_{_{\left[0,\tau(J(q_1))\right)}} +(1+\delta)\mu\left(\cdot-\tau(J(q_1));y_2\right)\chi_{_{[\tau(J(q_1)),\tau(J(q_1+q_2)))}}.\notag 
\end{align}
A combination of \eqref{y1+y2 1}, \eqref{y1+y2 2} and \eqref{y1+y2 3} that
\[\max\left\{\mu\left(t;(1+\delta)y_1+y_2\right),\mu\left(t;(1+\delta)y_2+y_1\right)\right\}=(1+\delta)\mu(t;y_1+y_2),~t\ge 0.\]
By the definition of $\norm{\cdot}_{M_\Phi\left(\cM,\tau\right)}$,
we have that
\[\begin{aligned}&\quad \max\left\{\norm{(1+\delta)y_1+y_2}_{M_\Phi\left(\cM,\tau\right)},\norm{(1+\delta)y_2+y_1}_{M_\Phi\left(\cM,\tau\right)}\right\}\\
&=(1+\delta)\norm{y_1+y_2}_{M_\Phi\left(\cM,\tau\right)}.\end{aligned}\]
This together with \eqref{1+delta} yields that
\[\Phi\left(\left(2s\right)_-\right)=(1+\delta)\Phi\left(\left(2s\right)_-\right),\]
a contradiction.

\end{enumerate}

Therefore, $b $ is a scalar multiple of $\mathbf{1}$.
By Case $1$, we have that
$\tau(J(p))=\tau(J(q))$ for all $\tau$-finite projections $p,q$ satisfying $\tau(p)=\tau(q)$.
\end{proof}

Note that, Zaidenberg shows that,   the description of surjective isometries on the uniformly bounded part of a symmetric function space   extends to the whole space
 \cite[Theorem 2]{Z76}. 
 However, the noncommutative version of  \cite[Theorem 2]{Z76} is not established yet. 
Even for noncommutative $L_p$-spaces \cite{Y81}, the description of isometries is given in terms of bounded operators. 
Below, we provide a complete characterization of positive surjective  isometries on Marcinkiewicz spaces of possibly unbounded operators.

\begin{proof}[Proof of Theorem \ref{Marcin6}]
    By Theorem \ref{non-atomic posi2}, there  exist a (normal)  Jordan $*$-isomorphism $J:\cM\to \cM$ and a positive scalar $b$ such that
    \begin{align*}
        T\left(x\right)=bJ\left(x\right),~x\in M_\Phi\left(\mathcal{M},\tau\right)\cap \mathcal{M},
    \end{align*}
    where $J$ is such that $\tau\left(J\left(p\right)\right)= \tau\left(J\left(q\right)\right)$ for all $\tau$-finite projections $p,q$   satisfying $\tau\left(p\right)=\tau\left(q\right)$.
    We claim that $J$ is continuous with respect to the measure topology. 
    Indeed, 
   let $\{p_n\}\subset P(\cM)\cap \cF(\cM,\tau)$ be arbitrary such that $\tau(p_n)\to 0$.
    Since $\mu\left(t;e\right)=\chi_{_{[0,\tau\left(e\right))}}$ for all  $e\in P\left(\cM\right)\cap \cF(\cM,\tau)$,
    it follows that
    \begin{align*}
    b\Phi\left(\tau\left(J\left(p_n\right)\right)_-\right)=\norm{T(p_n)}_{M_\Phi\left(\mathcal{M},\tau\right)}=\norm{p_n}_{M_\Phi\left(\mathcal{M},\tau\right)}=\Phi(\tau(p_n)_-)\to \Phi(0)=0
    \end{align*}
    whenever $n\to \infty$, which implies that $\tau(J(p_n))\to 0$.
    By \cite[Theorem 3.13]{W09} (see also \cite[Proposition 4.7(i)]{L99}), we obtain that $J$ is continuous with respect to the measure topology,
    which proves our claim.
    Therefore, by
    \cite[Proposition 3.4]{BHS24} (or \cite[Proposition 2.9.2(i)]{DPS}, \cite[Proposition 3.3]{PS}),
   $J$ can be uniquely extended to a normal Jordan $*$-isomorphism 
    \[\tilde J:S\left(\mathcal{M,\tau}\right)\to S\left(\mathcal{M},\tau\right),\] 
    which is continuous with respect to the measure topology.
From now on, we identify $J$ with $\tilde{J}$.

    We first consider the case that $x\in \operatorname{extr}\left(B^+_{M_{\Phi}\left(\mathcal{M},\tau\right)}\right)$. 
    Then, $\mu\left(x\right)=\frac1 \Phi\chi_{_{\left(0,s\right)}}$ for some constant $0<s\le\tau(\mathbf{1})$ (see Theorem \ref{posi non-atomic}).
    Decompose $x$ into $x_n+\left(x-x_n\right)$, where $0\le x_n\in M_\Phi\left(\mathcal{M},\tau\right)\cap\cM,~ 0\le x-x_n\in M_\Phi\left(\mathcal{M},\tau\right),~x_n \perp x-x_n$ and 
    \[\mu\left(x_n\right)=\mu\left(\frac1\Phi\chi_{_{\left(\frac1n,s\right)}}\right),~\mu\left(x-x_n\right)=\frac1\Phi\chi_{_{\left(0,\frac1n\right)}} . \]
 Noting that $x,x-x_n\in \operatorname{extr}\left(B^+_{M_{\Phi}\left(\mathcal{M},\tau\right)}\right)$ (see Theorem~\ref{posi non-atomic}) implies that $T(x),T(x-x_n)\in \operatorname{extr}\left(B^+_{M_{\Phi}\left(\mathcal{M},\tau\right)}\right)$,
    let 
    \begin{align}\label{T measure}
    \mu(T(x-x_n))=\frac1\Phi\chi_{_{(0,t(n))}},~ \mu(T(x))=\frac1\Phi\chi_{_{(0,r)}}
    \end{align}
    for some constant $0<t(n),r\le \tau(\mathbf{1})$.
    By the same argument as used in the proof of Proposition~\ref{non-atomic posi1}, we obtain that $t(n)\downarrow 0$ as $n\to \infty$.

    For   $0\le y\in M_\Phi\left(\mathcal{M},\tau\right)$, we decompose $y= \left(y-y_n\right) +y_n $, where $y_n$ is bounded and $\tau(r(y-y_n))<\frac{1}{n}$.
 There exists $0\le z_n\in\operatorname{extr}\left(B^+_{M_{\Phi}\left(\mathcal{M},\tau\right)}\right)$ such that $y-y_n\le cz_n$ and $\mu(cz_n)=c\frac 1\Phi\chi_{_{\left(0,\frac 1n\right)}}$ for some constant $c>0$.
    Since $T$ is positive, it follows from \eqref{T measure} that  
    \[\mu\left(T\left(y-y_n\right)\right) \le \mu\left(T\left(cz_n\right)\right)=c\frac1{\Phi}\chi_{_{(0,t_1(n) )}}.   \]
    
    Arguing similarly as above $t(n)\downarrow 0$ as $n\to \infty$, we have that $t_1(n)\downarrow 0$ as $n\to \infty$ and so,
    $T(y-y_n)\to 0$ in measure as $n\to \infty$.
    On the other hand, since $y-y_n\to 0$ in measure, 
    it follows that
    $T\left(y_n\right)=b J\left(y_n\right)\to bJ\left(y\right)$ in measure as $n\to \infty $,
    which implies that
    $T\left(y\right)=bJ\left(y\right)$. 
    Consequently,
    \[T(x)=bJ(x),~x\in M_{\Phi}(\mathcal{M},\tau).\]
\end{proof}

The following special case of Theorem \ref{Marcin6} is noteworthy.

\begin{corollary} \label{non-atomic posi3}
     Assume that $1<p<\infty $.
     Let $\mathcal{M}$ be 
     non-atomic.
     A positive continuous linear surjection $T:L_{p,\infty }\left(\mathcal{M},\tau\right)\to L_{p,\infty }\left(\mathcal{M},\tau\right)$ 
     is an isometry if and only if 
     there exist a (normal) Jordan $*$-isomorphism $J:S(\mathcal{M},\tau)\to S(\mathcal{M},\tau)$
     and a positive scalar $b$ such that
     \[T\left(x\right)=bJ\left(x\right),~ \forall x \in  L_{p,\infty }\left(\mathcal{M},\tau\right) \]
     and $ \tau\left( J\left(x\right) \right) = b^{-p }\tau\left(x\right) $ for all $x\in S\left(\mathcal{M},\tau\right)_+$. 
\end{corollary}
\begin{proof}
   $\left(\Leftarrow\right)$. 
    By \cite[Theorem 3.13]{W09} (see also \cite[Proposition 4.7(i)]{L99}), we obtain that $J$ is continuous with respect to the measure topology. 
  Let $0\le x\in L_{p,\infty}(\mathcal{M},\tau)$ be arbitrary.
    Since $ \tau\left( J\left(x\right) \right) = b^{-p }\tau\left(x\right) $, 
    it follows from \cite[Proposition 2.9.2(iv)]{DPS}  that for any $n>0$,
    \[\tau\left(e^{T(x)}(bn,\infty)\right)=\tau\left(e^{J(x)}(n,\infty)\right)=\tau(J(e^x(n,\infty)))=b^{-p}\tau(e^x(n,\infty)).\]
    This implies that for all $s>0$,
    \begin{align*} 
    \mu(s;T(x))=b\mu(sb^p;x).
    \end{align*}
    In particular,
    \begin{align} \label{isometry on postive part}
    \norm{T(x)}_{L_{p,\infty}(\mathcal{M},\tau)}&=\sup_{t>0} \mu(t;T(x))t^{\frac1p}\nonumber\\
    &=\sup_{t>0}b\mu(tb^p;x)t^{\frac1p}\stackrel{\tiny \mbox{$s=tb^p$}}{\xlongequal{\quad}}\sup_{s>0}\mu(s;x)s^{\frac1p}=\norm{x}_{L_{p,\infty}(\mathcal{M},\tau)}.
    \end{align}

     For a  (not necessarily positive) element $x\in L_{p,\infty}\left(\cM,\tau\right)$, 
    by \cite[Proposition 3.1]{BHS24} (and the same argument as used in the proof of \cite[Proposition 2.19]{HSZ}), there exists a projection $q\in Z(\cM)$ such that 
    \begin{align} \label{J(x)}
    |J(x)|=J\left(q|x|+\left(\mathbf{1}-q\right)|x^*|\right).
    \end{align}
    Let
    $y:=q|x|+\left(\mathbf{1}-q\right)|x^*|\ge0.
    $
    Observe that 
    \begin{align} \label{= mu}
    \mu(x)=\mu(y).
    \end{align}
    Therefore, for all $x\in L_{p,\infty}(\mathcal{M},\tau)$,
    \begin{gather*}
    \norm{T(x)}_{L_{p,\infty}(\mathcal{M},\tau)}
    \stackrel{\eqref{J(x)}}{=}
    \norm{T(y)}_{L_{p,\infty}(\mathcal{M},\tau)}
    \stackrel{\eqref{isometry on postive part}}{=}\norm{y}_{L_{p,\infty}(\mathcal{M},\tau)}
    \stackrel{\eqref{= mu}}{=}
    \norm{x}_{L_{p,\infty}(\mathcal{M},\tau)}.
    \end{gather*}

    Hence,  $T$ is a positive surjective isometry on $L_{p,\infty}(\mathcal{M},\tau)$.
    
 $\left(\Rightarrow\right)$.  By Theorem \ref{Marcin6}, there  exist a normal  Jordan $*$-isomorphism $J:S(\cM,\tau)\to S(\cM,\tau)$ 
    and a positive scalar $b$ such that
    \[T\left(x\right)=bJ\left(x\right),~x\in L_{p,\infty}\left(\mathcal{M},\tau\right).\]
For each $q\in P(\cM)\cap\cF(\cM,\tau)$,    
    it follows from $\mu\left(t;q\right)=\chi_{_{[0,\tau\left(q\right))}}$ 
    that 
    \[\left(\tau\left(q\right)\right)^{\frac{1}{p}}=\norm{q}_{L_{p,\infty}\left(\mathcal{M},\tau\right)}=\norm{T\left(q\right)}_{L_{p,\infty}\left(\mathcal{M},\tau\right)}=b\left(\tau\left(J\left(q\right)\right)\right)^{\frac{1}{p}},\]
    i.e., $b^{-p}\tau(q)=\tau(J(q))$.
    Then, for any 
    $0\le x \in \cF\left(\mathcal{M},\tau\right)$, 
    by uniform approximation with simple operators,
    we have 
    \[\tau\left(J\left(x\right)\right)=b^{-p}\tau\left(x\right).\] 
    Given $0\le x\in S\left(\mathcal{M},\tau\right) $, by \cite[Proposition 2.3.12]{DPS}, there exists an upward directed system $\{x_\beta\}_{\beta\in\Gamma}\subset \mathcal{F}\left(\mathcal{M},\tau\right) $  such that $0\le x_\beta\uparrow x$. 
 By the normality of $J$ and $\tau$, we obtain 
    \[\tau\left(J\left(x\right)\right)=b^{-p}\tau\left(x\right),~  x\in S\left(\mathcal{M},\tau\right)_+ .\]
The proof is complete.
\end{proof}

Let $\left(\Omega_1,\Sigma_1,\nu_1\right),\left(\Omega_2,\Sigma_2,\nu_2\right)$ be two measure spaces.
A invertible function $f:\left(\Omega_1,\Sigma_1,\nu_1\right)\to \left(\Omega_2,\Sigma_2,\nu_2\right) $ is bi-nonsingular if $f,f^{-1}$ are nonsingular, that is, $\nu_2(A)=0$ implies $\nu_1(f^{-1}(A))=0$ and $\nu_1(B)=0$ implies $\nu_2(f(B))=0$.

Let $J$ be a Jordan $*$-isomorphism  from $\mathcal{M}$ onto $\mathcal{M}$.
In the special setting when $\mathcal{M}=L_\infty\left(0,\infty\right)$, there exists an invertible bi-nonsingular measurable transformation
$\phi:\left(0,\infty\right)\to \left(0,\infty\right)$ such that for all $f\in L_\infty \left(0,\infty\right)$, 
\[J(f)(t)=f(\phi(t)) ~\text{a.e.} ~ t\in \left(0,\infty\right)\] \cite[Lemma 8.22]{Tak}.
This together with Corollary \ref{non-atomic posi3} yields the following result.
  
\begin{corollary} \label{measure continu}
    Assume that $1<p<\infty$. 
    A positive continuous linear surjection $T:L_{p,\infty}\left(0,\infty\right)\to L_{p,\infty}\left(0,\infty\right)$ 
    is an isometry if and only if there exist an invertible bi-nonsingular measurable transformation modulo null sets $\phi:\left(0,\infty\right)\to \left(0,\infty\right)$
    and a positive scalar $b$  such that for all $x\in L_{p,\infty}\left(0,\infty\right)$ 
    and measurable sets $B\subset \left(0,\infty\right)$,
    \[T\left(x\right)\left(t\right)=b x\left(\phi\left(t\right)\right), ~ \text{a.e.}~ t\in \left(0,\infty\right)\]
    and
    \[m\left(\phi^{-1}\left(B\right) \right)=b^{-p}m \left(B\right) . \]
\end{corollary}
Now, we are ready to present the proof of Theorem \ref{atomic J}. 
\begin{proof}[Proof of Theorem \ref{atomic J}]
The implication $\Leftarrow$ is trivial.

$(\Rightarrow).$
By Proposition~\ref{atomic posi}, Proposition \ref{T(x)ge0} and Proposition~\ref{order isom}, 
    we have $T$ is disjointness-preserving and normal, $T^{-1}$ is positive.
By the same argument used in the proof of \cite[Corollary 5.3]{HSZ},
    there exist a positive self-adjoint operator $b$ affiliated with $Z\left(\mathcal{M}\right)$ and a normal Jordan $*$-isomorphism $J:\cM\to \cM$ such that \[T\left(x\right)=bJ\left(x\right),~ x \in M_\Phi\left(\mathcal{M},\tau\right).\]



    Next we will prove that $b=\mathbf{1}$. 
    Indeed, 
    let 
    $e_\alpha\in P(\cM)$ be an atom. 
    Since $T$ is a positive isometry, it follows that there exists an atom
    $e_\beta \in P(\cM)$ 
    such that $T\left(e_\alpha  \right) \ge m e_\beta $ for some constant $m>0$. 
   Then, the positivity of $T^{-1}$ yields that $T^{-1}\left(m e_\beta \right)\le e_\alpha $ and so, there exists a constant $k>0$ such that $T^{-1}(e_\beta) = k e_\alpha$, that is,
   \[ T\left(k e_\alpha \right)=e_\beta.\] 
    Since $T$ is an isometry, it follows that $k=1$, i.e., \[T\left(e_\alpha  \right)=e_\beta .\] 
    This implies that $J$ maps  an atom to an atom.

    For an arbitrary atom $e_\gamma\in P(\cM)$, 
    applying above argument to $T^{-1}$, we obtain that there exists an atom $e_\lambda \in P(\mathcal{M})$ such that \[T^{-1}\left(e_\gamma\right)=e_\lambda.\] 
    Then, $T\left(e_\lambda\right)=e_\gamma=bJ\left(e_\lambda\right)=be_\gamma$. 
  It follows from the arbitrariness of $e_\gamma$ that
    $b=\mathbf{1}$. 
    Hence, $T\left(x\right)=J\left(x\right),~ x\in M_\Phi\left(\mathcal{M},\tau\right)$. 
    By \cite[Theorem 3.5]{BHS24}, $J$ preserves  singular value  functions.
\end{proof}

A combination of Theorem \ref{atomic J} and \cite[Theorem A.8]{M07} immediately yields the following corollary.

\begin{corollary} 
  Let $\Phi:[0,\infty) \to[0,\infty)$ be admissible
    satisfying $\Phi\left(\infty \right):=\displaystyle\lim_{t \to \infty} \Phi\left(t\right)=\infty $.
    Let $\mathcal{M}=B(H)$ equipped with the standard trace $\tau$.
    A positive continuous linear mapping $T$ from $M_{\Phi}(B(H),\tau)$ onto $M_{\Phi}(B(H),\tau)$ is an isometry
    if and only if  there exists a unitary operator $U$ such that one of the following two statements holds:
    \begin{enumerate} [label=(\roman*)]
    \item $T(x)=UxU^*, ~\forall x\in M_{\Phi}(B(H),\tau)$.
    \item $T(x)=Ux'U^*, ~\forall x\in M_{\Phi}(B(H),\tau)$, where $x'$ denotes the transpose of $x$.
    \end{enumerate}
\end{corollary}

\chapter{Isometries on weak \texorpdfstring{$L_p$}{Lp}-spaces, \texorpdfstring{$1<p<\infty $}{1 < p < infinity}} \label{s6}

Let $\left(X,\nu\right)$ be a quasi-Banach space equipped with a quasi-norm $\nu$.
Denote by $\operatorname{co(B_X)}$ the convex hull of the unit ball of $X$.
Let the Minkowski functional $\norm{\cdot}_{\hat{X}}$ of $\operatorname{co(B_X)}$  be defined by
\[\norm{x}_{\hat{X}}=\inf\left\{\lambda>0: \frac x\lambda \in \operatorname{co(B_X)}\right\}, ~\forall x\in X.\]
Note that $\norm{x}_{\hat{X}}$ is a semi-norm \cite[Theorem 1.35]{R91}.
If its dual $X^*$ separates points, then $\norm{\cdot}_{\hat{X}}$ is a norm \cite{CK17,K86}.
Let $N:=\left\{x\in X:\norm{x}_{\hat{X}}=0\right\}$ and $\iota:X\to X^{**}$ is the canonical embedding defined by 
\[\iota(x)(\lambda)=\lambda(x), ~\forall x\in X, ~\lambda\in X^*.\]
Note that $X/N$ is a normed linear space \cite{CK17,P09}.
Moreover, since $X/N\cong \iota(X)$ (see, e.g., \cite[Theorem 1.7.13]{M12}, \cite[p. 28]{KPR} and the proof of \cite[Theorem 1.7.14]{M12}), it follows that
$X/N$ is complete if and only if $\iota(X)$ is a closed subspace of $X^{**}$ (see, e.g.,  \cite[p.~213]{P09}).
The completion $\hat{X}$ of the space $(X/N,\norm{\cdot}_{\hat{X}})$ is called the {\itshape Banach envelope} of $X$
\cite[Section 2.4]{KPR}.
The Banach envelope $\hat{X}$ coincides with the closure of $\iota(X)$ in $X^{**}$
(see \cite[p. 213]{P09} and \cite[p. 27]{KPR}).
It should be noted that there are symmetric quasi-Banach spaces  for which $\operatorname{co}(B_X)=X$ and $\hat{X}=\{0\}$ (for example, $X=L_p(0,1), ~0<p<1$ \cite[p.~3]{CK17}).
For any given nonzero quasi-Banach space $X$, if the condition \eqref{rho is a norm} below  is satisfied, then the Banach envelope $\hat{X}\ne\{0\}$.

Decomposition norm $\rho$ of $\nu$ is defined by
\begin{align} \label{rho}
\rho(x):=\inf\left\{\sum_{i=1}^n\nu(x_i): ~x=\sum_{i=1}^nx_i,x_i\in X,n\in\mathbb{N}\right\},~x\in X.
\end{align}
Note that  $\rho$ is a seminorm on $X$ and $\norm{x}_{\hat{X}}=\rho(x)\le \nu(x), ~\forall x\in X$ \cite{CK17}.

We state the following lemma for later reference, whose proof is straightforward.
For the sake of completeness, we provide a short proof below.
\begin{lemma} \label{isometric norm}
    Let $T$ be a surjective isometry on $\left(X,\nu\right)$  and the seminorm $\rho$ be as defined in \eqref{rho}.
    Then, $\rho(Tx)=\rho(x)$.
    If, in addition, there exists $c\ge1$ such that
    \begin{align} \label{rho is a norm}
        \nu(x)\le c\rho(x), ~\forall x\in X,
    \end{align}
    then $\rho$ is a norm.
\end{lemma}
\begin{proof}
    Let $x=\sum_{i=0}^n x_i,x_i\in X$ with $n<\infty $.
    Note that 
    \[\rho(T(x))\le \sum_{i=0}^n \nu(T(x_i))=\sum_{i=0}^n \nu(x_i).\]
    This implies that $\rho(T(x))\le \rho(x)$.
    Applying above argument to $T^{-1}$ and $T(x)$, we obtain that $\rho(x)=\rho(T^{-1}(T(x)))\le \rho(T(x))$.
    Hence, $\rho(T(x))=\rho(x)$.

    
    Moreover, if \eqref{rho is a norm} holds, then $\rho(x)=0, ~x\in X$ implies that $x=0$, 
    that is, $\rho$ is a norm.
\end{proof}

Recall that the von Neumann algebra $\mathcal{M}$ is said to be {\itshape $\sigma$-finite} if every system   $\{e_\alpha\}_{\alpha\in \Lambda}$ of nonzero pairwise orthogonal projections in $P(\mathcal{M})$ is at most countable \cite[Definition 1.14.14]{DPS}.

Let $\cM$ be a non-atomic semifinite von Neumann algebra equipped with a semifinite faithful normal trace $\tau$.
Let $\left(E\left(\mathcal{M},\tau\right),\nu\right)$ be a quasi-Banach symmetric space (affiliated with $\cM$) equipped with a quasi-norm $\nu$.
Define the corresponding quasi-Banach symmetric function space by setting 
\[
E(0,\tau(\mathbf{1})):=\{f\in S(0,\tau(\mathbf{1})): \mu(f)=\mu(x){\text{ for some } x\in E(\cM,\tau)}\}
\]
equipped with the quasi-norm \cite{S14}
$\hat{\nu}(f):=\nu(x)$
for some $x\in E(\cM,\tau)$ with $\mu(x)=\mu(f)$.

A combination of Lemma~\ref{isometric norm},
\cite[Lemma~2.1]{KM07} and \cite[Theorem 2.9]{CK17} yields the following result.

\begin{lemma}\label{lem:symmetric}
Let $\mathcal{M}$ be non-atomic $\sigma$-finite.
If there exists $c\ge1$ such that
        \begin{align}\label{condition}
        \nu(x)\le c\rho(x), ~\forall x\in E\left(\mathcal{M},\tau\right),
        \end{align}
    then $\left(E\left(\mathcal{M},\tau\right),\rho\right)$ is a Banach symmetric space. 
\end{lemma}

Let
\[
E(0,1):=\left\{f\chi_{_{(0,1)}}:f\in E(0,\tau(\mathbf{1}))\right\}.
\]
If $0<\tau(\mathbf{1})<1$, then we may replace $\tau$ with its normalized trace $\tau_1:=\frac{\tau}{\tau(\mathbf{1})}$.

The following result is an immediate corollary of Lemma \ref{isometric norm}, Lemma \ref{lem:symmetric}, \cite[Theorem~7.0.5]{FGHS} and \cite[Theorem~1.4.1]{FGHS}.

\begin{theorem} \label{general}
Let $\mathcal{M}$ is non-atomic $\sigma$-finite.
    Let 
    $T$ be a surjective isometry on $\left(E\left(\mathcal{M},\tau\right),\nu\right)$.
    Assume that
    \begin{enumerate} [label=(\roman*)]
        \item $E(0,1)\ne L_p(0,1), ~1\le p\le \infty$;
        \item inequality \eqref{condition} holds;
        \item 
    $\left(E\left(\mathcal{M},\tau\right),\rho\right)$
        has the Fatou property or is minimal (see \cite[p. 118]{LT2}).
    \end{enumerate}
    Then, there exist
    a unitary operator $u\in\mathcal{M}$,
    a (normal) Jordan $*$-isomorphism $J:S\left(\mathcal{M},\tau\right)\to S\left(\mathcal{M},\tau\right)$ and positive numbers $\alpha,a$
    such that
    \[T(x)=auJ(x), ~\forall x\in E\left(\mathcal{M},\tau\right)\]
    and 
    \[\alpha\tau(J(x))=\tau(x), ~\forall x\in S\left(\mathcal{M},\tau\right)_+.\]

    If, in addition, $\tau(\mathbf{1})<\infty$, then there exist a trace-preserving (normal) Jordan $*$-isomorphism $J:S\left(\mathcal{M},\tau\right)\to S\left(\mathcal{M},\tau\right)$ and a unitary operator $u\in \mathcal{M}$ such that
    \[T(x)=uJ(x), ~\forall x\in E\left(\mathcal{M},\tau\right).\]

\end{theorem}

In the commutative setting, we obtain the following corollary by using   \cite[Theorem 1.1]{Z95}.
\begin{corollary} 
    Let $\left(X,\nu\right)$ be a quasi-Banach symmetric space on $\left(0,\infty\right)$ equipped with a  quasi-norm $\nu$.
    Let $T$ be a surjective isometry on $\left(X,\nu\right)$.
    Assume that $\rho$, as defined in \eqref{rho}, is not proportional to $\norm{\cdot}_2$.
    If there exists $c\ge1$ such that
    \[\nu(f)\le c\rho(f), ~\forall f\in X,\]
    then there exist a measurable function $a$ and an invertible  measurable transformation 
    $\phi:\left(0,\infty\right)\to \left(0,\infty\right)$ such that
    \[T(f)(t)=a(t)f(\phi(t)), ~\forall f\in X, ~\text{a.e.}~t\in \left(0,\infty\right).\]
\end{corollary}


Recall that
for $1<p<\infty$, the norm $\norm{\cdot}_{\left(p,\infty\right)}$ on $L_{p,\infty}\left(\mathcal{M},\tau\right)$ is defined by
\[\norm{x}_{\left(p,\infty\right)}:=\sup_{t>0}t^{1/p}\int_0^t \mu(s;x)\frac{ds}t, ~\forall x\in L_{p,\infty}\left(\mathcal{M},\tau\right)\]
(see \cite[Section 6.7]{DPS}).

Decomposition norm $\rho_p^\cM$ of $\norm{\cdot}_{L_{p,\infty}\left(\mathcal{M},\tau \right)}$ is defined by
\[\begin{aligned}&\rho_p^{\mathcal{M}}(x)
:=\inf\left\{\sum_{i=1}^n\norm{x_i}_{L_{p,\infty}\left(\mathcal{M},\tau\right)}: ~x=\sum_{i=1}^nx_i,x_i\in L_{p,\infty}\left(\mathcal{M},\tau\right),n\in\mathbb{N}\right\}
\end{aligned}\]
for all $x\in L_{p,\infty}\left(\mathcal{M},\tau\right)$.

The following result implies that if $T$ is a surjective isometry on a weak $L_p$-space $\left(L_{p,\infty}\left(\mathcal{M},\tau\right), \norm{\cdot}_{L_{p,\infty}\left(\mathcal{M},\tau\right)}\right)$, then $T$ is a surjective isometry on $\left(L_{p,\infty}\left(\mathcal{M},\tau\right), \norm{\cdot}_{(p,\infty)}\right)$.


\begin{lemma} \label{equi}
    Assume that $1<p<\infty$ and $\frac1p+\frac1{p'}=1$.
    Let $\mathcal{M}$ be a (not necessarily $\sigma$-finite) semifinite non-atomic von Neumann algebra equipped with a semifinite faithful normal trace $\tau$.
    Then,
    \[\rho_p^\cM(x)=\frac{1}{p'}\norm{x}_{\left(p,\infty\right)}, ~\forall x\in L_{p,\infty}\left(\mathcal{M},\tau\right).\]
    In particular, it follows from \cite[Lemma 6.7.7]{DPS} that the assumption on $\nu(\cdot)=\norm{\cdot}_{L_{p,\infty}\left(\mathcal{M},\tau\right)}$ and $\rho=\rho_p^\mathcal{M}$ in Lemma \ref{isometric norm} and Lemma \ref{lem:symmetric} is satisfied.
\end{lemma}
\begin{proof}
The case when $\cM$ is $L_\infty (0,\infty )$ follows from \cite[Theorems 4.1 and   5.2]{BKS}. 
The case when   $\mathcal{M}$ is $\sigma$-finite follows from
       \cite[Theorem 2.9]{CK17}.

      Assume that $\mathcal{M}$ is not $\sigma$-finite.
         Fix  $x\in L_{p,\infty}\left(\mathcal{M},\tau\right)$.
        Let $e=l(x)\vee r(x)$ and $x=u|x|$ be the polar decomposition of $x$.
        Define
        \[r_n:=e^{|x|}\left(\frac1n,\infty\right)\uparrow r(x),~l_n:=ur_nu^*\uparrow l(x).\]
        Note that $e_n:=l_n\vee r_n\uparrow e$ and $\tau(e_n)<\infty$.
        This implies that $e\mathcal{M}e$ is $\sigma$-finite.
    It follows from the $\sigma$-finite case  that
\[\rho_p^{e\mathcal{M}e}(x)= \frac1{p'}\norm{x}_{\left(p,\infty\right)}.\]
        It suffices to prove that $\rho_p^\cM(x) = \rho_p^{e\mathcal{M}e}(x)$.
        By the definition of $\rho_p^\cM$, we have that
        \[\rho_p^\cM(x)\le \rho_p^{e\mathcal{M}e}(x).\]
        For any $x=\sum_{i=0}^n x_i,x_i\in L_{p,\infty}\left(\mathcal{M},\tau\right),n<\infty$.
        Note that $x=\sum_{i=0}^n ex_ie$.
        Since
        \[\sum_{i=0}^n \norm{ex_ie}_{L_{p,\infty}\left(\mathcal{M},\tau\right)}\le\sum_{i=0}^n \norm{x_i}_{L_{p,\infty}\left(\mathcal{M},\tau\right)},\]
        it follows that
        \[\rho_p^{e\mathcal{M}e}(x)\le \rho_p^\cM(x).\]
        Hence, $\rho_p^{\cM}(x) = \rho_p^{e\mathcal{M}e}(x)$.
    The proof is complete.
\end{proof}

 Note that $\rho^{\ell_\infty}$ and $\norm{\cdot}_{(p,\infty)}$ are equivalent norms.
For the situation when $\mathcal{M}=B(H)$ or $\ell_{\infty}$, 
$\rho_p^\mathcal{M}$ is no longer proportional to $\norm{\cdot}_{\left(p,\infty\right)}$. 

\begin{lemma} 
    Assume that $1<p<\infty$ and $H$ is a complex separable Hilbert space.
    \[\rho_p^{\ell_\infty}(x)=\sup_{n\ge0}\frac{\sum_{k=0}^n\mu(k;x)}{\sum_{k=0}^n (k+1)^{-1/p}} , ~\forall x\in\ell_{p,\infty}.\]

    Moreover, 
    \[\rho_p^{B(H)}(x)=\sup_{n\ge0}\frac{\sum_{k=0}^n\mu(k;x)}{\sum_{k=0}^n (k+1)^{-1/p}} , ~\forall x\in\cC_{p,\infty}.\]

    In particular, $\rho_p^{\ell_{\infty}},\rho_p^{B(H)}$ both are norms.
\end{lemma}
\begin{proof}
   By \cite[Theorems 4.2 and 5.2]{BMP12}, it suffices to prove that
   \[\inf_{x\prec\prec y}\norm{y}_{\ell_{p,\infty}}=\sup_{n\ge0}\frac{\sum_{k=0}^n\mu(k;x)}{\sum_{k=0}^n (k+1)^{-1/p}} , ~\forall x\in\ell_{p,\infty}.\]
   For each $x\in \ell_{p,\infty}$,
   fix $y\in \ell_{p,\infty}$ such that $x\prec\prec y$.
   Note that for all $n\ge0$,
   \[\sum_{k=0}^n\mu(k;x)\le\sum_{k=0}^n\mu(k;y)\le \sum_{k=0}^n\norm{y}_{\ell_{p,\infty}}(k+1)^{-1/p},\]
   which implies that
   \[\norm{y}_{p,\infty}\ge \frac{\sum_{k=0}^n\mu(k;x)}{\sum_{k=0}^n (k+1)^{-1/p}}.\]
   Then, 
   \[\inf_{x\prec\prec y}\norm{y}_{\ell_{p,\infty}}\ge\sup_{n\ge0}\frac{\sum_{k=0}^n\mu(k;x)}{\sum_{k=0}^n (k+1)^{-1/p}}.\]
   For the converse implication, let $z=(z_k)_{k=0}^\infty\in \ell_{p,\infty}$ be a decreasing sequence such that
   \[z_k:=(k+1)^{-1/p} \cdot \sup_{n\ge0}\frac{\sum_{i=0}^n\mu(i;x)}{\sum_{i=0}^n (i+1)^{-1/p}}, ~\forall k\ge0.\]
   Note that
   \[\norm{z}_{\ell_{p,\infty}}=\sup_{k\ge0} \mu(k;z)(k+1)^{1/p}=\sup_{n\ge0}\frac{\sum_{i=0}^n\mu(i;x)}{\sum_{i=0}^n (i+1)^{-1/p}}\]
   and
   \[\sum_{k=0}^m \mu(k;z)=\sup_{n\ge0}\frac{\sum_{i=0}^n\mu(i;x)}{\sum_{i=0}^n (i+1)^{-1/p}}\sum_{k=0}^m(k+1)^{-1/p}\ge\sum_{k=0}^m \mu(k;x)\]
   for all $m\ge0$.
   This implies that
   \[\inf_{x\prec\prec y}\norm{y}_{\ell_{p,\infty}}\le \norm{z}_{\ell_{p,\infty}}=\sup_{n\ge0}\frac{\sum_{k=0}^n\mu(k;x)}{\sum_{k=0}^n (k+1)^{-1/p}}.\]
   Therefore,
   \[\inf_{x\prec\prec y}\norm{y}_{\ell_{p,\infty}}=\sup_{n\ge0}\frac{\sum_{k=0}^n\mu(k;x)}{\sum_{k=0}^n (k+1)^{-1/p}} , ~\forall x\in\ell_{p,\infty}.\]

   The second statement follows immediately from \cite[Theorem 2.18]{CK17}.
\end{proof}

Note that $n\mapsto\sum_{k=0}^{n-1}(k+1)^{-1/p}$ generates a Banach Marcinkiewicz sequence space which is equivalent to the classical Banach norm $\left(x\mapsto \sup _{n\ge 0} \frac{\sum_{k=0}^n  \mu(k;x)}{(n+1)^{1-\frac{1}{p}}}\right)$ on $\ell_{p,\infty}$ . 
Consequently, $\left(\ell_{p,\infty},\rho_p^{\ell_\infty}\right)$ and $\left(\mathcal{C}_{p,\infty},\rho_p^{B(H)}\right)$  have the Fatou property \cite{DPS}.

\begin{lemma} \label{Fatou}
    Assume that $1<p<\infty$  and $H$ is a complex separable Hilbert space.
    Then, $\left(\ell_{p,\infty},\rho_p^{\ell_{\infty}}\right)$ and $\left(\mathcal{C}_{p,\infty},\rho_p^{B(H)}\right)$ have the Fatou property.
\end{lemma}

A combination of Lemma \ref{isometric norm}, 
Lemma \ref{Fatou}, \cite[Theorem 5.6]{AC17}  and \cite[Theorem 7]{AC} yields the following result.
\begin{theorem}
    Assume that $1<p<\infty$  and $H$ is a complex separable Hilbert space.
    Then,  a linear operator $T$ is a surjective isometry on $\mathcal{C}_{p,\infty}$ (respectively, $\ell_{p,\infty}$) if and only if  there exist unitary operators $U,V\in B(H)$ (respectively, $\ell_\infty$) such that
    \[T(x)=UxV^*, ~\forall x\in \mathcal{C}_{p,\infty}\]
    or 
    \[T(x)=Ux'V^*, ~\forall x\in \mathcal{C}_{p,\infty}\]
(respectively, there exist a sequence $\{u_n\}_{n=0}^\infty$ satisfying $|u_n|=1$ and a permutation $\sigma:\mathbb{N} \to \mathbb{N}$ of the set $\{0,1,2,\dots,\}$ such that
    \[T(x)(n)=u_nx(\sigma(n)), ~\forall n\in\mathbb{N}, x\in \ell_{p,\infty}.\]
\end{theorem}

Below, we prove  Theorem \ref{sigma-finite}. 
\begin{theorem} \label{sigma Lp,infty}
    Assume that $1<p<\infty$. 
    Let $\mathcal{M}$ be a $\sigma$-finite non-atomic von Neumann algebra equipped with a semifinite faithful normal trace $\tau$.
    Let $T$ be a surjective linear operator on a weak $L_p$-space $L_{p,\infty}\left(\mathcal{M},\tau\right)$.
    Then, $T$ is an isometry if and only if there exist
    a (normal) Jordan $*$-isomorphism $J:S\left(\mathcal{M},\tau\right)\to S\left(\mathcal{M},\tau\right)$, a unitary operator $u$ in $\mathcal{M}$ and a positive scalar $a$ such that
    \[T(x)=auJ(x), ~\forall x\in L_{p,\infty}\left(\mathcal{M},\tau\right)\]
    and 
    \[\tau(J(x))=a^{-p}\tau(x), ~\forall x\in S\left(\mathcal{M},\tau\right)_+.\]

    If, in addition, $\tau(\mathbf{1})<\infty$, 
    then
    $T$ is an isometry if and only if  there exist
    a trace-preserving (normal) Jordan $*$-isomorphism $J:S\left(\mathcal{M},\tau\right)\to S\left(\mathcal{M},\tau\right)$ and a unitary operator $u\in\mathcal{M}$ such that
    \[T(x)=uJ(x), ~\forall x\in L_{p,\infty}\left(\mathcal{M},\tau\right).\]
\end{theorem}
\begin{proof}
   $(\Leftarrow)$. By a similar argument as used in the proof of Corollary \ref{non-atomic posi3}, we obtain that $u^*T$ is a surjective isometry.
     This implies that $T$ is an isometry.
    
     $\left(\Rightarrow\right)$.
    By Lemma \ref{equi} and Lemma \ref{isometric norm}, 
    we obtain that $T$ is a surjective isometry on $\left(L_{p,\infty}\left(\mathcal{M},\tau\right),\norm{\cdot}_{\left(p,\infty\right)}\right)$.
Note that $\left(L_{p,\infty}\left(\mathcal{M},\tau\right),\norm{\cdot}_{\left(p,\infty\right)}\right)$ has the Fatou property \cite[Proposition 6.7.8]{DPS}. 
By 
 Theorem \ref{general}, 
 there exist
 a normal Jordan $*$-isomorphism $J:S\left(\mathcal{M},\tau\right)\to S\left(\mathcal{M},\tau\right)$, a positive scalar $a$ and $u$ is a unitary operator in $\mathcal{M}$
 such that
    \[T(x)=auJ(x), ~\forall x\in L_{p,\infty}\left(\mathcal{M},\tau\right).\]
    Since $u^*T$ is a positive isometry, it follows from   Corollary \ref{non-atomic posi3} that
    \[\tau(J(x))=a^{-p}\tau(x), ~\forall x\in S\left(\mathcal{M},\tau\right)_+.\]

    Moreover, the remaining statements follow from \cite[Theorem 1.4.1]{FGHS} when $\tau(\mathbf{1})<\infty$.
\end{proof}
\begin{remark}
One should compare the results in Theorem \ref{sigma Lp,infty} and 
Corollary \ref{non-atomic posi3}:
the unitary operator in Theorem \ref{sigma Lp,infty} becomes ${\bf 1}$ in the case of   positive surjective isometries.

    We do not know about the description of (positive) surjective isometries on $L_{p,\infty}\left(\mathcal{M},\tau\right)$ for $0<p\le 1$ when $\mathcal{M}$ is non-atomic. 
\end{remark}

As Corollary \ref{measure continu} above, 
an immediate application of Theorem \ref{sigma-finite} and  \cite[Lemma 8.22]{Tak}
yields the following corollary.

\begin{corollary}
    Assume that $1<p<\infty$.
 Let $T$ be a surjective isometry on a weak $L_p$-function space $L_{p,\infty}\left(0,\infty\right)$.
    Then, there exist an invertible 
bi-nonsingular
measurable transformation
$\phi:\left(0,\infty\right)\to \left(0,\infty\right)$, a unimodular function $u$ in $L_\infty\left(0,\infty\right)$ and a positive scalar $a$ such that for all $f\in L_{p,\infty}\left(0,\infty\right)$,
    \[T(f)(t)=au(t)f(\phi(t)) ~\text{a.e.} ~ t\in \left(0,\infty\right),\]
    and \[m\left(\phi^{-1}(B)\right)=a^{-p}m(B), ~\forall ~\text{measurable set}~ B \subset \left(0,\infty\right).\]
\end{corollary}

\chapter{Isometries on  quasi-Banach Marcinkiewicz sequence spaces and weak \texorpdfstring{$\ell_p$}{lp}-ideals, \texorpdfstring{$0< p\le 1$}{0 < p <= 1}}
\chaptermark{Isometries on Marcinkiewicz sequence spaces and weak ideals} \label{s5}
 Let $\Phi:[0,\infty)\to [0,\infty)$ be admissible and $\Phi(\infty):=\lim_{t\to\infty}\Phi(t)=\infty$.
In this chapter, we establish the general form of surjective isometries on  quasi-Banach Marcinkiewicz sequence spaces and weak $\ell_p$ operator ideals, $0<p\le1$. 
We point out that in this chapter, $T$ is not assumed to be positive and hence, approaches used in  Chapter \ref{positive isometries}  are not applicable.

\section{Isometries on quasi-Banach Marcinkiewicz sequence spaces}


The goal of this section is to provide a complete description of a surjective isometry $T$ on a quasi-Banach Marcinkiewicz sequence space $m_\Phi$.
In this section, unless stated otherwise, assume that $T$ is a surjective isometry from $m_\Phi$ onto $m_\Phi$.
We denote by $\operatorname{extr}\left(B_{m_\Phi}\right)$ the set of all extreme points of the unit ball of $m_\Phi$.

\begin{lemma} \label{Sigma}
  Let $\Phi:[0,\infty)\to [0,\infty)$ be admissible and $\Phi(\infty):=\lim_{t\to\infty}\Phi(t)=\infty$.
    Let 
    \[
    \Sigma:=\left\{w\in \operatorname{extr}\left(B_{m_\Phi}\right):w=\left(w_n\right)_{n=0}^\infty,~w_n\ne0 \text{ for all }n\ge0\right\}.\]
    Then, for any nonzero $x\in m_\Phi$, there exist $y,z\in\Sigma$ such that
    $x=\frac{c(y+z)}{2}$ for some $c> 0$.
\end{lemma}
\begin{proof}
 Fix  $0\ne x\in m_\Phi$.
 We claim that there exists a permutation $\sigma$ of the set $\{0,1,2,\dots\}$ such that
    \[|x\left(\sigma(n)\right)|\le \frac c{ l_\Phi(n)}, ~\forall n\ge0\]
    for some $c> 0$, where  $l_{\Phi}(n): =\Phi\left((n+1)_-\right) = \lim_{t\uparrow n+1} \Phi(t), ~\forall n\in \mathbb{N}$.
    \begin{enumerate} [label=(\roman*)]
    \item Assume that the support of $x$ is finite.
 Choose a suitable permutation  $\sigma$ such that    $x(\sigma)$ is of the following form 
   \[  ( a_0,\dots, a_M, 0,0,\dots     ) \]
where $|a_i|=\mu(i;x)$.
In this case, we 
     take $c=\norm{x}_{m_\Phi}$.
    \item Assume that both the support of $x$ and its complement are infinite. Choose a suitable permutation  $\sigma$ such that    $x(\sigma)$ is of the following form 
   \[  ( a_0, 0, a_1,0,\dots, a_i, 0,a_{i+1},0,\dots     ) \]
where $|a_i|=\mu(i;x)$.
  Since $\Phi$ is admissible, it follows that there exists $C_\Phi\ge1$ such that \[\Phi(2t)\le C_\Phi \Phi(t), ~\forall t\ge0.\]
    This implies that
     \[\frac{l_\Phi(2n)}{l_\Phi(n)}
     \le \frac{l_\Phi(2n+1)}{l_\Phi(n)}
     =
    \lim_{t\uparrow n+1}\frac{\Phi(2t)}{\Phi(t)}\le C_\Phi.\]
     Then,
    \[\mu(n;x)=|x(\sigma(2n))|\le \frac{\norm{x}_{m_\Phi}}{l_\Phi(n)}\le \frac{C_\Phi\norm{x}_{m_\Phi}}{l_\Phi(2n)}\]
    and so,
    \[|x(\sigma(2n))|l_\Phi(2n)\le\norm{x}_{m_\Phi}\frac{l_\Phi(2n)}{l_\Phi(n)}\le C_\Phi\norm{x}_{m_\Phi}, ~\forall n\in \mathbb{N}.\]
    In this case, we can take $c=C_\Phi\norm{x}_{m_\Phi}$ and we have 
    \[ | x(\sigma (n))|l_\Phi(n)  \le \sup_{i\ge 0} \mu(i;x)l_\Phi(2i)  \le \sup_{i\ge 0} \frac{\norm{x}_{m_\Phi}}{l_{\Phi}(i)} l_\Phi(2i) \le c.\]
    \item Assume that the complement  of the support of $x$ has $m<\infty$ elements.
   By the $\Delta_2$-condition, there exists $C_\Phi>0$ such that 
   \begin{align} \label{mlen+1}
    \frac{l_\Phi(m+n)}{l_\Phi(n)} \le \frac{l_\Phi(2n+1)}{l_\Phi(n)}=\frac{\lim_{t \uparrow 2n+2}\Phi(t)}{\lim_{t \uparrow n+1}\Phi(t)}=\lim_{t \uparrow n+1}\frac{\Phi(2t)}{\Phi(t)}
    \le C_\Phi
    \end{align}
    when $m\le n+1$.
 Choose a suitable permutation  $\sigma$ such that    $x(\sigma)$ is of the following form 
   \[  ( 0,\dots, 0,a_0,a_1,a_2,\dots     ) ,\]
    where $|a_i|=\mu(i;x)$.
Then, for each $n$ such that $m\le n+1$, we have
\[\mu(n;x)=|x(\sigma(m+n))|\le \frac{\norm{x}_{m_\Phi}}{l_\Phi(n)}\stackrel{\eqref{mlen+1}}{\le}\frac{C_\Phi\norm{x}_{m_\Phi}}{l_\Phi(m+n)}\]
and so,
\[|x(\sigma(m+n))|l_\Phi(m+n)\le\norm{x}_{m_\Phi}\frac{l_\Phi(m+n)}{l_\Phi(n)}\le C_\Phi\norm{x}_{m_\Phi}.\]
This shows that

\begin{align*} 
&~\quad | x(\sigma (n))|l_\Phi(n) \\
&\le \max \left\{   \sup_{  n+1 \ge m} \mu(n ; x) l_\Phi(m+n),   \sup_{  n+1 <  m} \mu(n ; x) l_\Phi(m+n) \right\}  \\
&\le c:=\max\left\{ C_\Phi\norm{x}_{m_\Phi}, \frac{l_{\Phi}(2m)}{l_\Phi(0)}\norm{x}_{m_\Phi} \right\}.
\end{align*}
    \end{enumerate}
    This proves the claim.
    Let $a_n:=\frac{x(\sigma(n))}{cl_\Phi(n)^{-1}}$.
    Observe that $|a_n|\le 1$.
    Then, there exist $b_n,c_n\in \mathbb{T}$ ($|b_n|=|c_n|=1$) such that
    \[a_n=\frac{b_n+c_n}{2}.\]
     In the special case that $|a_n|=1$ for all $n$, we may choose $b_n=c_n=a_n$ for all $n$.
Let
\[y_{\sigma(n)}:=\frac{b_n}{l_\Phi(n)},~ z_{\sigma(n)}:=\frac{c_n}{l_\Phi(n)}.\]
    Note that $x_n=\frac{c(y_n+z_n)}{2}$ and $y:=\left(y_n\right)_{n=0}^\infty, z:=\left(z_n\right)_{n=0}^\infty\in \Sigma$ 
    (see Theorem~\ref{atomic}).
    This implies that $x=\frac{c(y+z)}{2}$.
\end{proof}

The following theorem presents a complete description of surjective isometries of quasi-Banach Marcinkiewicz sequence spaces $m_\Phi$.

\begin{theorem}\label{thm:isom on Marcin seq} 
Let $\Phi:[0,\infty)\to [0,\infty)$ be admissible and $\Phi(\infty):=\lim_{t\to\infty}\Phi(t)=\infty$.
Let $T$ be a  surjective linear mapping on a quasi-Banach Marcinkiewicz sequence space $m_\Phi$.
    Then, $T$ is an isometry if and only if there exist a permutation $\sigma$ of the set $\{0,1,2,\dots\}$ and  a unimodular  sequence $\{u_k\}_{k\ge0}$ such that 
    \[T\left(x\right)\left(n\right)=u_n x\left(\sigma\left(n\right)\right),~\forall x\in m_\Phi.\]
\end{theorem}
\begin{proof}
  The implication $\Leftarrow$ is trivial.

    $(\Rightarrow).$
    Let $\{e_n\}_{n=0}^\infty$ be the standard basis of $c_0$. 
  Fix an arbitrary   $i\in \mathbb{N}$.
    Let the linear functional $\phi_i:m_\Phi\to \mathbb{C}$ be  defined by
    \[\phi_i(x)=\delta_i\circ T(x)=\left(T(x)\right)_i, ~\forall x\in m_\Phi.\]
    \begin{enumerate} [label=(\roman*)]
    \item Assume that $\phi_i(e_k)\ne0$ for some $k\in \mathbb{N}$.
    Fix $y=\left(y_n\right)_{n=0}^\infty\in \Sigma$ and $\lambda\in \mathbb{T}$.
    Let
    \[y(\lambda):=y+(\lambda-1)y_ke_k\in \Sigma, ~a:=\phi_i\left(y-y_ke_k\right), ~b:=y_k\phi_i(e_k)\ne 0.\]
    Note that $\phi_i\left(y(\lambda)\right)=a+\lambda b$.
    Since $T$ is a surjective isometry and $y(\lambda)\in \operatorname{extr}\left(B_{m_\Phi}\right)$, it follows that $T(y(\lambda))\in \operatorname{extr}\left(B_{m_\Phi}\right)$ and so, \[|a+\lambda b|=|\phi_i(y(\lambda))|\in \left\{0,\frac1{l_\Phi(0)},\frac{1}{l_\Phi(1)},\dots\right\}.\]
    Note that $\lambda\mapsto |a+\lambda b|$ is continuous on $\mathbb{T}$ and $\mathbb{T}$ is a connected set.
    This implies that $\{|a+\lambda b|:\lambda\in\mathbb{T}\}$ is a connected set and so, $|a+\lambda b|$ is a constant which is independent of $\lambda$.
    Since $|a+\lambda b|^2=|a|^2+|b|^2+2{\rm Re}\left(a\overline{\lambda b}\right)$ and $b\ne0$, it follows that for the chosen $(i,k)$, we have  
    \[a=\phi_i\left(y-y_ke_k\right)=0.\]
    This shows that 
    \[\phi_i(y)=\phi_i(e_k)y_k=\phi_i(e_k)\delta_k(y), ~\forall y\in\Sigma,\]
    that is, $\phi_i=\phi_i(e_k)\delta_k$ on $\Sigma$.
    Since  any element in $m_\Phi$ is a linear combination of elements in $\Sigma$ (see Lemma \ref{Sigma}),
    we obtain that $\phi_i=\phi_i(e_k)\delta_k$ on $m_\Phi$.
    \item Assume that $\phi_i(e_k)=0$ for any $k\in \mathbb{N}$.
    Let $0\ne y=T^{-1}(e_i)$.
    There exists $j$ such that $y_j\ne0$.
    Let 
    \begin{align} \label{delta}
    \psi_j(x)=\delta_j\circ T^{-1}(x)=\left(T^{-1}(x)\right)_j, ~\forall x\in m_\Phi.
    \end{align}
    Applying (i) to $T^{-1}$ and $j$ (with $\psi_j(e_i)\ne 0$), we obtain that
    \[\psi_j=\psi_j(e_i)\delta_i ~\text{on}~ m_\Phi.\]
    Then,
    \[1=\delta_j\circ T^{-1}\left(T(e_j)\right)\stackrel{\eqref{delta}}{=}\psi_j\left(T(e_j)\right)=\psi_j(e_i)\delta_i\left(T(e_j)\right)=\psi_j(e_i) \phi_i(e_j) =0,\]
    a contradiction.
    \end{enumerate}
    Consequently, for each $i$, there exists $k$ such that 
    \[\phi_i=\phi_i(e_k)\delta_k ~\text{on}~ m_\Phi.\]
    That is,
    \begin{align}\label{Txiek}
    T(x)(i)=\phi_i(e_k)x_k, ~\forall x\in m_\Phi.
    \end{align}
    Let the mapping $\sigma:i\mapsto k$.
    We claim that $\sigma$ is well-defined.
    Indeed, otherwise, there exists $m\ne k$ such that
    \[T(x)(i)=\phi_i(e_k)x_k=\phi_i(e_m)x_m, ~\forall x\in m_\Phi.\]
    Choosing  $x=e_k$,   $0\ne\phi_i(e_k)=\phi_i(e_k)(e_k)_k=\phi_i(e_m)(e_k)_m=0$, a contradiction.

    Next, we will prove that $\sigma$ is a bijection.
    Since $T$ is a bijection, it follows that $T(e_k)\ne0$ and so, there exists $i$ such that $(T(e_k))_i\ne0$.
    This implies that $\sigma$ is a surjection.
    It suffices to prove that $\sigma$ is an injection on $\mathbb{N}$.
    Assume that $\sigma(i)=\sigma(j)=k$.
By \eqref{Txiek}, we have 
    \[T(x)(i)=\phi_i(e_k)x_k, ~T(x)(j)=\phi_j(e_k)x_k, ~\forall x\in m_\Phi.\]
    Choose $x$ such that $T(x)=e_i$.
    In this case, $\phi_i(e_k)x_k=1$ and so, $x_k\ne0$.
    This implies that $e_i(j)=T(x)(j)\ne0$ and hence, $i=j$.
    In particular, $\sigma$ is an injection.
    
    Therefore, $\sigma$ is a permutation of the set $\{0,1,2,\dots\}$.
    Moreover, 
    \[T(x)(n)=\phi_n(e_{\sigma(n)})x(\sigma(n)), ~\forall x\in m_\Phi.\]
    Since $T$ is an isometry, it follows that
\[l_\Phi(0)=\norm{e_{\sigma(n)}}_{m_\Phi}=\norm{T(e_{\sigma(n)})}_{m_\Phi}=|\phi_n(e_{\sigma(n)})|l_\Phi(0)\]
and so, $|\phi_n(e_{\sigma(n)})|=1$.
    Let $u_n:=\phi_n\left(e_{\sigma(n)}\right)$,
    which completes the proof.
\end{proof}

Applying Theorem~\ref{thm:isom on Marcin seq} to $\Phi(t)=t^{1/p}$, we have the following result.
\begin{corollary} 
    Assume that $0<p<\infty$. 
    Let $T$ be a  surjective linear operator on a weak $\ell_p$-sequence space $\ell_{p,\infty}$. 
Then, $T$ is an isometry if and only if there exist a permutation $\sigma$ of the set $\{0,1,2,\dots\}$ and  a  unimodular sequence $\{u_k\}_{k\ge0}$ such that 
    \[T\left(x\right)\left(n\right)=u_n x\left(\sigma\left(n\right)\right),~\forall x\in \ell_{p,\infty}.\]
\end{corollary}

\section{Isometries on weak \texorpdfstring{$\ell_p$}{lp}-ideals}


The main purpose of this section is to characterize a surjective isometry $T$ on $\mathcal{C}_{p,\infty}$, $0<p\le1$.
Throughout this section, unless stated otherwise, we always assume that $T$ is a surjective isometry from $\mathcal{C}_{p,\infty }$ onto $\cC_{p,\infty}$.

We start with the following useful  observation.

\begin{prop}
       For each $x\in \ell_{p,\infty} $, $0<p\le1 $, there exist  $z,w$   such that 
    \[x=z-w,\]
    where $\mu\left(n;z\right)=\mu\left(n;w\right)=c_1\left(1+n\right)^{-1/p}$,   $c_1>\frac{1}{2}\norm{x}_{\ell_{p,\infty} }$. 
\end{prop}  
  
\begin{proof}
For any $n\in \mathbb{N}$, let $e_n$ denote the sequence with $1$ in the $n$-th entry and zeros elsewhere.
There exists  a mutually orthogonal sequence  $\{e_n\}_{n=0}^\infty$ of $\ell_\infty$ such that 
    \begin{align}\label{xform}
        x=\sum_{n=0}^\infty \lambda_n e_{n},
    \end{align}
    where $|\lambda_n|=\mu\left(n;x\right)  <2c_1\left(1+n\right)^{-1/p}$ for some $c_1>0$. 
    Let $n\ge0$ be given. 
    Denote by $a_n$ a point of intersection of the circle centered at $0$ with radius $c_1\left(n+1\right)^{-1/p}$ and the circle centered at $c_1\left(n+1\right)^{-1/p}$ with radius $ \mu(n;x)$. 
    Then, 
    \[|a_n|=c_1\left(n+1\right)^{-1/p}, \quad|a_n-c_1\left(n+1\right)^{-1/p}|= \mu(n;x).\]
    Let $\theta_n\in [0,2\pi)$ be such that 
    \begin{align}\label{rotate}
        e^{i\theta_n}\left(a_n-c_1\left(n+1\right)^{-1/p}\right)= \lambda_n.
    \end{align}
    Set 
    \[z=\sum_{n=0}^\infty e^{i\theta_n}a_ne_{n},~ w=\sum_{n=0}^\infty e^{i\theta_n}c_1\left(n+1\right)^{-1/p} e_{n}.\] 
    Then, we have 
    \begin{align*}
    z-w
    =\sum_{n=0}^\infty e^{i\theta_n}\left(a_n-c_1\left(n+1\right)^{-1/p}\right) e_{n}
    \stackrel{\eqref{rotate}}{=}\sum_{n=0}^\infty \lambda_n e_{n}
    \stackrel{\eqref{xform}}{=}x
    \end{align*}
and 
$\mu(n;z)=|e^{i\theta_n}a_n|=c_1\left(n+1\right)^{-1/p}=\mu(n;w),~n\ge0$. 
\end{proof}
 
\begin{corollary} \label{linear combination}  
Let $0<p\le1 $.  
Let $\mathcal{M} $   be a semifinite atomic von Neumann algebra with all atoms having the same trace equal to one. 
     For each $x\in L_{p,\infty}(\cM,\tau) $, there exist  $z,w$   such that 
    \[x=z-w,\]
    where $\mu\left(n;z\right)=\mu\left(n;w\right)=c\left(1+n\right)^{-1/p}, ~\forall 0\le n<\tau(\mathbf{1})$,   $c>\frac{1}{2}\norm{x}_{L_{p,\infty}(\cM,\tau) }$. 
\end{corollary}

Let the inner product $\langle \cdot,\cdot \rangle_2$ on $\cC_2$ be defined by
\[\langle x,y\rangle_2=\tau(y^*x), ~\forall x,y \in \mathcal{C}_2.\]
We denote by $\operatorname{extr}\left(B_{\mathcal{C}_{p,\infty}}\right)$ the set of all extreme points of the unit ball of $\cC_{p,\infty}$.

Let $\varphi$ and $\psi$ be two given vectors in $H$. 
The linear operator $\varphi \otimes \psi$ is defined by
$$\left(\varphi \otimes \psi\right)h=\left(h,\psi\right)\varphi, ~ \forall h\in H.$$

\begin{lemma} \label{constant}
     Assume that $0<p\le1$.
     Let $e_0,f_0, e_1,f_1\in H$ be such that $\norm{e_0}_H\norm{f_0}_H=\norm{e_1}_H\norm{f_1}_H=1$ and $\left(e_0\otimes f_0\right)\perp \left(e_1\otimes f_1\right)$.
     Then, 
     \[\norm{T(e_0\otimes f_0)}_2=\norm{T(e_1\otimes f_1)}_2<\infty\]
     and
     \[\left\langle T\left(e_0\otimes f_0\right),T\left(e_1\otimes f_1\right)\right\rangle_2=0.\]
     In particular, 
     \[\norm{T(R)}_2=c  ~\text{for any rank-one operator} ~R ~\text{with} ~\norm{R}_{B(H)}=1,\]
     for some constant $c>0$.
\end{lemma}
\begin{proof}
    Let 
    \[x=\sum_{n=2}^\infty (n+1)^{-1/p}e_n\otimes f_n,\]
    where $\{e_n\otimes f_n\}_{n=2}^\infty$ is a mutually orthogonal sequence of rank-one operators with uniform norm $1$ satisfying $\left(e_n\otimes f_n \right)\perp \left(e_i\otimes f_i\right)$ for all $n\ge2$ and $i=0,1$.
    For any $\lambda,\eta\in \mathbb{T}$, let
    \[y_{\lambda,\eta}:=x+\lambda e_0\otimes f_0+\eta 2^{-1/p} e_1\otimes f_1 \in \operatorname{extr}\left(B_{\mathcal{C}_{p,\infty}}\right).\]
    Since $T$ is a surjective isometry, it follows that $T(y_{\lambda,\eta})\in \operatorname{extr}\left(B_{\mathcal{C}_{p,\infty}}\right)$.
    Then,
    \[\left\langle T(y_{\lambda,\eta}),T(y_{\lambda,\eta})\right\rangle_2=\norm{T(y_{\lambda,\eta})}_2^2=\sum_{n=0}^\infty (n+1)^{-2/p}.\]
    That is,
    \begin{align*} 
    &\quad \bigl\langle T(x)+\lambda T(e_0\otimes f_0)+\eta 2^{-1/p}T(e_1\otimes f_1),\\
    &\quad \quad T(x)+\lambda T(e_0\otimes f_0)+\eta 2^{-1/p}T(e_1\otimes f_1)\bigr\rangle_2\\
    &=\norm{T(x)}_2^2+\norm{T(e_0\otimes f_0)}_2^2+\norm{2^{-1/p}T(e_1\otimes f_1)}_2^2+2{\rm Re}\left\langle  T(x),\lambda T(e_0\otimes f_0) \right\rangle_2\\
    &~+2{\rm Re}\left\langle  T(x),\eta 2^{-1/p} T(e_1\otimes f_1) \right\rangle_2+2{\rm Re}\left\langle \lambda T(e_0\otimes f_0),\eta2^{-1/p}T(e_1\otimes f_1) \right\rangle_2.
    \end{align*}
    Since $\lambda, \eta\in \mathbb{T}$ are arbitrary,
    it follows that
    \begin{align} \label{orthogonal}
    \left\langle T(x),T\left(e_0\otimes f_0\right)\right\rangle_2&=\left\langle T(x),T\left(e_1\otimes f_1\right)\right\rangle_2\nonumber\\
    &=\left\langle T\left(e_0\otimes f_0\right),T\left(e_1\otimes f_1\right)\right\rangle_2=0.
    \end{align}
    Since \[\begin{aligned}
&\quad\norm{T(x)+T\left(e_0\otimes f_0\right)+T\left(2^{-1/p}e_1\otimes f_1\right)}_2\\
&=\norm{T(x)+T\left(e_1\otimes f_1\right)+T\left(2^{-1/p}e_0\otimes f_0\right)}_2,\end{aligned}\]
it follows from \eqref{orthogonal} that
    \[\norm{T\left(e_0\otimes f_0\right)}_2^2+2^{-2/p}\norm{T\left(e_1\otimes f_1\right)}_2^2=\norm{T\left(e_1\otimes f_1\right)}_2^2+2^{-2/p}\norm{T\left(e_0\otimes f_0\right)}_2^2.\]
    This shows that
    \[\norm{T\left(e_0\otimes f_0\right)}_2=\norm{T\left(e_1\otimes f_1\right)}_2.\]
\end{proof}

\begin{lemma} \label{exten}
    Assume that $0<p\le 1$.
    For any $x,y\in\mathcal{F}\left(B(H),\tau\right)$, 
    \[\left\langle T(x),T(y)\right\rangle_2=c^2\left\langle x,y\right\rangle_2,\]
    where $c$ is as defined in Lemma \ref{constant}.

    In particular,  the operator 
    \begin{align} \label{u0}
    U_0x:=c^{-1}T(x), ~\forall x\in \mathcal{F}\left(B(H),\tau\right),
    \end{align}
   can be extended uniquely to an isometry $U_1:\mathcal{C}_2\to \mathcal{C}_2$ (here, we do not claim that $U_1$ is surjective).
\end{lemma}

\begin{proof}
For each $z\in \cF(B(H),\tau)$, let
\[
z=\sum_{i=1}^m \alpha_iR_i,
\]
where $\{R_i\}_{i=1}^m$ is a sequence of mutually orthogonal rank-one operators with uniform norm $1$.
By Lemma~\ref{constant}, we have that $\norm{T(R_i)}_2=c$ and 
$\{T(R_i)\}_{i=1}^m$ is mutually orthogonal with respect to $\langle \cdot,\cdot\rangle_2$,
which implies that
\[
\norm{T(z)}_2^2=c^2\sum_{i=1}^m|\alpha_i|^2=c^2\norm{z}_2^2.
\]
By the polarization identity, the proof is complete.
\end{proof}

\begin{lemma} \label{C p, infty}
    Assume that $0<p\le1$.
    Let $U_1$ and $c$ be as defined in Lemma \ref{exten} and Lemma \ref{constant}, respectively.
    Then,
    \[U_1^*T(x)=cx, ~\forall x\in \mathcal{C}_{p,\infty}.\]
\end{lemma}
\begin{proof}
    Fix $x\in \operatorname{extr}\left(B_{\mathcal{C}_{p,\infty}}\right)$ and $e\otimes f$ with $\norm{e\otimes f}_{B(H)}=1$. 
    Choose a Schmidt entry $e_1\otimes f_1$ of $x$ with $\norm{e_1\otimes f_1}_{B(H)}=1$.
    There exist unitary operators $U,V$ satisfying $U-\mathbf{1},V-\mathbf{1}$ having finite rank such that
    \[Ue_1=e, ~V^*f_1=f.\]
    Let $y:=UxV$.
    Then, $y-x=(U-\mathbf{1})xV+x(V-\mathbf{1})$ has finite rank and
    \[y\in \operatorname{extr}\left(B_{\mathcal{C}_{p,\infty}}\right), ~ y=e\otimes f+ z,\]
    where $z\perp \left(e\otimes f\right)$.
    For any $\lambda\in \mathbb{T}$, we have that
    \[z+\lambda e\otimes f\in\operatorname{extr}\left(B_{\mathcal{C}_{p,\infty}}\right)\]
    and so,
    \[\norm{T(z)+\lambda T(e\otimes f)}_2^2=\sum_{n=0}^\infty(n+1)^{-2/p}.\]
    Arguing similarly as \eqref{orthogonal},  we obtain
    \[\left\langle T(z),T(e\otimes f)\right\rangle_2=0.\]
    Note that
    \[\left\langle T(y),T(e\otimes f)\right\rangle_2=\left\langle T(e\otimes f),T(e\otimes f)\right\rangle_2\stackrel{\tiny\mbox{Lem. \ref{constant}}}{=}c^2=c^2\left\langle y,e\otimes f\right\rangle_2.\]
    Then,
    \begin{eqnarray*}
    \left\langle T(x),T(e\otimes f)\right\rangle_2
    &=&\left\langle T(y),T(e\otimes f)\right\rangle_2+\left\langle T(x-y),T(e\otimes f)\right\rangle_2\\
    &\stackrel{\tiny\mbox{Lem.~\ref{exten}}}{=}&c^2\left\langle y,e\otimes f\right\rangle_2+c^2\langle x-y, e\otimes f\rangle_2\\
    &=&c^2\left\langle x,e\otimes f\right\rangle_2.
    \end{eqnarray*}
    For any $w\in \mathcal{F}\left(B(H),\tau\right)$,
    we have that
    \[\left\langle T(x),T(w)\right\rangle_2=c^2\left\langle x,w\right\rangle_2.\]
    By Corollary \ref{linear combination}, we obtain that for any $x\in \mathcal{C}_{p,\infty}$,
    \[\left\langle T(x),T(w)\right\rangle_2=c^2\left\langle x,w\right\rangle_2, ~\forall w\in \mathcal{F}\left(B(H),\tau\right).\]
    This implies that
    \[\left\langle U_1^*T(x),w\right\rangle_2\stackrel{\eqref{u0}}{=}\frac1c\left\langle T(x),T(w)\right\rangle_2=c\left\langle x,w\right\rangle_2, ~\forall w\in \mathcal{F}\left(B(H),\tau\right).\]
    By the density of $\mathcal{F}\left(B(H),\tau\right)$ in $\mathcal{C}_2$, we have that
    \[\left\langle U_1^*T(x),y\right\rangle_2=c\left\langle x,y\right\rangle_2, ~\forall x\in \mathcal{C}_{p,\infty},~y\in \mathcal{C}_2.\]
    In particular,
    \[U_1^*T(x)=cx, ~\forall x\in \mathcal{C}_{p,\infty}.\]
\end{proof}

\begin{lemma} \label{the second}
    Assume that $0<p\le1$. Then, 
    \begin{enumerate} [label=(\roman*)]
    \item $U_1$ is a unitary operator on $\mathcal{C}_2$.
    \item $c=1$.
    \item $T(x)=U_1x, ~\forall x\in \mathcal{C}_{p,\infty}$.
    \end{enumerate}
\end{lemma}
\begin{proof}
    Applying Lemma \ref{constant} and Lemma \ref{exten} to $T^{-1}$, we obtain that
    \begin{align} \label{T-1}
    \norm{T^{-1}(x)}_2=d\norm{x}_2, ~\forall x\in \mathcal{F}\left(B(H),\tau\right)
    \end{align}
    for some constant $d>0$.
    Since $U_1$ is an isometry, it follows that $U_1^*U_1=\mathbf{1}$.
    Let $W:=U_1U_1^*$.
    For any $x\in \mathcal{F}\left(B(H),\tau\right)$, by Lemma \ref{C p, infty}, we have that
    \[U_1^*x=cT^{-1}(x).\]
    This implies that
    \[\left\langle Wx,x\right\rangle_2=\norm{U_1^*x}_2^2=c^2\norm{T^{-1}(x)}_2^2\stackrel{\eqref{T-1}}{=}c^2d^2\norm{x}_2^2.\]
    By the density of $\mathcal{F}\left(B(H),\tau\right)$ in $\mathcal{C}_2$, we obtain that \[W=\mathbf{1}_{\mathcal{C}_2}, ~c^2d^2=1.\]
    In particular, $U_1$ is a unitary operator on $\mathcal{C}_2$.
    By Lemma \ref{C p, infty}, we have that
    \[T(x)=cU_1x, ~\forall x\in \mathcal{C}_{p,\infty}.\]
    For any $x\in \operatorname{extr}\left(B_{\cC_{p,\infty}}\right)$,
    since $T$ is a surjective isometry, it follows from $$\norm{Tx}_2^2=\norm{x}_2^2=\sum_{n=0}^\infty(n+1)^{-2/p}$$
    that $c=1$, 
    that is,
    \[T(x)=U_1x, ~\forall x\in \mathcal{C}_{p,\infty}.\] 
\end{proof}

The K\"{o}the dual $\mathcal{C}_{p,\infty}^\times$ of $\mathcal{C}_{p,\infty}$ is defined by setting \cite[Section 4.3]{DPS}
\begin{align*}
\mathcal{C}_{p,\infty}^\times
&=\{y\in B(H):\sup\{|\tau(xy)|:x\in \mathcal{C}_{p,\infty},\norm{x}_{\mathcal{C}_{p,\infty}}\le1\}<\infty\}\\
&=\{y\in B(H):\sup\{\tau(|xy|):x\in \mathcal{C}_{p,\infty},\norm{x}_{\mathcal{C}_{p,\infty}}\le1\}<\infty\},
\end{align*}
and for $y\in \mathcal{C}_{p,\infty}^\times$, the quantity $\norm{y}_{\mathcal{C}_{p,\infty}^\times}$ is defined by setting
\begin{align*}
\norm{y}_{\mathcal{C}_{p,\infty}^\times}
&=\sup\{|\tau(xy)|:x\in \mathcal{C}_{p,\infty},\norm{x}_{\mathcal{C}_{p,\infty}}\le1\}\\
&=\sup\{\tau(|xy|):x\in \mathcal{C}_{p,\infty},\norm{x}_{\mathcal{C}_{p,\infty}}\le1\}.
\end{align*}
By the triangle inequality,
 $\norm{\cdot}_{\mathcal{C}_{p,\infty}^\times}$ is a norm.



Moreover, for any $y\in \mathcal{C}_{p,\infty}^\times$,  we have (see, e.g., the proof of \cite[Proposition~4.5.1]{DPS})
\begin{equation}
\begin{aligned}
\norm{y}_{\mathcal{C}_{p,\infty}^\times}&= \sup\left\{\sum_{n=0}^\infty \mu(n;x)\mu(n;y):x\in \mathcal{C}_{p,\infty},\norm{x}_{\mathcal{C}_{p,\infty}}\le 1\right\}\\
&=\sum_{n=0}^\infty (n+1)^{-1/p}\mu(n;y).\label{def kothe}
\end{aligned}
\end{equation}
This shows that 
$\norm{\cdot}_{\mathcal{C}_{p,\infty}^\times}$ is a symmetric norm.
By the Cauchy--Schwarz inequality, we have 
\begin{equation} \label{strictly subset}
\begin{aligned}
\sum_{n=0}^\infty (n+1)^{-1/p}\mu(n;y)
&\le \left(\sum_{n=0}^\infty (n+1)^{-2/p}\right)^{1/2 }\left(\sum_{n=0}^\infty \mu(n;y)^2\right)^{1/2 }\\
&=\left(\sum_{n=0}^\infty (n+1)^{-2/p}\right)^{1/2 }\norm{y}_2.
\end{aligned}
\end{equation}
This implies that $\mathcal{C}_2\subset \mathcal{C}_{p,\infty}^\times$.

\begin{lemma} \label{final1}
    Assume that $0<p\le1$. Then,  \[\norm{U_1x}_{\mathcal{C}_{p,\infty}^\times}=\norm{x}_{\mathcal{C}_{p,\infty}^\times}, ~\forall x\in \mathcal{C}_2.\]
\end{lemma}
\begin{proof}
    Since $U_1$ is a unitary operator on $\mathcal{C}_2$, it follows that for all $x\in \mathcal{C}_2$, we have 
\begin{eqnarray*}
\norm{U_1x}_{\mathcal{C}_{p,\infty}^\times}
    &=&\sup_{\norm{y}_{\mathcal{C}_{p,\infty}\le 1}}\left|\tau(y U_1x)\right|\\
    &=&\sup_{\norm{y}_{\mathcal{C}_{p,\infty}\le 1}}\left|\langle U_1x,y\rangle_2\right|\\
    &=&\sup_{\norm{y}_{\mathcal{C}_{p,\infty}\le 1}}\left|\langle x,U_1^*y\rangle_2\right|\\
    &=&
    \sup_{\norm{y}_{\mathcal{C}_{p,\infty}\le 1}}\left|\tau((U_1^*y)^*x)\right|\\
    &\stackrel{\tiny\mbox{Lem.~\ref{the second}}}{=}&\sup_{\norm{y}_{\mathcal{C}_{p,\infty}}\le1}\left|\tau\left(\left(T^{-1}y\right)^*x\right)\right|\\
    &=&\norm{x}_{\mathcal{C}_{p,\infty}^\times}.
\end{eqnarray*}
\end{proof}

Now, we come to the  main result of this section.

\begin{theorem} \label{final n}
    Assume that $0<p\le1$. 
    Let $H$ be a complex separable infinite-dimensional Hilbert space. 
    Let $T$ be a surjective linear operator from weak $\ell_p$-ideal $\mathcal{C}_{p,\infty}$ onto $\mathcal{C}_{p,\infty}$. 
    Then, $T$ is an isometry if and only if there exist unitary operators $U,V$ such that one of the following two statements holds:
    \begin{enumerate} [label=(\roman*)]
    \item $T(x)=UxV^*, ~\forall x\in \mathcal{C}_{p,\infty}$;
    \item $T(x)=Ux'V^*, ~\forall x\in \mathcal{C}_{p,\infty}$, where $x'$ denotes the transpose of $x$. 
    \end{enumerate}
\end{theorem}
\begin{proof}
The implication 
    $\Leftarrow$  is trivial. 

$\left(\Rightarrow\right)$.
    By Lemmas  \ref{constant}-\ref{the second}, there exists a unitary operator $U_1$ on $\mathcal{C}_2$ such that
    \begin{align} \label{T U}
    T(x)=U_1x, ~\forall x\in \mathcal{C}_{p,\infty}.
    \end{align}
    Let $\mathcal{E}_p=\overline{\mathcal{F}\left(B(H),\tau\right)}^{\norm{\cdot}_{\mathcal{C}_{p,\infty}^\times}}$, which is an order continuous  symmetric operator ideal (see \cite[Definition 4.4.9, Lemma 4.4.10]{DPS}).
    Since $U_1$ is bijective on $\mathcal{C}_2$, it follows from Lemma \ref{final1} that $U_1$ is a surjective isometry on $\left(\mathcal{C}_2,\norm{\cdot}_{\mathcal{C}_{p,\infty}^\times}\right)$. 
    By \eqref{strictly subset}, we have that $\mathcal{C}_2\subset \mathcal{E}_p$.
    Then,
    \[\mathcal{F}\left(B(H),\tau\right)\subset \mathcal{C}_2\subset \mathcal{E}_p\subset \overline{\mathcal{C}_2}^{\norm{\cdot}_{\mathcal{C}_{p,\infty}^\times}},\]
    which implies that $\mathcal{E}_p= \overline{\mathcal{C}_2}^{\norm{\cdot}_{\mathcal{C}_{p,\infty}^\times}}$.
    Moreover, $\mathcal{C}_2\subsetneq \overline{\mathcal{C}_2}^{\norm{\cdot}_{\mathcal{C}_{p,\infty}^\times}}$.
    Indeed, let $\{e_k\}_{k=0}^\infty$ be 
   a sequence of mutually orthogonal rank-one projections of
    $B(H)$ and
    \[x_n:=\sum_{k=0}^n(k+1)^{-1/2}e_k\in  \mathcal{F}\left(B(H),\tau\right), ~x:=\sum_{k=0}^\infty(k+1)^{-1/2}e_k \stackrel{\eqref{def kothe}}{\in} \mathcal{E}_p.\]
    Note that $\norm{x_n-x}_{\mathcal{C}_{p,\infty}^\times}\to0$ as $n\to \infty$ and $x\notin \mathcal{C}_2$.
    Consequently, $\mathcal{C}_2\subsetneq \mathcal{E}_p$.
    Since $\mathcal{E}_p= \overline{\mathcal{C}_2}^{\norm{\cdot}_{\mathcal{C}_{p,\infty}^\times}}$, it follows that $\mathcal{C}_2$ is dense in $\mathcal{E}_p$.
    Then, $U_1$ can be uniquely extended to a surjective isometry $U_2$ on $\mathcal{E}_p$.
    Applying \cite[Theorem 2]{S81} to $U_2$, we obtain that there exist unitary operators $U,V$ on $H$ such that 
    \[U_2x=UxV^*, ~\forall x\in\mathcal{E}_p\]
    or
    \[U_2x=Ux'V^*, ~\forall x\in\mathcal{E}_p.\]
    Since $U_1=U_2|_{\mathcal{C}_2}$, it follows from
    \eqref{T U} that
    \[T(x)=UxV^*, ~\forall x\in \mathcal{C}_{p,\infty}\]
    or
    \[T(x)=Ux'V^*, ~\forall x\in \mathcal{C}_{p,\infty}.\]
    This completes the proof.
\end{proof}


\begin{remark}
    Let $\Phi:[0,\infty)\to [0,\infty)$ be admissible satisfying $\Phi(\infty):=\lim_{t\to \infty}\Phi(t)=\infty$.
    The result of Theorem \ref{final n} is also valid for quasi-Banach Marcinkiewicz ideals $M_\Phi\left(B(H),\tau\right)\subset \mathcal{C}_2$  by the same argument.
\end{remark}

\chapter[One-parameter groups of surjective isometries] {Some remarks concerning   one-parameter groups of  surjective isometries on weak \texorpdfstring{$\ell_p$}{lp} operator ideals}\label{s8}
Assume that $0<p<\infty$.
Let $X$ denote a complex  Banach (respectively, quasi-Banach) space equipped with the norm (respectively, the quasi-norm) $\norm{\cdot}_X$.
Recall that a family $\{T_t\}_{t\in \mathbb{R}}$ of bounded linear operators on $X$ is called a one-parameter group if
\begin{enumerate} [label=(\roman*)]
    \item $T_{s+t}=T_sT_t, ~\forall s,t\in \mathbb{R}.$
    \item $T_0=\mathbf{1}_{B(X)}$.
\end{enumerate}
If, in addition, $\norm{T_t(x)-x}_X \to0$ as $t\to 0$ for any $x\in X$, then $\{T_t\}_{t\in\mathbb{R}}$ is called a strongly continuous one-parameter group.
Moreover, $\{T_t\}_{t\in\mathbb{R}}$ is called a uniformly continuous one-parameter group if $\norm{T_t-\mathbf{1}_{B(X)}}_{B(X)}\to0$ as $t\to0$.
For more details, the reader is referred to \cite[$\mathrm{VIII}$.1]{DS}.
From now on, unless stated otherwise, we always assume that $\{T_t\}_{t\in\mathbb{R}}$ is a one-parameter group of isometries.

 
Recall that 
\begin{enumerate}
\item every strongly continuous one-parameter group of isometries on $B(H)$ is continuous in the uniform operator topology of $B\left(B(H)\right)$ \cite[Theorem 2.8]{BFJ}.
\item $\{T_t\}_{t\in \mathbb{R}}$ is continuous in the uniform operator topology of $B\left(B(H)\right)$
if and only if
there exist bounded self-adjoint operators $A,B$ such that 
\[T_t(x)=e^{itA}xe^{itB},~ \forall t\in \mathbb{R}, ~x\in B(H)\]
(see \cite[Theorem 2.6]{BFJ}).
\item $\{T_t\}_{t\in \mathbb{R}}$ is a strongly continuous one-parameter group of isometries of Schatten~$p$-class  $\mathcal{C}_p$ for $1\le p\le\infty, ~p\ne2$ 
if and only if
there exist self-adjoint (not necessarily bounded) operators $A,B$ such that 
\begin{equation*}  
T_t(x)=e^{itA}xe^{itB},~ \forall t\in \mathbb{R}, ~x\in \mathcal{C}_p
\end{equation*}
(see \cite[Theorem]{BFGJ}).
\item Let $\left(\cC_E,\nu\right)\ne \cC_2$ be a minimal norm ideal. 
Then, $\{T_t\}_{t\in \mathbb{R}}$ is a strongly continuous one-parameter group of isometries of  $\mathcal{C}_E$ 
if and only if
there exist self-adjoint (not necessarily bounded) operators $A,B$ such that 
\begin{equation*}  
T_t(x)=e^{itA}xe^{itB},~ \forall t\in \mathbb{R}, ~x\in \mathcal{C}_E
\end{equation*}
(see \cite[Proposition 2]{S81}).
\item Let $E\ne \ell_2$ be a complex Banach sequence space such that the standard unit vectors $\{e_n\}_{n=0}^\infty$, defined by $e_n(k)=\delta_{n,k}$, form a $1$-symmetric basis for $E$.
Then, 
$\{T_t\}_{t\in \mathbb{R}}$ is a uniformly continuous one-parameter group of isometries of  $E$ 
if and only if
there exists a real sequence $A=\left(a_n\right)_{n=0}^\infty\in \ell_\infty$ such that 
\begin{equation*}  
T_t(x)(n)=e^{ita_n}x(n),~ \forall t\in \mathbb{R}, ~n\in\mathbb{N}, ~x\in E
\end{equation*}
(see \cite[Theorem 2]{A85}).
\end{enumerate}

 As is well known, $B(H)$ and $\cC_{p,\infty},0<p<\infty$, are non-separable quasi-Banach spaces, which   have the Fatou property (this follows from \cite[Proposition 3.2.14]{DPS}). 
Below, we demonstrate that  the situation of $\cC_{p,\infty}$ is similar to that of $ B(H)$, which is  completely different from  case of separable operator ideals \cite{S81}.

It is noteworthy that self-adjoint operators $A,B$, obtained below, are bounded on $H$ (see Theorem \ref{8.2}).
The following result immediately follows from Theorem \ref{final} in combination with a similar argument used in the proof of \cite[p. 152-153]{BFJ}.
For the sake of completeness, we provide a proof here.
\begin{prop} \label{8.1}
    Assume that $0<p<\infty$.
    Let $H$ be a complex separable infinite-dimensional Hilbert space.
    If $\{T_t\}_{t\in \mathbb{R}}$ is a strongly continuous one-parameter group of isometries on $\cC_{p,\infty}$,
    then there exist two  self-adjoint operators $A,B$ on $H$ such that
    $$T_t(x)=e^{itA}xe^{itB}, ~\forall t\in\mathbb{R}, ~x\in \cC_{p,\infty}.$$
    Here, we do not claim the boundedness of $A$ and $B$. 
\end{prop}
\begin{proof}
    For any $t\in \mathbb{R}$, by Theorem \ref{final}, there exist two unitary operators $U_t,V_t\in B(H)$ such that one of the following two statements holds:
    \begin{enumerate} [label=(\roman*)]
    \item $T_t(x)=U_txV_t, ~\forall x\in \mathcal{C}_{p,\infty}$;
    \item $T_t(x)=U_tx'V_t, ~\forall x\in \mathcal{C}_{p,\infty}$, where $x'$ denotes the transpose of $x$ with respect to some orthonormal basis $\{e_n\}_{n\in\mathbb{N}}$ of $H$. 
    \end{enumerate}

    We claim that only Case (i) above holds.
    Indeed, assume by contradiction that  there exists $t_0\in \mathbb{R}$ such that for any $x\in \cC_{p,\infty}$,
    $$T_{t_0}(x)=U_{t_0}x'V_{t_0}.$$
 By Theorem \ref{final},   one of the following holds, 
    \begin{align} \label{1 1}
    T_{\frac{t_0}2}(x)=U_{\frac{t_0}2}xV_{\frac{t_0}2}, ~\forall x\in \cC_{p,\infty}
    \end{align}
    or
    \begin{align} \label{1 2}
    T_{\frac{t_0}2}(x)=U_{\frac{t_0}2}x'V_{\frac{t_0}2}, ~\forall x\in \cC_{p,\infty}.
    \end{align}
    We first assume that \eqref{1 2} holds.
    Since $T_{t_0}=T_{\frac{t_0}2}^2$, it follows that for all $x\in\cC_{p,\infty}$,
    $$U_{t_0}x'V_{t_0}=U_{\frac{t_0}2}\left(U_{\frac{t_0}2}x'V_{\frac{t_0}2}\right)'V_{\frac{t_0}2}=U_{\frac {t_0}2} V_{\frac{t_0}2}'xU_{\frac{t_0}2}'V_{\frac{t_0}2}$$
    and so,
    $$x'=\left(U_{t_0}^*U_{\frac {t_0}2} V_{\frac{t_0}2}'\right)x\left(U_{\frac{t_0}2}'V_{\frac{t_0}2}V_{t_0}^*\right)$$
    Let $C_{t}:=U_{t}^*U_{\frac{t}2}V_{\frac{t}2}'$, $D_{t}:=U_{\frac{t}2}'V_{\frac{t}2}V_{t}^*$, $t\in \mathbb{R}$. Then, we have
    \begin{align} \label{xtx}
    x'=C_{t_0}xD_{t_0}.
    \end{align}
    Take $x=e\otimes f$ with $\norm{e\otimes f}_{B(H)}=1$.
   We define 
    the conjugation $C:H\to H$ by
    $$Ch=\sum_{n=0}^\infty\overline{\langle h,e_n\rangle}_He_n, ~\forall h=\sum_{n=0}^\infty \langle h,e_n\rangle_H e_n \in H.$$
    We have that (see, e.g., \cite[p. 241--242]{E94})
    $$C^2=\mathbf{1}, ~Cx'C=x^*, ~\forall x\in \cC_{p,\infty}.$$
    This together with \eqref{xtx} shows that
    $$\left(C_{t_0}e\right)\otimes \left(D_{t_0}^*f\right)= \left(e\otimes f\right)'=C\left(e\otimes f\right)^*C =  C\left(f\otimes e\right)C=
    (Cf)\otimes (C^*e).$$
    Then, for any  $h\in H$, we have that
    $$(h,C^*e)Cf=\left(h,D_{t_0}^*f\right)C_{t_0}e.$$
    This implies that
    \begin{align} \label{Cauchy}
\left(D_{t_0}^*f,C^*e\right)Cf=\left(D_{t_0}^*f,D_{t_0}^*f\right)C_{t_0}e=\norm{D_{t_0}^*f}_H^2 C_{t_0}e.
    \end{align}
    By Cauchy--Schwarz inequality, we have that
    \begin{eqnarray*}
\norm{f}_H^2\norm{e}_H&=&\norm{D_{t_0}^*f}_H^2\norm{C_{t_0}e}_H\\
&\stackrel{\eqref{Cauchy}}{=}&\left|\left(D_{t_0}^*f,C^*e\right)\right|\norm{Cf}_H\\
&\le&\norm{D_{t_0}^*f}_H\norm{C^*e}_H\norm{Cf}_H\\
&=&\norm{e}_H\norm{f}_H^2.
    \end{eqnarray*}
    This implies that $D_{t_0}^*f$ and $C^*e$ are linearly dependent,
    which contradicts with the fact that $e,f$ are arbitrary such that $\norm{e\otimes f}_{B(H)}=1$.
    Now we assume that \eqref{1 1} holds.
    Arguing similarly, we get the contradiction.
    Consequently, the claim is proved.
    That is, for any $x\in \cC_{p,\infty}$ and $t\in \mathbb{R}$,
    \begin{align} \label{no or}
    T_t(x)=U_txV_t.
    \end{align}

    Since $T_{t+s}=T_tT_s$ for all $t,s\in\mathbb{R}$, it follows from \eqref{no or} that for all $x\in \cC_{p,\infty}$,
    $$U_{t+s}xV_{t+s}=U_tU_sxV_sV_t,$$
    that is,
    $$U_s^*U_t^*U_{t+s}x=xV_sV_tV_{t+s}^*.$$
    Let
    $$u_{t,s}^*:=U_s^*U_t^*U_{t+s}, ~ v_{t,s}:=V_sV_tV_{t+s}^*.$$
    Then, we have that
    \begin{align} \label{utvt}
        u_{t,s}^*x=xv_{t,s}, ~\forall x\in \cC_{p,\infty}.
    \end{align}
    Below, we will prove that there exists a unimodular scalar $w=w(s,t)$ such that
    \begin{align} \label{w1}
    u_{t,s}=w\mathbf{1}, ~v_{t,s}=\overline{w}\mathbf{1}.
    \end{align}
    By \eqref{utvt}, we have that
    $$u_{t,s}^*x\cdot \mathbf{1}=xv_{t,s}, ~\forall x\in \cC_{p,\infty}.$$
    This together with \cite[Lemma 2]{S81} yields that
    $u_{t,s}^*$ is a scalar or $v_{t,s}$ is a scalar.
    Without loss of generality, we may assume that $u_{t,s}=w\mathbf{1}$ for some scalar $w=w(s,t)$.
    This together with \eqref{utvt} yields that
    \begin{align} \label{all finite-rank}
    x(\overline w\mathbf{1}-v_{t,s})=0, ~\forall x\in \cC_{p,\infty}.
    \end{align}
    Fix $y\in B(H)$.
    Since the set of all finite-rank operators in  $H$ is dense in $B(H)$ with respect to the wo-topology, it follows that 
    there exists a net $\{y_\alpha\}_{\alpha\in \Lambda}$ of finite-rank operators in $H$ such that
    $$0\stackrel{\eqref{all finite-rank}}{=}\langle y_\alpha(\overline w\mathbf{1}-v_{t,s})\xi,\eta\rangle_H\to \langle y(\overline w\mathbf{1}-v_{t,s})\xi,\eta\rangle_H, ~\forall \xi,\eta\in H,$$
    which implies that
    \begin{align*} 
    y\left(\overline w\mathbf{1}-v_{t,s}\right)=0, ~\forall y\in B(H).
    \end{align*}
In particular,     $$v_{t,s}=\overline w\mathbf{1}.$$

    Note that for all $x\in \cC_{p,\infty}$,
    $$\norm{x}_{B(H)}\stackrel{\tiny \mbox{\cite[Lem. 3.2.3]{DPS}}}{=}\mu(0;x)\le \norm{x}_{\cC_{p,\infty}}.$$
    By a similar argument as used in the proof of \cite[p. 152-153]{BFJ}, we have that for any unit vectors $f,g,h$ in $H$,
    \begin{align} \label{algebra}
        \lim_{s\to t}\left|\langle V_sh,g\rangle_H\right|=\left|\langle V_th,g\rangle_H\right|, \qquad 
        \lim_{s\to t}\left|\langle U_sf,h\rangle_H\right|=\left|\langle U_tf,h\rangle_H\right|.
    \end{align}
    For each $t$, let
    $$\tilde{U}_t:=\{\lambda U_t: \lambda\in \mathbb{C},|\lambda|=1\}.$$
    By \eqref{w1} and the definitions of $u_{t,s},v_{t,s}$, we have that
    \begin{align} \label{ray repre}
    U_{t+s}=\overline{w}U_tU_s, ~V_{t+s}=wV_sV_t.
    \end{align}
    In particular, $\tilde{U}_s \tilde{U}_t :=  \{\lambda U_sU_t: \lambda\in \mathbb{C},|\lambda|=1\}  =  \{\lambda U_{s+t}: \lambda\in \mathbb{C},|\lambda|=1\}    =\tilde{U}_{t+s}$. 
    By \eqref{algebra} and \eqref{ray repre}, we have that $t\mapsto \tilde U_t$ is a continuous ray representation of $\left(\mathbb{R},+\right)$ (see the definition of ray representation in \cite[p. 6]{B54}).
    Since $\left(\mathbb{R},+\right)$ is a one-dimensional connected and simply connected group,
    it follows from 
    \cite[(6.2)]{B54} that
    $t\mapsto \tilde{U}_t$ satisfies the assumption of \cite[Theorem 3.2]{B54}.
    Applying \cite[Theorem 3.2]{B54} to $t\mapsto \tilde{U}_t$, we obtain that there exists a strongly continuous one-parameter unitary group $\{\hat{U}_t\}_{t\in \mathbb{R}}$ on $H$  and a unimodular complex number $\lambda_t$ (dependent on $t$) such that
    $$U_t=\lambda_t\hat{U}_t.$$
    Arguing similarly, we have that there exists a strongly continuous one-parameter unitary group $\{\hat V_t\}_{t\in \mathbb{R}}$ on $H$  and a unimodular complex number $\mu_t$ (dependent on $t$) such that
    $$V_t=\mu_t\hat V_t.$$
    Then,
    $$T_t(x)=\lambda_t\mu_t \hat U_t x \hat V_t, \quad  \forall t\in \mathbb{R},~  x\in \cC_{p,\infty}.$$
    Since $\{T_t\}_{t\in\mathbb{R}},\left\{\hat U_t\right\}_{t\in\mathbb{R}},\left\{\hat V_t\right\}_{t\in\mathbb{R}}$ are strongly continuous one-parameter groups, 
    it follows that $\{\lambda_t\mu_t \hat U_t\}_{t\in \mathbb{R}}$ is also a strongly continuous one-parameter unitary group on $H$.
    From now on, we identify $\lambda_t\mu_t \hat U_t$ and $\hat V_t$ with $U_t$ and $V_t$, respectively.
    By \cite[Theorem 12.2]{S18}, there exist two self-adjoint operators $A,B$ such that
    $U_t=e^{itA}$ and $V_t=e^{itB}$.
    The proof is complete.
\end{proof}

Some criteria for continuity of one-parameter group $\left\{T_t(\cdot):=e^{itA}\cdot e^{itB}\right\}_{t\in \mathbb{R}}$ will be discussed next.

\begin{theorem} \label{8.2} 
    Assume that $0<p<\infty$.
    Let $A,B$ be two self-adjoint operators on a complex infinite-dimensional separable Hilbert space   $H$. 
    For a one-parameter group $\left\{T_t(\cdot):=e^{itA} \cdot e^{itB}\right\}_{t\in \mathbb{R}}$ of isometries on $\mathcal{C}_{p,\infty}$, 
    the following statements are equivalent:
    \begin{enumerate} [label=(\roman*)]
        \item $\{T_t\}_{t\in \mathbb{R}}$ is strongly continuous.
        \item $\{T_t\}_{t\in \mathbb{R}}$ is continuous in the uniform operator topology of $B\left(\mathcal{C}_{p,\infty}\right)$.
        \item $A,B$ are bounded.
    \end{enumerate}
\end{theorem}
\begin{proof}
The implication (ii) $\Rightarrow$ (i) is trivial.

(iii) $\Rightarrow$ (ii).
Assume that $A,B$ are bounded. The uniform continuity 
  follows from 
the fact that  bounded operators $A,B$ generate uniformly continuous one-parameter group of isometries on $B(H)$ \cite{BFJ}.

(i) $\Rightarrow$ (iii). 
   
   {\bf Step $1$}.
    Assume by contradiction that  
    at least one of $A,B$ is unbounded. 
    It suffices to find $x=\sum_{n=0}^\infty c_n x_n\in \mathcal{C}_{p,\infty}$ such that $\{T_t x \}_{t\in\mathbb{R}}$ is not  continuous,  where $\{c_n\}_{n=0}^\infty$ is a sequence of scalars in $\mathbb{C}$ and $\{x_n\}_{n=0}^\infty\subset B(H)$ is a mutually orthogonal  sequence of rank-one operators.
    Without loss of generality, we may assume that $A$ is unbounded.
    We claim that there exists a mutually orthogonal sequence $ \{p_n\}_{n=0}^\infty$ of nonzero projections commuting with $B$ such that 
    \[0\le  b_n:=\norm{Bp_n}_{B(H)}<\infty, ~ \]
    Indeed, let $e^B$ denote the spectral measure of $B$.
    We first assume that $\sigma(B)$ contains infinitely many  points.
    Then, there exists a strictly increasing (or decreasing) sequence 
    $\left\{\mu_n\right\}_{n=0}^\infty\subset \sigma(B)$.
    Since 
    $$\sigma(B)\stackrel{\tiny\mbox{\cite[Thm. 1.5.13]{DPS}}}{=}\operatorname{supp}\left(e^B\right)\stackrel{\tiny\mbox{\cite[p. 14]{DPS}}}{=}\mathbb{R}\setminus \bigcup_{\substack{V\subset \mathbb{R} ~\text{open}\\e^B(V)=0}}V,$$
    it follows that
    $\mu_n\in \sigma(B)$ implies that 
    \begin{align} \label{>0}
    e^B\left(\mu_n-\varepsilon,\mu_n+\varepsilon\right)\ne0, ~\forall \varepsilon>0.
    \end{align}
    For any $n\ge1$, let
    $$d_n:=\frac 12\min\{\left|\mu_n-\mu_{n-1}\right|,\left|\mu_{n+1}-\mu_{n}\right|\}>0$$
    and for $n=0$, let
    $$d_0:=\frac  {|\mu_1-\mu_0|}2.$$
    For any $n\in\mathbb{N}$, 
    take $p_n:=e^B(\mu_n-d_n,\mu_n+d_n)\stackrel{\eqref{>0}}{\ne}0$.
    In this case, 
    $$0\le b_n:=\norm{Bp_n}_{B(H)}\le |\mu_n|+d_n<\infty.$$
    Now we assume that $\sigma(B)=\{\mu_1,\dots,\mu_N\}$ is a finite set.
    In this case, $B$ is bounded and
    there exists $1\le m\le N$  such that $e^{B}(\{\mu_m\})H$ is infinite-dimensional.
    Let $\{r_n\}_{n=0}^\infty\subset H$ be an orthonormal basis of $e^{B}(\{\mu_m\})H$ and 
    for $n\in \mathbb{N}$, $p_n$ denote  the projection onto the one-dimensional linear space spanned by $r_n$.
    Note that for all $n\in \mathbb{N}$, 
     \begin{align}\label{eq:pnB=Bpn}
        p_nB
    \stackrel{\tiny \mbox{\cite[p. 64]{DPS}}}{=}Bp_n
     \end{align}
     and $b_n:=\norm{Bp_n}_{B(H)}=\norm{\mu_mp_n}_{B(H)}=|\mu_m|<\infty$. 
    This proves the claim.

    {\bf Step $2$}. 
    Choose $g_n\in H$ with $\norm{g_n}_H=1$ such that
    \[g_n\otimes g_n\le p_n.\]
    Denote the $\sigma$-algebra of all Borel subsets of $\mathbb{R}$ by $\mathcal{B}(\mathbb{R})$. 
    Let $e^B$ denote the spectral measure of $B$.
    Define a  $\sigma$-additive probability measure $e^B_{g_n,g_n}:\mathcal{B}(\mathbb{R})\to \mathbb{C}$ by
    \begin{align} \label{gn,gn}
    e_{g_n,g_n}^B(\delta):=\left\langle e^B(\delta )g_n, g_n \right\rangle_H\stackrel{g_n\otimes g_n \le p_n}{=}\left\langle e^B(\delta )p_n g_n, g_n \right\rangle_H
    \end{align}
    for all $\delta\in \mathcal{B}(\mathbb{R})$ \cite[p. 9]{DPS}.
    Below, we will prove that $\operatorname{supp} e^B_{g_n,g_n}\subset [-b_n,b_n]$ for all $n\in \mathbb{N}$.
By \eqref{eq:pnB=Bpn}, we have $B|_{p_nH}$ is self-adjoint.
The uniqueness in the spectral theorem
implies that 
\[
e^{B|_{p_nH}}(\delta)=e^B(\delta)|_{p_nH}.
\]
Since $g_n\otimes g_n\le p_n$ and $\norm{B|_{p_nH}}_{B(p_nH)}\le b_n$,
it follows from \cite[Theorem 1.5.13]{DPS} that
    \begin{align} \label{supp B}
    \operatorname{supp} e^B_{g_n,g_n}
    \stackrel{\eqref{gn,gn}}{=}\operatorname{supp} e^{B|_{p_nH}}_{g_n,g_n}\subset\sigma\left(B|_{p_nH}\right)
    \subset [-b_n,b_n].
    \end{align}

    {\bf Step $3$}.
    Choose an increasing sequence $\{a_n\}_{n=0}^\infty$ such that $0< a_n\uparrow \infty$ as $n\to \infty$.
    For each $j\in \mathbb{N}$, let $I_j:=\left\{k\in\mathbb{N}:2^j-1\le k< 2^{j+1}-1\right\}$ and $N_j:=\#\left(I_j\right)=2^j$.
    Then, we have
    \begin{align} \label{le22}
    N_0+N_1+\dots+N_j\le 2N_j.
    \end{align}
    Note that $N_j\uparrow \infty$ as $j\to \infty$.
    Let $e^A$ denote the spectral measure of $A$.
Since $A$ is unbounded, it follows that there exists
a  mutually disjoint sequence 
$\{\delta_n\}_{n=0}^\infty$ of nonempty bounded closed intervals of $\mathbb{R}$ such that $e_n:=e^A(\delta_n)\ne 0$ and 
    for all $k\in I_j$,
    \begin{align*} 
    |\lambda_k|> b_k+a_j, ~\forall \lambda_k\in \delta_k.
    \end{align*}
    Without loss of generality, we may assume that $\delta_n\subset \left(0,\infty\right)$ for all $n$ and in this case, for all $k\in I_j$,
    \begin{align} \label{infimum1}
    \lambda_k> b_k+a_j, ~\forall \lambda_k\in \delta_k.
    \end{align}
    For all $j\in \mathbb{N}$, $k\in I_j$, $\lambda_k\in \delta_k$ and $\gamma_k\in \operatorname{supp}\left(e_{g_k,g_k}^B\right)\subset [-b_k,b_k]$,  we have that
    \begin{align} \label{>aj}
    |\lambda_k+\gamma_k|\ge |\lambda_k|-|\gamma_k|\stackrel{\eqref{infimum1}}{>}a_j.
    \end{align}
    For any $n$, choose $f_n\in H$ with $\norm{f_n}_H=1$ such that
    \begin{align}\label{fn en}
    f_n\otimes f_n\le e_n =e^A(\delta_n).
    \end{align}
    Define a $\sigma$-additive measure $e^A_{f_n,f_n}:\mathcal{B}(\mathbb{R})\to \mathbb{C}$ by
    \[e_{f_n,f_n}^A(\delta)=\left\langle e^A(\delta )f_n, f_n \right\rangle_H \]
    for all $\delta\in \mathcal{B}(\mathbb{R})$.
    Observe that 
    \begin{align} \label{supp A}
    \operatorname{supp} e^A_{f_n,f_n}\stackrel{\eqref{fn en}}{\subset} \delta_n.
    \end{align}
    Choose  $\{t_j\}$ such that $t_j\downarrow0$ as $j\to \infty$ and for all $j$,
    \[a_j\ge\frac2{t_j}.\]

    {\bf Step $4$}. We claim that for all $k\in I_j$,
    $$\norm{e^{itA}(f_k\otimes g_k)e^{itB}-f_k\otimes g_k}_{2}^2=\int_{\delta_k}\int_{\operatorname{supp}\left(e^B_{g_k,g_k}\right)}\left|e^{it(\lambda+\gamma)}-1\right|^2de^B_{g_k,g_k}(\gamma)de^A_{f_k,f_k}(\lambda).$$
    Indeed, by \cite[p. 317]{B08}, we have that for any $\xi_1,\xi_2,\eta_1,\eta_2\in H$,
    \begin{align} \label{2 H}
    \langle\xi_1\otimes \xi_2,\eta_1\otimes \eta_2\rangle_2=\langle\xi_1,\eta_1\rangle_H
    \overline{\langle\xi_2,\eta_2\rangle_H}.
    \end{align}
    Then, for all $k\in I_j$,
    \begin{eqnarray} \label{H 2}
        &\quad&\norm{e^{itA}(f_k\otimes g_k)e^{itB}-f_k\otimes g_k}_{2}^2\nonumber\\
        &=&\norm{\left(e^{itA}f_k\right)\otimes \left(e^{-itB} g_k\right)-f_k\otimes g_k}_{2}^2\nonumber\\
        &=&\Bigl\langle\left(e^{itA}f_k\right)\otimes \left(e^{-itB} g_k\right)-f_k\otimes g_k,\nonumber\\
        &&\qquad \left(e^{itA}f_k\right)\otimes \left(e^{-itB} g_k\right)-f_k\otimes g_k\Bigr\rangle_2\nonumber\\
        &=&\left\langle\left(e^{itA}f_k\right)\otimes \left(e^{-itB} g_k\right),\left(e^{itA}f_k\right)\otimes \left(e^{-itB} g_k\right)\right\rangle_2\nonumber\\
        &&+\left\langle f_k\otimes  g_k,f_k\otimes  g_k\right\rangle_2-2{\rm Re}\left\langle\left(e^{itA}f_k\right)\otimes \left(e^{-itB} g_k\right),f_k\otimes  g_k\right\rangle_2\nonumber\\
        &\stackrel{\eqref{2 H}}{=}& \langle e^{itA}f_k,e^{itA}f_k \rangle_H\langle e^{-itB}g_k,e^{-itB}g_k \rangle_H+\langle f_k,f_k\rangle_H\langle g_k,g_k\rangle_H\nonumber\\
        &&-2{\rm Re}\left(\langle e^{itA}f_k,f_k\rangle_H \overline{\langle e^{-itB}g_k,g_k\rangle_H }\right)\nonumber\\
        &\stackrel{\norm{f_k}_H=\norm{g_k}_H=1}{=}& 
        2-2{\rm Re}\left(\langle e^{itA}f_k,f_k\rangle_H 
        \overline{\langle e^{-itB}g_k,g_k\rangle_H }\right)\nonumber\\
        &=&2-2{\rm Re}\left(\langle e^{itA}f_k,f_k\rangle_H 
        \langle e^{itB}g_k,g_k\rangle_H \right).
    \end{eqnarray}
    Since $e^{it(\lambda+\gamma)}=\cos(t(\lambda+\gamma))+i\sin(t(\lambda+\gamma))$ and ${\rm Re}(e^{it(\lambda+\gamma)})=\cos(t(\lambda+\gamma))$,
    it follows that for all $k\in I_j$,
    \begin{align} \label{Re}
        &\quad {\rm Re}\left(\int_\mathbb{R}e^{it\lambda}de^A_{f_k,f_k}(\lambda)\int_\mathbb{R}e^{it\gamma}de^B_{g_k,g_k}(\gamma)\right)\nonumber\\
        &={\rm Re}\int_\mathbb{R}\int_\mathbb{R} e^{it(\lambda+\gamma)} de^B_{g_k,g_k}(\gamma)de^A_{f_k,f_k}(\lambda)\nonumber\\
        &={\rm Re}\int_\mathbb{R}\int_\mathbb{R} \cos(t(\lambda+\gamma))+i\sin(t(\lambda+\gamma)) de^B_{g_k,g_k}(\gamma)de^A_{f_k,f_k}(\lambda)\nonumber\\
        &=\int_\mathbb{R} \int_\mathbb{R}\cos(t(\lambda+\gamma))de^B_{g_k,g_k}(\gamma)de^A_{f_k,f_k}(\lambda)\nonumber\\
        &=\int_\mathbb{R}\int_\mathbb{R}{\rm Re}(e^{it(\lambda+\gamma)})de^B_{g_k,g_k}(\gamma)de^A_{f_k,f_k}(\lambda).
    \end{align}
    Observe that for all $k\in I_j$,
    \begin{eqnarray} \label{lambdagamma}
    &\quad&\norm{e^{itA}(f_k\otimes g_k)e^{itB}-f_k\otimes g_k}_{2}^2\nonumber\\
    &\stackrel{\eqref{H 2}}{=}&2-2{\rm Re}\left(\langle e^{itA}f_k,f_k \rangle_H \langle e^{itB}g_k,g_k \rangle_H\right)\nonumber\\
    &\stackrel{\tiny\mbox{\cite[Thm. 1.5.7(ii)]{DPS}}}{=}&2-2{\rm Re}\left(\int_\mathbb{R}e^{it\lambda}de^A_{f_k,f_k}(\lambda)\int_\mathbb{R}e^{it\gamma}de^B_{g_k,g_k}(\gamma)\right)\nonumber\\
    &\stackrel{\eqref{Re}}{=}&\int_{\mathbb{R}}\int_{\mathbb{R}}2de^B_{g_k,g_k}(\gamma)de^A_{f_k,f_k}(\lambda)\nonumber\\
    &&-\int_{\mathbb{R}}\int_{\mathbb{R}}2{\rm Re} \left(e^{it(\lambda+\gamma)}\right)de^B_{g_k,g_k}(\gamma)de^A_{f_k,f_k}(\lambda)\nonumber\\
    &=&\int_{\mathbb{R}}\int_{\mathbb{R}}2-2{\rm Re} \left(e^{it(\lambda+\gamma)}\right)de^B_{g_k,g_k}(\gamma)de^A_{f_k,f_k}(\lambda)\nonumber\\
    &=&\int_{\mathbb{R}}\int_{\mathbb{R}} \left(e^{it(\lambda+\gamma)}-1\right)\left(e^{-it(\lambda+\gamma)}-1\right)de^B_{g_k,g_k}(\gamma)de^A_{f_k,f_k}(\lambda)\nonumber\\
    &=&\int_{\mathbb{R}}\int_{\mathbb{R}} \left|e^{it(\lambda+\gamma)}-1\right|^2de^B_{g_k,g_k}(\gamma)de^A_{f_k,f_k}(\lambda)\nonumber\\
    &\stackrel{\substack{\eqref{supp A}\\\eqref{supp B}}}{=}&\int_{\delta_k}\int_{\operatorname{supp}\left(e^B_{g_k,g_k}\right)}\left|e^{it(\lambda+\gamma)}-1\right|^2de^B_{g_k,g_k}(\gamma)de^A_{f_k,f_k}(\lambda).
    \end{eqnarray}
    This proves the claim.


{\bf Step $5$}. We will construct a decreasing sequence $\left\{t_j'\right\}_{j=0}^\infty$ tending to $0$ and for all $j\in\mathbb{N}$, a subset $E_j\subset I_j$ with $\#(E_j)\ge \frac{N_j}{5}$ consisting of $k\in I_j$ satisfying 
$$\norm{e^{it_j'A}(f_k\otimes g_k)e^{it_j'B}-f_k\otimes g_k}_2\ge\frac12.$$
    Since $t_j|\lambda_k+\gamma_k|\stackrel{\eqref{>aj}}{>}t_ja_j\ge2$ for all $k\in I_j$, $\lambda_k\in \delta_k$ and $\gamma_k\in \operatorname{supp}\left(e^B_{g_k,g_k}\right)$,
    it follows that 
    \begin{align} \label{ge11}
\frac1{t_j}\int_0^{t_j}\left|e^{it(\lambda_k+\gamma_k)}-1\right|^2dt
    &=\frac1{t_j}\int_0^{t_j}\left|\cos \left(t(\lambda_k+\gamma_k)\right)+i\sin \left(t(\lambda_k+\gamma_k)\right)-1\right|^2dt\nonumber\\
    &=\frac1{t_j}\int_0^{t_j}\left(\cos \left(t(\lambda_k+\gamma_k)\right)+i\sin \left(t(\lambda_k+\gamma_k)\right)-1\right)\nonumber\\
    &\quad\left(\cos \left(t(\lambda_k+\gamma_k)\right)-i\sin \left(t(\lambda_k+\gamma_k)\right)-1\right)dt\nonumber\\
    &=\frac1{t_j}\int_0^{t_j}2-2\cos\left(t(\lambda_k+\gamma_k)\right)dt\nonumber\\
    &=2-\frac{2\sin\left((\lambda_k+\gamma_k)t_j\right)}{\left(\lambda_k+\gamma_k\right)t_j}\nonumber\\
    &\ge1.
    \end{align}
    For any $k\in I_j$, we have that
    \begin{eqnarray} \label{8.2 ge1}
        & &\frac{1}{t_j}\int_0^{t_j}\norm{e^{itA}\left(f_k\otimes g_k\right)e^{itB}-f_k\otimes g_k}_2^2dt\nonumber\\
        &\stackrel{\eqref{lambdagamma}}{=}&\frac{1}{t_j}\int_0^{t_j}\int_{\delta_k}\int_{\operatorname{supp}\left(e^B_{g_k,g_k}\right)}\left|e^{it(\lambda+\gamma)}-1\right|^2de^B_{g_k,g_k}(\gamma)de^A_{f_k,f_k}(\lambda)dt\nonumber\\
        &\stackrel{\tiny\mbox{\cite[Thm. 2.37]{F99}}}{=}&\frac{1}{t_j}\int_{\delta_k}\int_{\operatorname{supp}\left(e^B_{g_k,g_k}\right)}\int_0^{t_j}\left|e^{it(\lambda+\gamma)}-1\right|^2dtde^B_{g_k,g_k}(\gamma)de^A_{f_k,f_k}(\lambda)\nonumber\\
        &\stackrel{\eqref{ge11}}{\ge} &\int_{\delta_k}\int_{\operatorname{supp}\left(e^B_{g_k,g_k}\right)}1de^B_{g_k,g_k}(\gamma)de^A_{f_k,f_k}(\lambda)=1.
    \end{eqnarray}
    Observe that for each $j$ and $k\in I_j$, 
    \begin{align*}
    &\quad \frac1{t_j}\int_0^{t_j}\left(\frac1{N_j}\sum_{k\in I_j}\norm{e^{itA}(f_k\otimes g_k)e^{itB}-f_k\otimes g_k}_2^2\right)dt\\
    &=\frac1{N_j}\sum_{k\in I_j}\frac1{t_j}\int_0^{t_j}\norm{e^{itA}(f_k\otimes g_k)e^{itB}-f_k\otimes g_k}_2^2dt\stackrel{\eqref{8.2 ge1}}{\ge}1.
    \end{align*}
    Then, there exists $t_j'\in\left(0,t_j\right)$ such that
    \begin{align} \label{ge111}
    \frac1{N_j}\sum_{k\in I_j}\norm{e^{it_j'A}(f_k\otimes g_k)e^{it_j'B}-f_k\otimes g_k}_2^2\ge1.
    \end{align}
    For any $j\in \mathbb{N}$,  let \[E_j:=\left\{k\in I_j:\norm{e^{it_j'A}(f_k\otimes g_k)e^{it_j'B}-f_k\otimes g_k}_2\ge\frac12\right\}.\]
    By $\norm{e^{it_j'A}(f_k\otimes g_k)e^{it_j'B}-f_k\otimes g_k}_2\le2$, we have that 
    \[\frac1{N_j}\sum_{k\in I_j}\norm{e^{it_j'A}(f_k\otimes g_k)e^{it_j'B}-f_k\otimes g_k}_2^2\le \frac1 {N_j}\left(4\#\left(E_j\right)+\frac14\left(N_j-\#\left(E_j\right)\right)\right).\]
    This together with \eqref{ge111} yields that
    \begin{align} \label{com33}
    \#\left(E_j\right)\ge\frac{N_j}5.
    \end{align}

    {\bf Step $6$}.
    Let $x=\sum_{n=0}^\infty N_{\{j:n\in I_j\}}^{-1/p} f_n\otimes g_n$.
    Below, we will prove that $x\in \cC_{p,\infty}$ and $t\mapsto e^{itA}xe^{itB}$ is not continuous with respect to $\norm{\cdot}_{\cC_{p,\infty}}$.
    Given any $n\ge0$, we obtain that $n\in I_{j}$ for some $j$ and so,
    \begin{align}\label{ineq11}
    n+1\le N_0+N_1+\dots+N_{j}\stackrel{\eqref{le22}}{\le}2N_{j}.
    \end{align}
    Then, for this $n$, we have  $\mu\left(n;x\right) = N^{-1/p}_{\{i: n\in I_{i}\}} =N^{-1/p}_{j}$, and therefore, 
    \begin{align*}
    \mu\left(n;x\right)\left(n+1\right)^{1/p}=\left(\frac {n+1} {N_{j}}\right)^{1/p}\stackrel{\eqref{ineq11}}{\le}2^{1/p}.
    \end{align*}
    In particular, $x\in \cC_{p,\infty }$.
    Since $e^{it_j'A}(f_n\otimes g_n)e^{it_j'B}-f_n\otimes g_n$ is at most rank $2$ ($f_n,g_n$ are unit vectors in $H$ such that $f_n\otimes f_n\le e_n,g_n\otimes g_n\le p_n$),
    it follows from the definition of $E_j$ that for all $k\in E_j$,
    \[\frac12\le\norm{e^{it_j'A}(f_k\otimes g_k)e^{it_j'B}-f_k\otimes g_k}_2\le\sqrt{2}\norm{e^{it_j'A}(f_k\otimes g_k)e^{it_j'B}-f_k\otimes g_k}_{B(H)},\]
    that is,
    \[\norm{e^{it_j'A}(f_k\otimes g_k)e^{it_j'B}-f_k\otimes g_k}_{B(H)}\ge\frac{\sqrt{2}}4.\]
    This shows that for all $k\in E_j\subset I_j$,
    \begin{align} \label{21}
    \mu\left(0;e^{it_j'A}(f_k\otimes g_k)e^{it_j'B}-f_k\otimes g_k\right)\ge \frac{\sqrt{2}}{4}.
    \end{align}
    Since $\{e_n\}_{n=0}^\infty$ and $\{p_n\}_{n=0}^\infty$ 
    commute with $A$ and $B$, respectively, it follows 
    that for any $n\in \mathbb{N}$,
    \begin{eqnarray*} 
    &\quad & e_n\left(e^{it_j'A}\left(f_n\otimes g_n\right)e^{it_j'B}-f_n\otimes g_n\right)p_n\nonumber\\
    &=&e^{it_j'A}e_n\left(f_n\otimes g_n\right)p_ne^{it_j'B}-e_n\left(f_n\otimes g_n\right)p_n\nonumber\\
    &\stackrel{\substack{f_n\otimes f_n\le e_n\\g_n\otimes g_n\le p_n}}{=}&e^{it_j'A}\left(f_n\otimes g_n\right)e^{it_j'B}-f_n\otimes g_n,
    \end{eqnarray*}
    which together with the fact that $\{e_n\}_{n=0}^\infty$ and $\{p_n\}_{n=0}^\infty$ are two pairwise orthogonal sequences, implies that  
    $$N^{-1/p}_{\{j:n\in I_j\}}\left(e^{it_j'A}\left(f_n\otimes g_n\right)e^{it_j'B}-f_n\otimes g_n\right), ~0\le n <\infty,$$
    are pairwise orthogonal from the left and the right. 
    On the other hand, we have 
      $$e^{it_j'A}xe^{it_j'B}-x =\sum_{n=0}^\infty N^{-1/p}_{\{j:n\in I_j\}}\left(e^{it_j'A}\left(f_n\otimes g_n\right)e^{it_j'B}-f_n\otimes g_n\right).$$
Hence, $\mu\left(e^{it_j'A}xe^{it_j'B}-x\right) \stackrel{\substack{\eqref{com33}\\ \eqref{21}}}{\ge} \frac{\sqrt{2}}{4} (N_j)^{-1/p} \chi_{_{\left(0,  \frac{N_j}{5} \right)}} $
      for every $j$. 
    This   yields that for every $j$, we have 
\[\norm{e^{it_j'A}xe^{it_j'B}-x}_{\mathcal{C}_{p,\infty}}\ge 
\left(\frac{N_j}{5}\right)^{1/p} \left( 
\frac{\sqrt{2}}{4} (N_j)^{-1/p}\right) =\frac{\sqrt{2}}{4}\frac{1}{5^{1/p}}.\]
    Since $t_j'\in\left(0,t_j\right)$ and $t_j\to0$ as $j\to \infty$, it follows that
    $\{T_t\}_{t\in \mathbb{R}}$ is not a strongly continuous one-parameter group of isometries on $\mathcal{C}_{p,\infty}$, a contradiction.
    Therefore, $A,B\in B(H)$.
\end{proof}

\begin{remark}
    A combination of Proposition \ref{8.1} and Theorem \ref{8.2} yields that $\{T_t\}_{t\in \mathbb{R}}$ is a strongly continuous one-parameter group of isometries of $\cC_{p,\infty}, ~0<p<\infty$ if and only if there exist two self-adjoint operators $A,B\in B(H)$ such that
    $$T_t(x)=e^{itA}xe^{itB}, ~\forall t\in \mathbb{R}, ~x\in \cC_{p,\infty}.$$
\end{remark}

The following result provides a necessary and sufficient condition for a  one-parameter group of isometries on $\ell_{p,\infty}$ to be strongly continuous.

\begin{theorem}
    Assume that $0<p<\infty$.
    Then, $\{T_t\}_{t\in \mathbb{R}}$ is a strongly continuous one-parameter group of isometries of $\ell_{p,\infty}$ if and only if there exists a real sequence $A=\left(a_n\right)_{n=0}^\infty\in \ell_\infty$ such that
    $$T_t(x)(n)=e^{ita_n}x(n), ~\forall t\in \mathbb{R},~n\in \mathbb{N},~x\in \ell_{p,\infty}.$$
\end{theorem}
\begin{proof}
    The implication 
    $\Leftarrow$  is trivial. 

    $\left(\Rightarrow\right)$.
    For any $t\in \mathbb{R}$, by Theorem \ref{111}, there exists a permutation $\sigma^t$ of the set $\{0,1,2,\dots\}$ and a unimodular sequence $\{u^t_n\}_{n\ge0}$ such that
    $$T_t(x)(n)=u^t_nx(\sigma^t(n)), ~\forall n\in \mathbb{N}, ~x\in \ell_{p,\infty}.$$

    Below, we will prove that
    $$\sigma^t(n)=n, ~\forall t\in \mathbb{R},~n\in \mathbb{N}.$$
    Let $\{e_n\}_{n=0}^\infty$ be the standard basis of $c_0$.
    Fix an integer $n\in \mathbb{N}$.
    Since $\{T_t\}_{t\in\mathbb{R}}$ is a strongly continuous one-parameter group of isometries on $\ell_{p,\infty}$, it follows that
    \begin{align*} 
    \norm{T_t\left(e_n\right)-e_n}_{\ell_{p,\infty}}\to0 ~\text{as}~t\to 0.
    \end{align*}
    This implies that  there exists $\delta_n>0$ such that
    $$\norm{T_t\left(e_n\right)-e_n}_{\ell_{p,\infty}}<1,~\forall |t|<\delta_n,$$
    which implies that
    \begin{align} \label{delta n}
    \sigma^t(n)=n,~\forall |t|<\delta_n.
    \end{align}
    For any $n\in \mathbb{N}$, let $G_n:=\{t\in \mathbb{R}:\sigma^t(n)=n\}$.
    By \eqref{delta n}, we have that
    $\left(-\delta_n,\delta_n\right)\subset G_n$.
    We claim that $\left(G_n,+\right)$ is a group. Indeed, it suffices to prove that $t,s\in G_n$ implies $t+s\in G_n$. 
    Assume that $t,s\in G_n$.
    Since $\{T_t\}_{t\in\mathbb{R}}$ is a one-parameter group of isometries on $\ell_{p,\infty}$, it follows that $T_{t+s}=T_tT_s$ and so, 
    \begin{align} \label{one-parameter}
    T_{t+s}e_n=T_t(T_se_n).
    \end{align}
    Observe that for some $n_1\in \mathbb{N}$,
    $$T_t(T_se_n)=T_t(u^se_n)=u^tu^se_n, ~\sigma^{t+s}(n_1)=n,~T_{t+s}(e_n)=u^{t+s}e_{n_1}.$$ 
    This together with \eqref{one-parameter} yields that $u^tu^se_n=u^{t+s}e_{n_1} $, and therefore, 
    \begin{align} \label{u t+s}
    \sigma^{t+s}(n)=n, ~u^{t+s}(n)=u^t(n)u^s(n). 
    \end{align}
    So, $t+s\in G_n$.
    For any $t\in \mathbb{R}$, we have that $|\frac t m|<\delta_n$ for some integer $m$.
    Then, we obtain that $\frac t m\in G_n$ and so, $t\in G_n$.
    Consequently, $\sigma^t(n)=n$ for all $t\in \mathbb{R}$.
    Since $n\in \mathbb{N}$ is arbitrary, it follows that
    $$\sigma^t(n)=n, ~\forall t\in \mathbb{R},~n\in \mathbb{N}.$$
    Therefore,
    $$T_t(x)(n)=u^t_nx(n), ~\forall n\in \mathbb{N}, ~t\in \mathbb{R}, ~x\in \ell_{p,\infty}.$$

    Since $\{T_t\}_{t\in\mathbb{R}}$ is strongly continuous, it follows that
    \begin{align} \label{ut strongly}
    \norm{T_t\left(e_n\right)-e_n}_{\ell_{p,\infty}}=|u^t_n-1|\to0 ~\text{as}~t\to 0 .
    \end{align}
    Letting $t=0$, we have that $u^0=\mathbf{1}$. 
    This together with \eqref{ut strongly} and \eqref{u t+s} shows that
    $\{u^t\}_{t\in \mathbb{R}}$ is a strongly continuous one-parameter unitary group.
    By \cite[Theorem 12.2]{S18}, there exists a real sequence $A=\left(a_n\right)_{n=0}^\infty$ such that
    $$u^t=e^{itA}, ~\forall t\in \mathbb{R}.$$
    Hence,
    $$T_t(x)(n)=e^{ita_n}x(n),~\forall n\in\mathbb{N}, ~t\in\mathbb{R}, ~x\in \ell_{p,\infty}.$$
    By a similar argument  as used in the proof of Theorem \ref{8.2}, we have that $A\in \ell_\infty$.
    This completes the proof.
\end{proof}

\appendix

\chapter{Historical comments concerning  extreme points and isometries on  Banach  spaces}
The study of   extreme points of the unit ball of a Banach space has been widely carried out. 
In particular, Kadison \cite{K51} showed that the set of all extreme points of the unit ball of a $C^*$-algebra $\mathcal{A}$ with the unit element $\mathbf{1}$ is exactly the set of all partial isometries $u$ such that $\left(\mathbf{1}-l\left(u\right)\right)\mathcal{A}\left(\mathbf{1}-r\left(u\right)\right)=\{0\}$. 
The formulation of extreme points of the unit ball of a Lorentz function space $L_{p,1}(0,1),~1<p<\infty$ (and, more generally,   $L_{w,1}\left(0,\infty\right)$)
is due to  
Semenov (see, e.g., \cite{Semenov_thesis} and \cite[Lemma II.5.2]{KPS}; see also \cite{CT,CDT,Bikt}).
A complete description of extreme points of the unit ball of a noncommutative Lorentz space $\Lambda_\Psi$ over  a finite von Neumann algebra, was given in  \cite{Sukochev_thesis,CMS}. 
Recently,   de Jager and  Conradie \cite{JC22}
characterized extreme points of the unit balls of $\left(L_1+L_\infty\right)\left(\mathcal{M},\tau\right)$, $\left(L_1\cap L_\infty\right)\left(\mathcal{M},\tau\right)$, and 
noncommutative Lorentz space $L_{\omega,1}\left(\mathcal{M},\tau\right)$ affiliated
with a semifinite von Neumann algebra $\mathcal{M}$.
We refer to \cite{CK,CK2} for more information concerning extreme points of the unit ball of a symmetric space.


Note that   Kami\'{n}ska and   Parrish \cite{KP} (see also \cite{Bikt84}) showed that $x$ is an extreme point of the unit ball of a  Marcinkiewicz function space $M_w$ equipped with a fully symmetric norm \cite{Har} given by $\norm{x}_w:=\sup\limits_{t>0}\frac{\int_0^t \mu\left(s;x\right)ds}{\int_0^t w\left(s\right)ds}$ if and only if $\mu\left(t;x\right)=w\left(t\right)$, when $w$ is a strictly decreasing weight function satisfying $\lim_{t\to0_+}w(t)=\infty$, $\lim_{t\to\infty}w(t)=0$, $\int_0^t w(s)ds<\infty$ for all $t>0$ and $\int_0^\infty w(t)dt=\infty$.
We also note that the unit ball of  $M_w$ is the orbit of $w$ \cite[p. 219]{CK2}.

Let the positive orbit $\Omega_+\left(x\right)$ of $0\le x\in \left(L_1+L_\infty\right)\left(\mathcal{M},\tau\right)$ be defined by 
\[\Omega_+\left(x\right)=\{0\le y\in \left(L_1+L_\infty\right)\left(\mathcal{M},\tau\right):y\prec\prec x\}.\]
It was shown in \cite{SZ} that for  $0\le x\in \left(L_1+L_\infty\right)\left(0,\infty\right)$,  if $\mu\left(\infty;x\right):=\lim_{t\to \infty}\mu\left(t;x\right)=0$, then 
\[\operatorname{extr}\left(\Omega_+\left(x\right)\right)=\left\{0\le y \in \left(L_1+L_\infty\right)\left(0,\infty\right):\mu\left(y\right)=\mu\left(x\right)\chi_{_{[0,\beta]}},~0\le \beta \le \infty\right\}\]
and, if $\mu\left(\infty;x\right)>0$, then 
\begin{align*}
&\operatorname{extr}\left(\Omega_+\left(x\right)\right)=\bigl\{0\le y \in \left(L_1+L_\infty\right)\left(0,\infty\right):\mu\left(y\right)=\mu\left(x\right)\chi_{_{[0,\beta]}}  \\
&\qquad ~\text{and}~y\chi_{_{\left(y<\mu(\infty;y)\right)}}=0,~ 0\le \beta \le \infty\bigr\}.
\end{align*}


Employing extreme points of the unit ball of a Banach space to characterize surjective isometries can be traced back to the seminal work of Kadison \cite[Theorem 7]{K51}, who showed that a surjective isometry from one unital $C^*$-algebra onto another is a Jordan $*$-isomorphism followed by left multiplication by a unitary operator.
Chilin, Medzhitov and Sukochev \cite{CMS} provided the description of a surjective isometry on Lorentz spaces with strictly monotone norm via extreme points when von Neumann algebra is equipped with a faithful normal finite trace.
Subsequently,  de Jager and   Conradie \cite{JC22} used the characterization of extreme points of the unit ball of $X$ to derive the general form of surjective isometry $T$ on $X$, where $X=L_{w,1}\left(\mathcal{M},\tau\right),~ L_1\cap L_\infty\left(\mathcal{M},\tau\right), ~\left(L_1+L_\infty\right)\left(\mathcal{M},\tau\right)$ (when $X=L_1\cap L_\infty\left(\mathcal{M},\tau\right)$, $T$ is required to be positive) and $w$ is a strictly decreasing weight function.
For characterizations of isometries by other approaches (e.g., Clarkson inequalities, hermitian operators etc.), we refer to \cite[Theorem 2]{Y81}, \cite{HS} and references therein.

\chapter{Some remarks concerning   the orthogonality of operators in noncommutative weak \texorpdfstring{$L_p$}{Lp}-spaces}
\chaptermark{Orthogonality in noncommutative weak $L_p$-spaces}
Recall that for any positive operators $x,y$, we have 
$\mu(x\oplus y) \prec \prec \mu (x+y)$,
 see, e.g.,  \cite[Lemma 2.3]{CDS97} or \cite[Lemma 3.3.7]{LSZ}. 
The main purpose of this appendix is to show that the above submajorization is strict for some point whenever $xy\ne 0$.

Before proceeding, some definitions are needed. 
A continuous real function $f$ on a real interval $I$ is called {\itshape operator convex} if for all $\lambda\in [0,1]$ and self-adjoint operators $x,y$ on an infinite-dimensional Hilbert space $H$ with spectra  included in $I$, 
\[f\left(\lambda x+\left(1-\lambda\right)y\right)\le \lambda f(x)+\left(1-\lambda\right)f(y)\]
(see \cite[Definition 1.10]{HPP}).

\begin{lemma} \label{convex}
 Let $\mathcal{M}$ be  a semifinite von Neumann algebra equipped with a semifinite faithful normal trace $\tau$.
    Let $f(t)$ be a non-negative continuous convex function on $[0,\infty)$ with $f(0)=0$.
    Assume that $0\le x,y\in S\left(\mathcal {M},\tau\right)$ is such that $f(x),f(y),f\left(x+y\right)\in L_1\left(\mathcal{M},\tau\right)$.
    If there exist a non-affine operator convex function $g(t)$ on $[0,\infty)$ with $g(0)=0$ and a convex strictly increasing function $\phi(t)$ (on the range of $g(t)$) with $\phi(0)=0$ such that $f=\phi \circ g$,
    then 
    \begin{align} \label{equality holds1}
        \tau\left(f\left(x+y\right)\right)\ge\tau\left(f(x)\right)+\tau\left(f(y)\right)
    \end{align}
    and equality holds in \eqref{equality holds1} if and only if $xy=0$. 
\end{lemma}
\begin{proof}
    By \cite[Theorem 5.3]{DS09}, we obtain that
\[\tau\left(f\left(x+y\right)\right)\ge\tau\left(f\left(x\right)\right)+\tau\left(f\left(y\right)\right).\]
The equivalent conditions in the above lemma follow immediately from \cite[Remark 11]{HK}.
\end{proof}


\begin{lemma} \label{concave}
Let $\mathcal{M}$ be  a semifinite von Neumann algebra equipped with a semifinite faithful normal trace $\tau$.
    Let $f(t)$ be a non-negative continuous concave function on $[0,\infty)$ with $f(0)=0$.
    Assume that $0\le x,y\in S\left(\mathcal {M},\tau\right)$ is such that $f(x),f(y),f\left(x+y\right)\in L_1\left(\mathcal{M},\tau\right)$.
    If there exist a non-affine operator convex function $g(t)$ on $[0,\infty)$ with $g(0)=0$ and a convex strictly increasing function $\phi(t)$ (on the range of $g(t)$) with $\phi(0)=0$ such that $-f=\phi \circ g$,
    then 
    \begin{align} \label{equality holds}
\tau\left(f\left(x+y\right)\right)\le\tau\left(f(x)\right)+\tau\left(f(y)\right)
    \end{align}
    and equality holds in \eqref{equality holds} if and only if $xy=0$.
\end{lemma}
\begin{proof}
    By \cite[Corollary 5.4]{DS09}, we obtain that
\[\tau\left(f\left(x+y\right)\right)\le\tau\left(f\left(x\right)\right)+\tau\left(f\left(y\right)\right).\]
An application of \cite[Remark 11]{HK} yields the remaining statements.
\end{proof}

Let $\mathcal{M}$ be a semifinite von Neuman algebra equipped with a semifinite faithful normal trace $\tau$.
Recall that $L_{\log}\left(\mathcal{M},\tau\right):=\{x\in S\left(\mathcal{M},\tau\right):\int_0^\infty \log\left(1+\mu\left(t;x\right)\right)dt<\infty\}$.
As is well known, $L_1\left(\mathcal{M},\tau\right)\subset L_{\log}\left(\mathcal{M},\tau\right)$ \cite[Proposition 2.8]{DSZ16}.
On the other hand, when $0\le x,y \in L_p\left(\mathcal{M},\tau\right), ~0<p<\infty,~p\ne1 $, the 
 following result is an immediate consequence of  \cite[Proposition A.2.]{RX}.  

\begin{prop}\label{positive disjoint}
Let $\mathcal{M}$ be  a semifinite von Neumann algebra equipped with a semifinite faithful normal trace $\tau$.
Let $0\le x,y\in L_{\log}\left(\mathcal M,\tau\right)$. 
Then, $\mu\left(x+y\right)=\mu\left(x\oplus y\right)$ if and only if $xy=0$.

\end{prop}
\begin{proof}
It suffices to prove the necessity.
Assume that $\mu\left(x+y\right)=\mu\left(x\oplus y\right)$. 
By \cite[Proposition 3.2.8]{DPS}, we obtain that
\[
\begin{aligned}
\tau\left(\log\left(\mathbf{1}+x+y\right)\right)
&=\int_0^\infty \log\left(1+\mu\left(t;x+y\right)\right)\,dt  \\
&=\int_0^\infty \log\left(1+\mu\left(t;x\oplus y\right)\right)\,dt  \\
&=\tau\left(\log\left(\mathbf{1}+x\right)\right)+\tau\left(\log\left(\mathbf{1}+y\right)\right).
\end{aligned}
\]
Note that $t\mapsto \log\left(1+t\right),t\ge0$ is concave and $t\mapsto -\log\left(1+t\right),t\ge0$ is a non-affine operator convex function \cite[Proposition 2.2.27]{DPS}.
Applying Lemma \ref{concave} to $\phi(t)=t,~g(t)=-\log\left(1+t\right), ~t\ge0$, we obtain that $xy=0$.
The proof is complete.
\end{proof}

The following fact is well-known to experts. We include a proof for completeness.  
\begin{lemma} \label{lemma 1}
Let $\mathcal{M}$ be  a semifinite von Neumann algebra equipped with a semifinite faithful normal trace $\tau$.
    Let $0\le z\in \left(L_1+L_\infty\right)\left(\mathcal{M},\tau\right)\cap S_0\left(\mathcal{M},\tau\right)$.
    Then, there exists a continuous strictly increasing convex function $f:[0,\infty)\to [0,\infty)$ such that
    \[f(0)=0,~ 0\le f(t)\le t ~ \text{and}~ \tau\left(f(z)\right)<\infty.\]
\end{lemma}
\begin{proof}
    Let
    \[E_n:=\{t\ge0:2^{-\left(n+1\right)}<\mu\left(t;z\right)\le 2^{-n}\}\]
    for all $n\ge 0$.
    Since $z\in S_0\left(\mathcal{M},\tau\right)$, we have that $m\left(E_n\right)<\infty$.
    Choose a sequence $\{\theta_n\}_{n=0}^\infty\subset \left(0,1\right)$ such that $\theta_n\downarrow0$ strictly as $n\to \infty$ and for all $n\in \mathbb{N}$, 
    \begin{align} \label{le1}
    ~\theta_nm\left(E_n\right)\le 1.
    \end{align}
    Let a continuous strictly increasing function $g:[0,\infty)\to [0,1)$ be defined by
    \[
    g(t)=
    \begin{cases}
    0, & t=0. \\
    \frac{\theta_{n+1}-\theta_n}{2^{-\left(n+1\right)}-2^{-n}}\left(t-2^{-n}\right)+\theta_n, & t\in [2^{-\left(n+1\right)},2^{-n}],~\forall n\ge0.\\
    \theta_0+\left(1-\theta_0\right)\frac{t-1}{t}, & t\ge 1.
    \end{cases}
    \]
    Define 
    \[f(t)=\int_0^t g(s)ds \]
    for $t\ge0$.
    Note that $f$ is continuous, strictly increasing and convex satisfying $f(0)=0$ and $0\le f(t)\le t$.
    It suffices to prove that $\tau\left(f(z)\right)<\infty$.
    By \cite[Proposition 3.2.8]{DPS}, we obtain that
    \[\mu\left(t; f(z)\right)=f\left(\mu\left(t;z\right)\right), ~\forall t\ge0.\]
    Observe that 
    \[\{t\ge0:0<\mu\left(t;z\right)\le1\}=\cup_{n=0}^\infty E_n\]
    and 
    \[f(t)=\int_0^tg(s)ds\le tg(t), ~\forall t\ge0.\]
    This implies that for any $t\in E_n$, we have that
    \[f\left(\mu\left(t;z\right)\right)\le \mu\left(t;z\right)g\left(\mu\left(t;z\right)\right)\le 2^{-n}g\left(2^{-n}\right)=2^{-n}\theta_n\]
    and so,
    \begin{align} \label{En}
        \int_{E_n}f\left(\mu\left(t;z\right)\right)dt \le 2^{-n}\theta_n m\left(E_n\right)\stackrel{\eqref{le1}}{\le} 2^{-n}, ~\forall n\in \mathbb{N}.
    \end{align}
    Since $z\in \left(L_1+L_\infty\right)\left(\mathcal{M},\tau\right)\cap S_0\left(\mathcal{M},\tau\right)$, it follows that
    \begin{align} \label{finite}
        \int_{\{t\ge0:\mu\left(t;z\right)>1\}} f\left(\mu\left(s;z\right)\right)ds\le \int_0^{m\{t\ge0:\mu\left(t;z\right)>1\}} f\left(\mu\left(s;z\right)\right)ds<\infty.
    \end{align}
    A combination of \eqref{En} and \eqref{finite} yields that
    \begin{align*}
        \int_0^\infty \mu\left(t;f(z)\right)dt&=\int_0^\infty f\left(\mu\left(t;z\right)\right)dt\\
        &=\sum_{n=0}^\infty\int_{E_n}f\left(\mu\left(t;z\right)\right)dt+\int_{\{t\ge0:\mu\left(t;z\right)>1\}} f\left(\mu\left(s;z\right)\right)ds<\infty.
    \end{align*}
    The proof is complete.
\end{proof}

\begin{lemma} \label{lemma 2}
Let $\mathcal{M}$ be  a semifinite von Neumann algebra equipped with a semifinite faithful normal trace $\tau$.
    Let $0\le z\in \left(L_1+L_\infty\right)\left(\mathcal{M},\tau\right)\cap S_0\left(\mathcal{M},\tau\right)$.
    Then, there exists a continuous increasing convex function $f:[0,\infty)\to [0,\infty)$ such that
    \[f(0)=0,~ 0\le f(t)\le t,~ \tau\left(f(z)\right)<\infty, ~\text{and}~ f=\phi \circ g ,\]
    where $g$ is a non-affine operator convex function on $[0,\infty)$ with $g(0)=0$ and $\phi$ is a continuous, convex and strictly increasing function on $[0,\infty)$ with $\phi(0)=0$.
\end{lemma}
\begin{proof}
    Let \[g(t):=\frac{t^2}{1+t}=t-1+\frac{1}{t+1},~t\ge0.\]
    Note that $g(0)=0,~0\le g(t)\le t$ and $g$ is not-affine, continuous, increasing and operator convex on $[0,\infty)$.
    This implies that $0\le g(z)\le z$ and so, $0\le g(z)\in \left(L_1+L_\infty\right)\left(\mathcal{M},\tau\right)\cap S_0\left(\mathcal{M},\tau\right)$.
    Applying Lemma \ref{lemma 1} to $g(z)$, there exists a continuous strictly  increasing convex function $\phi:[0,\infty)\to [0,\infty)$ such that
    \[\phi(0)=0,~ 0\le \phi(t)\le t ~ \text{and}~ \tau\left(\phi\left(g\left(z\right)\right)\right)<\infty.\]
    Let $f=\phi\circ g$, that is, $f\left(t\right)=\phi\left(g\left(t\right)\right),~t\ge0$.
    Then, 
    \[f\left(0\right)=\phi\left(g\left(0\right)\right)=0,~0\le f\left(t\right)\le g\left(t\right)\le t, ~\tau\left(f\left(z\right)\right)=\tau\left(\phi\left(g\left(z\right)\right)\right)<\infty.\]
    Note  that $f$ is continuous, increasing and convex.
    The proof is complete.
\end{proof}

The following theorem is the first main result of this chapter. 
\begin{theorem} \label{L1+M}
   Let $\mathcal{M}$ be  a semifinite von Neumann algebra equipped with a semifinite faithful normal trace $\tau$.
    Let $0\le x,y\in \left(L_1+L_\infty\right)\left(\mathcal{M},\tau\right)\cap S_0\left(\mathcal{M},\tau\right)$.
    Then, 
    \[\mu\left(x+y\right)=\mu\left(x\oplus y\right)\]
    if and only if $xy=0$.
\end{theorem}
\begin{proof}
    The implication $\Leftarrow$ is trivial.
    
    $\left(\Rightarrow\right). $ Let $z=x+y$.
    Applying Lemma \ref{lemma 2} to $z$, there exists a continuous increasing convex function $f:[0,\infty)\to [0,\infty)$ satisfying the conditions in Lemma \ref{convex} such that 
    \[\int_0^\infty f\left(\mu\left(t;x\right)\right)dt,~\int_0^\infty f\left(\mu\left(t;y\right)\right)dt\le \int_0^\infty f\left(\mu\left(t;z\right)\right)dt<\infty,\]
    that is,
    \[f\left(x\right),f\left(y\right),f\left(z\right)\in L_1\left(\mathcal{M},\tau\right).\]
    Since $\mu\left(x+y\right)=\mu\left(x\oplus y\right)$, it follows    that
    \begin{eqnarray*}
    \tau\left(f\left(x+y\right)\right)
    &\stackrel{\mbox{\tiny \cite[Prop. 3.2.8]{DPS}}}{=}&\int_0^\infty f\left(\mu\left(t;x+y\right)\right)\,dt  \\
    &=& \int_0^\infty f\left(\mu\left(t;x\oplus y\right)\right)\,dt  \stackrel{\mbox{\tiny \cite[Prop. 3.2.8]{DPS}}}{=}\tau\left(f\left(x\right)\right)+\tau\left(f\left(y\right)\right).
    \end{eqnarray*}
    It follows from Lemma \ref{convex} that $xy=0$.
    This completes the proof.
\end{proof}

The following corollary is an immediate consequence of Theorem  \ref{L1+M}.
\begin{corollary} 
 Let $\mathcal{M}$ be  a semifinite von Neumann algebra equipped with a semifinite faithful normal trace $\tau$.
    Assume that $1\le p\le \infty$.
    Let $0\le x,y\in \left(L_p+L_\infty\right)\left(\mathcal{M},\tau\right)\cap S_0\left(\mathcal{M},\tau\right)$.
    Then, 
    \[\mu\left(x+y\right)=\mu\left(x\oplus y\right)\]
    if and only if $xy=0$.
\end{corollary}

Below, we consider the orthogonality of (not necessarily positive) operators.

\begin{lemma}\label{lemma lu rx}
Let $\mathcal{M}$ be  a semifinite von Neumann algebra equipped with a semifinite faithful normal trace $\tau$.
Let $x\in S_0\left(\cM,\tau\right)$ and let $u$ be such that $\norm{u}_\infty \le 1$. 
If $\mu\left(xu\right) =\mu\left(x\right)$, then $r\left(x\right)=uu^*r\left(x\right)$.
\end{lemma}
\begin{proof}
Assume that $\left({\bf 1}-uu^*\right) r\left(x\right) \ne 0$. We have
$r\left(x\right) -r\left(x\right) uu^* r\left(x\right) \ne 0$. 
Since $x\in S_0\left(\cM,\tau\right)$, it follows 
that there exists $t>0$ such that
\begin{align*}
    \mu\left(t;|x|uu^* |x| \right)^{1/2}=\mu\left(t;|x|r\left(x\right)uu^*r\left(x\right)|x|\right)^{1/2}\stackrel{\tiny \mbox{\cite[Prop. 2.2]{CKS}}}{<}\mu\left(t;|x|r(x)|x|\right) ^{1/2}= \mu\left(t;x\right).
\end{align*}
However, 
\[
\mu\left(|x|uu^* |x| \right)^{1/2}=\mu\left(u^* |x|^2u\right)^{1/2}=
\mu\left(|xu |\right)=\mu\left(x\right), 
\]
which is a contradiction. 
Hence, $\left({\bf 1}-uu^*\right) r\left(x\right) =0$, that is, $r\left(x\right)=uu^*r\left(x\right)$.

\end{proof}

The following theorem is the second  main result of this chapter. 
\begin{theorem}  \label{prop disjoint}
Let $\mathcal{M}$ be  a semifinite von Neumann algebra equipped with a semifinite faithful normal trace $\tau$.
    Let $  x,  y \in L_1(\cM,\tau)$.
    Then,   $\mu\left(x+ y \right) =  \mu\left( x \oplus y \right)$ if and only if  $x\perp y$. 
\end{theorem}
\begin{proof} 
It suffices to prove the necessity. 



By
 $ \mu\left( x +  y \right) =\mu\left(x  \oplus   y \right) $, we have  
\[\tau\left( | x +   y|\right)  = \tau\left(|x|\right) +\tau\left(|y|\right) .  \]
Let $  x +    y= v| x +   y|$ be the polar decomposition.  
Then, we have 
\begin{align*}
\tau\left(|x|\right)+\tau\left(|y|\right) =\tau\left( | x +  y| \right)&  =  \tau\left(v^* x \right)+ \tau\left(v^*  y  \right)  \\
&\le 
| \tau\left(v^* x \right)|+ |\tau\left(v^*  y  \right)| \le  \tau\left(|x|\right) +\tau\left(|y|\right) . 
\end{align*}
This implies that $ \tau\left(v^* x \right)= \tau\left(|x|\right)$ and $\tau\left(v^*  y  \right)  =\tau\left(|y|\right) $. 
Clearly, $v^*x $ is  self-adjoint. 
Moreover, \[\tau\left(v^* x  \right) =\tau\left(\left(v^*x \right)_+ \right)  - \tau\left(\left(v^*x\right)_-\right)
\le \tau\left(\left(v^*x\right)_+ \right)  + \tau\left(\left(v^*x\right)_-\right)
=\tau\left(|v^* x| \right) \le \tau\left(|x|\right).\]
This implies that 
$v^*x$ is positive. 
Hence, 
\[\int_0^\infty \mu(t;v^* x) dt =\tau(v^*x )=\tau(|x|) =\int_0^\infty \mu(t;  x) dt  ,\]
 which together with $\mu(v^*x)\le \mu(x)$ yields that
$\mu(v^*x)=\mu(x)$. 
By Lemma \ref{lemma lu rx}, we have $r\left(v^*\right) \ge l\left(x\right)$ and $r\left(v^*\right) \ge l\left(y\right)$. 
Moreover, let $z_x:=v^*x$.
Then, we have \[ l\left(z_x\right)  \le l\left(v^*\right) =r\left(v\right)  . \]
Hence, $vl\left(z_x\right)$ is a partial isometry.
By $x =v l  \left(z_x\right) z_x$ and the uniqueness of polar decomposition \cite[Theorem 1.7.3]{DPS}, we obtain that $z_x=|x|$. 

The same argument yields that $y= v l \left(|y|\right) |y|$ is the polar decomposition of $y$.

Note that
\[ \mu\left(|x|\oplus |y|\right)=  \mu\left(x\oplus y\right) =  \mu\left(x+y \right)  =\mu\left(v^*   \left(x+y\right)\right) = \mu\left(|x|+|y|\right). \]
By Proposition \ref{positive disjoint}, $|x|\perp |y|$.

Observe that \[ x^* y = |x| l\left(|x|\right) v^*  v l \left(|y|\right) |y| = |x| l\left(|x|\right)  l\left(|y|\right) |y|=|x||y| =0  .\] 
We obtain that $x\perp y$. 
The proof is complete. 
\end{proof}

Theorem \ref{prop disjoint} is however not enough to  prove the disjointness-preserving property of a surjective isometry between weak $L_p$-spaces.
Indeed, for atomless von Neumann algebras, 
 weak $L_p$-spaces are  not  subsets   of $L_1$ for any $0< p\le1$.
On the other hand, even though $\cC_{p,\infty }\subset \cC_1$ when $0<p<1$, 
the sum of two mutually disjoint extreme points $x,y$ in the unit ball is no longer an extreme point of the ball
of any radius in $\cC_{p,\infty}$ (neither is $T\left(x+y\right)$). 
This shows that one does not have enough information on the singular value sequences of $x+y$ and $T\left(x+y\right)$.

The following result is a special case of \cite[Theorem 1.2]{FHH}. For the sake of completeness, we include a short proof below. 

\begin{prop} \label{0<p<2}
    Assume that $0<p<\infty,~ p\ne2$. 
    Let $\mathcal{M}$ be a semifinite von Neumann algebra equipped with a semifinite faithful normal trace $\tau$.
    Let $x,y\in L_{p,\infty}(\mathcal{M},\tau)$.
    Then, $\mu(x+\lambda y)=\mu(x\oplus y),~\lambda=-1,1$ if and only if
    $x\perp y$.
\end{prop}
\begin{proof}
    It suffices to prove the necessity.
    The proof will be decomposed in two parts.
    \begin{enumerate} [label=(\roman*)]
    \item\label{(i} Assume that $0<p<2$.
    Let $\phi(t)=\frac{t^2}{1+t^2},~t\ge 0$ and $\psi(t)=\phi(\sqrt t)=\frac{t}{1+t},~t\ge 0$.
    Observe that 
    \[\begin{aligned}
    \tau\left(\phi(|x|)\right)&=\int_0^\infty \phi\left(\mu\left(t;x\right)\right)dt\\
&\le\int_0^\infty \phi\left(ct^{-1/p}\right)dt\\
&=\int_0^\infty \frac{\left(ct^{-1/p}\right)^2}{1+\left(ct^{-1/p}\right)^2}dt\\
&=\int_0^\infty \frac{c^2}{c^2+t^{2/p}}dt<\infty
\end{aligned}\]
for some constant $c>0$ independent of $x$.
    Similarly, $\tau\left(\phi(|y|)\right)<\infty$.
    Since $\mu\left(x+\lambda y\right)=\mu\left(x\oplus y\right),~\lambda=-1,1$, it follows that
    \[\tau\left(\phi\left(|x+y|\right)\right)+\tau\left(\phi\left(|x-y|\right)\right)=2\tau\left(\phi\left(|x|\right)\right)+2\tau\left(\phi\left(|y|\right)\right)<\infty.\]
    Note that $\psi$ is strictly concave.
    Applying the noncommutative Lamperti's theorem (see \cite[Theorem 1.2]{FHH}) to $\psi$, we obtain that $x^*y=xy^*=0$.
    \item Assume that $2<p<\infty$.
    Let $\phi(t)=\left(\frac{t^4}{1+t^2}\right)^{\frac{p+2}{4}}, ~t\ge0$.
    The remaining statements follow from a similar argument as used in the proof of \ref{(i}. 
    \end{enumerate}
    
    The proof is complete.
\end{proof}

Although the above proposition applies to weak $L_p$-space  $L_{p,\infty}\left(\mathcal{M},\tau\right), ~0<p\le 1$, the unit ball of $L_{p,\infty}\left(\mathcal{M},\tau\right)$ has no extreme point when $\mathcal{M}$ is non-atomic (see Corollary \ref{0<p le1}).


\begin{remark}
    The condition in Theorem  \ref{L1+M} that $0\le x,y\in \left(L_{1}+L_\infty\right)\left(\mathcal{M},\tau\right)\cap S_0(\mathcal{M},\tau)$ cannot be weakened to assuming $0\le x,y\in S_0\left(\mathcal{M},\tau\right)$.
    Indeed, let $\mathcal{M}=L_\infty\left(0,\infty\right)$ be equipped with the usual integral trace $\tau$. 
    Assume that $x=y=\frac1t,~t\in \left(0,\infty\right)$. We have $x,y\in S_0(\cM,\tau)$.
    Note that $\mu\left(x+y\right)=\mu\left(x\oplus y\right)=\frac2t$.
    However, $x\not\perp y$.
\end{remark}

 The assertion of Theorem \ref{L1+M}   is not valid when $x,y$ are not necessarily positive, even in the
commutative case. That is,   equality $\mu(x+y)=\mu(x\oplus y)$ may  not imply orthogonality of $x$  and $y$. 
We would like to thank Zhizheng Yu for providing the following example.

\begin{remark}
   Let 
\[x=\left(1,-\frac{1}{2},\frac{1}{2},\frac{1}{2},\frac{1}{4},\frac{1}{4},\frac{1}{4},\frac{1}{4},\frac{1}{8},\frac{1}{8} ,\frac{1}{8},\frac{1}{8},\frac{1}{8},\frac{1}{8},\frac{1}{8},\frac{1}{8},\frac{1}{16},\dots\right),\]
\[y=\left(-\frac{1}{2},1,\frac{1}{2},\frac{1}{2},\frac{1}{4},\frac{1}{4},\frac{1}{4},\frac{1}{4},\frac{1}{8},\frac{1}{8},\frac{1}{8},\frac{1}{8},\frac{1}{8},\frac{1}{8},\frac{1}{8},\frac{1}{8},\frac{1}{16},\dots\right).\]
We have 
\[x+y = \left(\frac{1}{2},\frac{1}{2},1,1, \frac{1}{2},\frac{1}{2},\frac{1}{2},\frac{1}{2},\frac{1}{4},\frac{1}{4},\frac{1}{4},\frac{1}{4},\frac{1}{4},\frac{1}{4},\frac{1}{4},\frac{1}{4},\frac{1}{8},\dots\right),\]
whose  decreasing rearrangement is the same as that of $|x|\oplus |y|$.

\end{remark}

\backmatter

\end{document}